\documentclass[12pt]{article}%
\usepackage{amsmath}
\usepackage{amsfonts}
\usepackage{amssymb}
\usepackage{graphicx}
\usepackage{fullpage}
\usepackage{color}
\usepackage{hyperref}%
\providecommand{\U}[1]{\protect\rule{.1in}{.1in}}
\hypersetup{
pdfborder = {0 0 0},
urlbordercolor = {0 0 0},
colorlinks = {false},
}
\newtheorem{theorem}{Theorem}

\newtheorem{algorithm}[theorem]{Algorithm}

\newtheorem{corollary}[theorem]{Corollary}

\newtheorem{definition}[theorem]{Definition}
\newtheorem{example}[theorem]{Example}

\newtheorem{lemma}[theorem]{Lemma}

\newtheorem{proposition}[theorem]{Proposition}
\newtheorem{remark}[theorem]{Remark}

\newenvironment{proof}[1][Proof]{\noindent\textbf{#1.} }{\ \rule{0.5em}{0.5em}}
\begin{document}

\title{{Introduction to Infinite Dimensional Statistics}\\{and Applications}}
\author{Jan Mandel\\University of Colorado Denver}
\date{Version \today}
\maketitle
\tableofcontents

\newpage

\section{Introduction}

These notes started to educate ourselves and to collect some background for
our future work, with the hope that perhaps they will be useful to others
also. The selection of the material is motivated by efforts to simplify and
generalize our analysis of the ensemble Kalman
filter~\cite{Kasanicky-2017-EKF,Kasanicky-2017-WBD,Kwiatkowski-2015-CSR,Mandel-2011-CEK}%
. Earlier versions served for short courses at INP-ENSEEIHT and L'\'{E}cole
Nationale de la M\'{e}t\'{e}orologie, Toulouse, France, in December 2012,
2013, and 2014.

Many if not all results are more or less elementary or available in the
literature, but we need to fill some holes (which are undoubtely statements so
trivial that the authors we use do not consider them holes at all) or make
straightforward extensions, and then we do the proofs in sufficient detail for
reference. The proofs might be included in a shortened form in a future paper
on an application, so we may present separate proofs for a more general and a
less general case rather than strive for the shortest writing. For the same
reason, we may spell out the argument rather than rely on a reference to a
numbered equation earlier.

The basic Hilbert space framework is based on \cite[Ch. 1]{DaPrato-2006-IIA},
but then we deviate in several ways. Our focus is on the (relatively
elementary) algebra of random elements and mappings rather than deep
properties of the probability measures on infinitely dimensional spaces as
found in standard monographs, such as
\cite{Bogachev-1998-GM,Kuo-1975-GMB,Ledoux-1991-PBS}, and in fact much of the
contents of these notes is taken for granted in those advanced sources and not
even spelled out. For example, \cite{Ledoux-1991-PBS} notes that many
statements in finite dimension hold true in Hilbert spaces, but he does not
see the need prove the Hilbert space extensions. \cite{vanderVaart-2000-AS}
uses arguments which carry over to infinite dimension. Our consistent
preference for $L^{p}$ norms rather than moments is somewhat unusual in
statistics, even if the norms are commonly used in, for example
\cite{Chow-1988-PT,Ledoux-1991-PBS}. The use of $L^{p}$ norms allows us to use
the properties of the norm easily, and it results in a style very similar to
$L^{p}$ spaces in real analysis, which we consider an advantage. We are
interested in estimates with explicit constants and explicit rates of
convergence rather that almost sure convergence or convergence under the
weakest possible conditions.

Although many results extend in a straightforward way to Banach spaces (or at
least reflexive ones), we prefer the simpler expression afforded by the less
general Hilbert space setting. The prerequisites for reading these notes are
only some Lebesgue integral and measure-theoretic probability, and basic
concepts from functional analysis, on the level of introductory graduate
courses. We review some results that should be known from such basic courses,
but we omit their proofs. Instead, we prove only the more unusual statements,
in more detail.

\section{Motivation: Spatial stochastic models and data assimilation}

\label{sec:prologue}

Much of the recent development in \emph{data assimilation} and
\emph{quantification uncertainty} can be described as \emph{statistics of the
solutions of partial differential equations}, or \emph{statistics of
scientific computing}. Here is picture of an eclectic collection of books
I\ saw in Summer 2012 at the Parallel Algorithms Group at CERFACS, which says
it all:

\begin{center}
\includegraphics[height=3in]{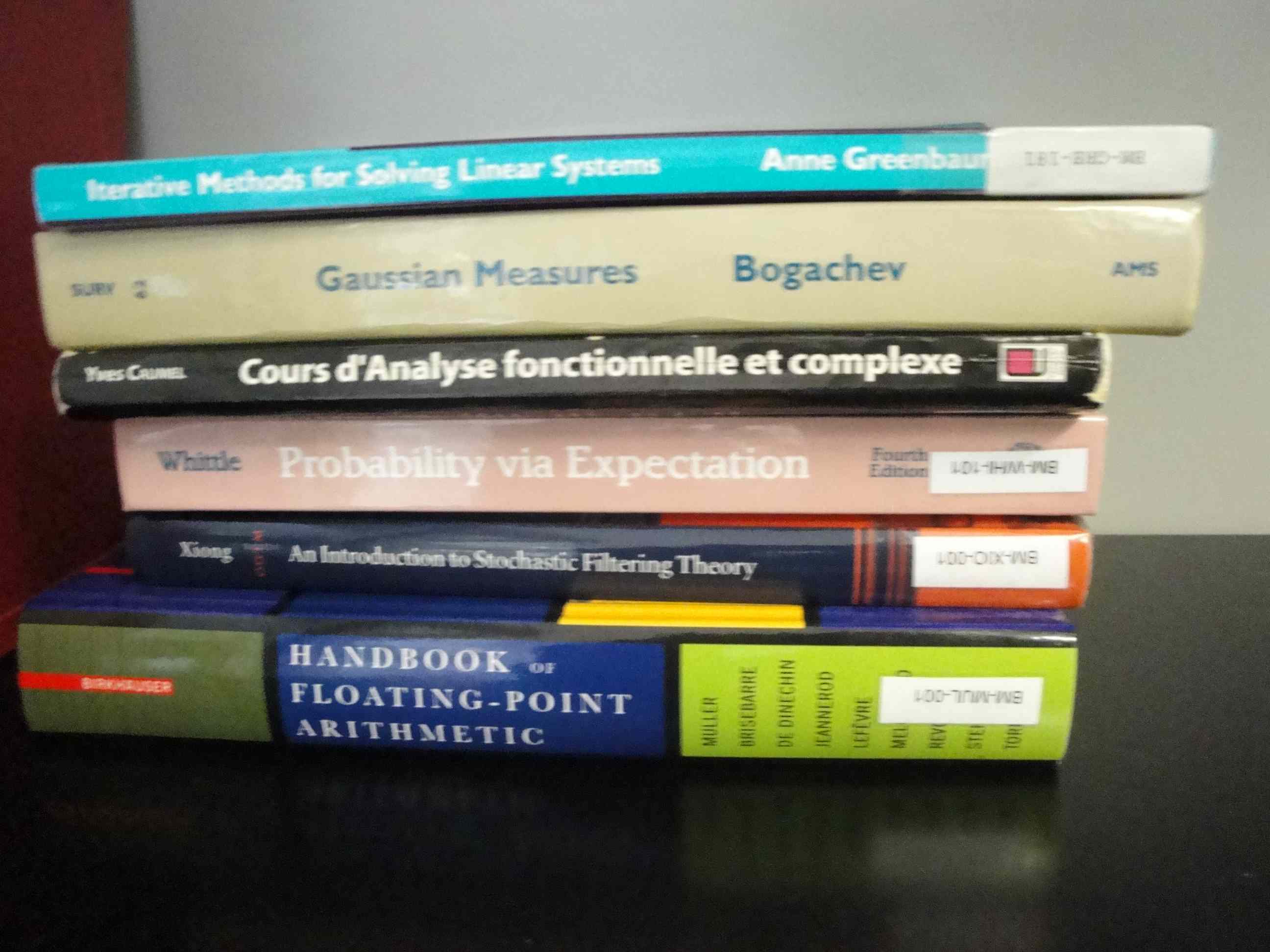}
\end{center}

\subsection{Continuum and discretization}

Physical models are formulated as partial differential equations, which have
solutions in spaces of functions, which are infinitely dimensional. The
computational realization of those models relies on discretizations, resulting
in models with finite dimensional state. Data assimilation methods rely on
multivariate statitics, applied to the discretized, finite dimensional models.
Finer resolution models result in models with higher dimension.

This higher dimension is generally just an artifact of the discretization.
Sometimes we have lower effective dimension: Nonlinear systems,
low-dimensional attractor. Curse of dimensionality: the performance of
stochastic algorithms sometimes deteriorates with the system dimension. But
sometimes not, the physical PDE model has so many scales that refining the
discretization keeps changing the solution. Then the effective dimension of
the model may be large and it has not been captured by the numerical model
yet... but then the adequacy of the numerical model is questionable in the
first place.

Standard approach in numerical PDEs: solution in infinitely dimensional
function space, finite-dimensional is approximation. Similarly here,
stochastic models in infinite dimension, fininite dimensional is an approximation.

Probability on infinitely-dimensional spaces can be tricky. There are many
statements that should be intuitively true but are not. Combines statistics
and functional analysis, so sometimes called functional statistics.

Asymptotics of interest:

\begin{itemize}
\item large number of samples

\item large number of time steps

\item large dimension from finer discretization
\end{itemize}

In following sections, we review some of the theoretical foundations and prove
an asymptotic result, convergence of the\ Ensemble Kalman Filter (EnKF)\ in
the limit for large number of samples (i.e., ensemble members).

\subsection{From least squares to Bayes theorem}

For symmetric positive definite matrix $A$, denote the vector norm $\left\vert
u\right\vert _{A}=\left(  u^{\mathrm{T}}Au\right)  ^{1/2}$. 

Data assimilation can be posed as solving approximately the inverse problem
\cite{Cotter-2009-BIP,Stuart-2010-IPB}%
\[
\mathcal{H}\left(  u\right)  =d
\]
where $u\in U$ is unknown system state, $\mathcal{H}$ is observation function,
and $y$ is the data. Solving by least squares gives%
\[
\left\vert \mathcal{H}\left(  u\right)  -d\right\vert _{R^{-1}}^{2}%
\rightarrow\min_{u\in U}.
\]
However $u$ can be not unique or be uncontrollably large, so we add a
Tikhonov-like regularization term to control the departure from some fixed
$u^{f}$, and get%
\[
\left\vert \mathcal{H}\left(  u\right)  -d\right\vert _{R^{-1}}^{2}+\left\vert
u-u^{f}\right\vert _{Q^{-1}}^{2}\rightarrow\min_{u\in U},
\]
which is equivalent to
\[
e^{-\frac{1}{2}\left\vert \mathcal{H}\left(  u\right)  -d\right\vert _{R^{-1}%
}^{2}}e^{-\frac{1}{2}\left\vert u-u^{f}\right\vert _{Q^{-1}}^{2}}%
\rightarrow\max_{u\in U},
\]
which has the interpretation of maximizing the \emph{analysis} probability
density from the Bayes theorem,%
\[
u^{a}=\arg\max_{u\in U}p^{a}\left(  u\right)  ,\quad p^{a}\left(  u\right)
\propto p\left(  d|u\right)  p^{f}\left(  u\right)  ,
\]
where $\propto$ means \textquotedblleft proportional to\textquotedblright,
\[
p^{f}\left(  u\right)  \propto e^{-\frac{1}{2}\left\vert u-u^{f}\right\vert
_{Q^{-1}}^{2}}%
\]
is the \emph{forecast} probability density, and
\begin{equation}
p(d|u)=e^{-\frac{1}{2}\left\vert \mathcal{H}\left(  u\right)  -d\right\vert
_{R^{-1}}^{2}} \label{eq:gaussian-data-likelihood}%
\end{equation}
the \emph{data likelihood}. The forecast is our best knowledge about the state
before the data, expressed as a probability density. Data likelihood describes
the data value as well as its error distribution. Here, That is,
$\mathcal{H}\left(  u\right)  $ is what the data would be if there were no
errors and the truth were $u$, and $R$ is the covariance of the normally
distributed data error. That is, given $u$, the data is assumed to be
distributed as $d\sim N\left(  \mathcal{H}\left(  u\right)  ,R\right)  $.

\subsection{Bayes theorem in infinite dimension}

All is good when the dimension of the state is finite. However, recall that by
definition, $p$ is a probability density of a measure $\mu$ on $\mathbb{R}%
^{n}$ defined by%
\begin{equation}
\mu\left(  A\right)  =\int_{A}p(u)d\nu\left(  u\right)  \label{eq:def-density}%
\end{equation}
for all Borel sets $A\subset U$, that is open, closed, or obtained by
countably many set operations from open or closed sets. Such sets are called
\emph{Borel sets}. The integration in (\ref{eq:def-density}) is with respect
to the Lebesgue measure $\nu$ (i.e. integration as we know it). But there is
no Lebesgue measure on an infinitely dimensional space (Theorem
\ref{thm:no-lebesgue}).

The key to applying the Bayes theorem in an infinitely dimensional space is to
realize that the primary object is the measure and not its density. So, write
the Bayes theorem for densities as%
\[
p^{a}(u)\propto p(d|u)p^{f}(u)
\]
and integrate over a Borel set $A\subset U$,
\[
\int_{A}p^{a}(u)du\propto\int_{A}p(d|u)p^{f}(u)du.
\]
Then, the analysis probability measure $\mu^{a}$ is given by%
\begin{equation}
\mu^{a}\left(  A\right)  =c\int_{A}p(d|u)d\mu^{f}(u)\quad\text{for all Borel
sets }A\subset U, \label{eq:bayes-infinitely-dimensional}%
\end{equation}
where the constant is determined from the condition that
\begin{equation}
\mu^{a}\left(  U\right)  =c\int_{U}p(d|u)d\mu^{f}(u)=1.
\label{eq:normalization}%
\end{equation}

The relation (\ref{eq:bayes-infinitely-dimensional}) between the measures
$\mu^{f}$ and $\mu^{a}$ is called the \emph{Radon-Nykodym derivative} and
denoted as%
\[
\frac{d\mu^{a}}{d\mu^{f}}\left(  u\right)  =cp(d|u).
\]
Unlike densities, the Radon-Nykodym derivative carries over the the general
case of infinitely dimensional space. But, how do we know that
\[
\int_{U}p(d|u)d\mu^{f}(u)>0
\]
in (\ref{eq:bayes-infinitely-dimensional})?

All is good when the data in (\ref{eq:gaussian-data-likelihood})\ is finite
dimensional, or, more generally, $C^{-1}$ is an operator defined everywhere
(but we will see that then the data likelihood cannot come from a probability
density, like in finite dimension). Then the data likelihood $p^{f}\left(
u\right)  \propto e^{-\frac{1}{2}\left\vert u-u^{f}\right\vert _{Q^{-1}}^{2}}$
is positive everywhere, and the integral of a nonegative function which is
positive on a set of positive measure is positive.

\subsection{Random field as a stochastic process}

The usual approach in geostatistics \cite{Cressie-1993-SSD} is the random
field as a stochastic process, that is a collection of random variables
$\left\{  U_{x}\right\}  $, one for every point $x$ in some domain. The mean
at $x$ and covariance between two points $x$ and $y$ are defined pointwise
\begin{align*}
\overline{U}_{x}  &  =E\left(  U_{x}\right) \\
C\left(  x,y\right)   &  =E\left(  \left(  U_{x}-\overline{U}_{x}\right)
\left(  U_{y}-\overline{U}_{y}\right)  \right)
\end{align*}

If the mean does not depend on $x$ and the covariance depends on the
difference $x-y$ only,%
\[
C_{U}\left(  x,y\right)  =f\left(  x-y\right)
\]
the random field $U$ is called stationary and $f$ is the covariance function
of $U$.

One common way of generating a random field is as a sum of a series of
functions with random coefficients,%
\begin{equation}
U_{x}=m\left(  x\right)  +\sum_{n=1}^{\infty}c_{n}^{1/2}\xi_{n}V_{n}\left(
x\right)  , \label{eq:random-series}%
\end{equation}
where $c_{n}\geq0$ are constants, $\xi_{n}$ are uncorrelated random variables
with $E\left(  \xi_{n}\right)  =0$, $E\left(  \xi_{n}^{2}\right)  =1$, and
$V_{n}$ are orthonormal functions on the spatial domain $S$, that is $%
{\displaystyle\int\limits_{S}}
v_{m}v_{n}dx=0$ if $m\neq n$ and $1$ if $m=n$. Then%
\begin{align}
\overline{U}_{x}  &  =m\left(  x\right) \nonumber\\
C\left(  x,y\right)   &  =E\left[  \left(  \sum_{m=1}^{\infty}c_{m}^{1/2}%
\xi_{m}v_{m}\left(  x\right)  \right)  \left(  \sum_{n=1}^{\infty}c_{n}%
^{1/2}\xi_{n}v_{n}\left(  y\right)  \right)  \right] \nonumber\\
&  =E\left[  \sum_{m=1}^{\infty}c_{m}\xi_{m}^{2}v_{m}\left(  x\right)
v_{m}\left(  y\right)  \right] \nonumber\\
&  =\sum_{m=1}^{\infty}c_{m}v_{m}\left(  x\right)  v_{m}\left(  y\right)
\label{eq:stoch-proc-cov-dec}%
\end{align}
The covariance operator is now a mapping that assigns to a function $w$ the
function $\boldsymbol{C}w$, defined by
\[
\left(  \boldsymbol{C}w\right)  \left(  x\right)  =%
{\displaystyle\int\limits_{S}}
C\left(  x,y\right)  w\left(  y\right)  dy.
\]
and $c_{n}$ are eigenvalues and $v_{n}$ the eigenvectors of the covariance
operator $\boldsymbol{C}$, that is,%
\[
\boldsymbol{C}v_{n}=c_{n}v_{n},
\]
because%

\begin{align*}
\left(  \boldsymbol{C}v_{n}\right)  \left(  x\right)   &  =%
{\displaystyle\int\limits_{S}}
\sum_{m=1}^{\infty}c_{m}v_{m}\left(  x\right)  v_{m}\left(  y\right)
v_{n}\left(  y\right)  dy\\
&  =\sum_{m=1}^{\infty}c_{m}v_{m}\left(  x\right)
{\displaystyle\int\limits_{S}}
v_{m}\left(  y\right)  v_{n}\left(  y\right)  dy=c_{n}v_{n}\left(  x\right)
\end{align*}
assuming that switching the integral and the infinite sum is justified. See,
e.g., \cite{Marcus-1981-RFS} studies of the convergence of random series. We
will see later that an \textbf{operator} $\boldsymbol{C}$ \textbf{is the
covariance of a probability measure if and only if it has finite trace},%
\begin{equation}
\operatorname*{Tr}\boldsymbol{C}=\sum_{n=1}^{\infty}c_{n}<\infty.
\label{eq:trace_condition}%
\end{equation}

When the trace condition (\ref{eq:trace_condition}) is not satisfied, random
series (\ref{eq:random-series}) not give a random element in the usual sense
or even converge. The extreme case is \textbf{white noise}, where all
$c_{n}=1$.

We find it more convenient to consider a realization of the random process
$U_{x}$ a random function (also called a path), and even better a random
element of some Hilbert space, which leads to considerable simplification. We
are then led naturally to the study of random elements and probability
measures on Hilbert spaces.

In computational applications, the model runs in a finite dimensional space,
such as values on a grid in the spatial domain $S$, and one can expect that
the statistical properties of the model then approach the infinite-dimensional
case as the grid gets refined. In some methods such as stochastic Galerkin
\cite{Babuska-2004-GFE}, rigorous analysis of such approximations is
available. Engineering literature, however, typically relies on heuristic
arguments and computational experiments, e.g. \cite{Xiu-2010-NMS}.

\subsection{Functions of the Laplace operator, random sine series, and the
heat kernel}

\label{sec:laplace}Can one use a function of the Laplace operator%
\[
-\Delta=-\frac{\partial^{2}}{\partial x^{2}}-\frac{\partial^{2}}{\partial
y^{2}}%
\]
as on the rectangle as a covariance operator? Consider the Laplace operator on
a rectangle%
\[
S=\left(  0,a\right)  \times\left(  0,b\right)  ,
\]
with boundary conditions $u=0$ on $\partial S$. The eigenvectors and
eigenvalues of $-\Delta$,%
\[
-\Delta u_{k\ell}=\lambda_{k\ell}u_{k\ell}%
\]
are%
\[
u_{k\ell}\left(  x,y\right)  =\sin\left(  \frac{x}{a}k\pi\right)  \sin\left(
\frac{y}{b}\ell\pi\right)
\]%
\[
\lambda_{k\ell}=\left(  \frac{k\pi}{a}\right)  ^{2}+\left(  \frac{\ell\pi}%
{b}\right)  ^{2}.
\]

Using the \textbf{negative powers of the Laplacian} $\left(  -\Delta\right)
^{^{-\alpha}}$ as the covariance gives the random sum%
\begin{equation}
U\left(  x,y\right)  =\sum_{k=1}^{\infty}\sum_{\ell=1}^{\infty}\lambda_{k\ell
}^{-\alpha/2}\xi_{k\ell}\sin\left(  \frac{x}{a}k\pi\right)  \sin\left(
\frac{y}{b}\ell\pi\right)  \label{eq:powers_laplacian}%
\end{equation}
where $\xi_{k\ell}\sim N\left(  0,1\right)  $ are independent. The numbers
$\left(  \lambda_{k\ell}\right)  ^{-\alpha}$ are the eigenvalues of the
operator $\left(  -\Delta\right)  ^{^{-\alpha}}$. Now consider when $U$ is
square integrable almost surely:%
\begin{align*}
E\left[  \int_{0}^{a}\int_{0}^{b}\left\vert U\left(  x,y\right)  \right\vert
^{2}dydx\right]   &  =\sum_{k=1}^{\infty}\sum_{\ell=1}^{\infty}\lambda_{k\ell
}^{-\alpha}E\left[  \left\vert \xi_{k\ell}\right\vert ^{2}\right]  \int%
_{0}^{a}\int_{0}^{b}\sin^{2}\left(  \frac{x}{a}k\pi\right)  \sin^{2}\left(
\frac{y}{b}\ell\pi\right)  dydx\\
&  =\operatorname*{const}\sum_{k=1}^{\infty}\sum_{\ell=1}^{\infty}%
\lambda_{k\ell}^{-\alpha}<\infty.
\end{align*}
Substituting and by summation over diagonals $k+\ell=n$, we see that up to a
constant multiplication factor,%
\begin{align*}
\sum_{k=1}^{\infty}\sum_{\ell=1}^{\infty}\lambda_{k\ell}^{-\alpha}  &
=\operatorname*{const}\sum_{k=1}^{\infty}\sum_{\ell=1}^{\infty}\left(
\frac{1}{k^{2}+\ell^{2}}\right)  ^{\alpha}\\
&  \approx\operatorname*{const}\sum_{k=1}^{\infty}\sum_{\ell=1}^{\infty
}\left(  \frac{1}{k+\ell}\right)  ^{2\alpha}\\
&  =\sum_{n=1}^{\infty}\sum_{\ell=1}^{n}\left(  \frac{1}{n}\right)  ^{2\alpha
}\\
&  =\sum_{n=1}^{\infty}n\left(  \frac{1}{n}\right)  ^{2\alpha}=\sum
_{n=1}^{\infty}n^{1-2\alpha}<\infty
\end{align*}
if and only if $1-2\alpha<-1$, that is $\alpha>1$. Hence, we have:

\begin{theorem}
$E\left[  \int_{0}^{a}\int_{0}^{b}\left\vert U\left(  x,y\right)  \right\vert
^{2}dydx\right]  <\infty$ a.s. if and only if $\alpha>1$. In particular, if
$\alpha>1$, then the series (\ref{eq:powers_laplacian}) converges to a sum
$U\in L^{2}\left(  S\right)  $ a.s.
\end{theorem}

In a similar way, we show that higher powers $\alpha$ make the random
functions smooth. Consider the partial derivatives of $U$ of order $s=p+q$,%
\begin{align*}
E\int_{0}^{a}\int_{0}^{b}\left\vert \frac{\partial^{s}U}{\partial x^{p}y^{q}%
}\right\vert ^{2}dydx  &  =\operatorname*{const}\sum_{k=1}^{\infty}\sum
_{\ell=1}^{\infty}\frac{k^{2p}\ell^{2q}}{\lambda_{k\ell}^{\alpha}}\\
&  =\operatorname*{const}\sum_{k=1}^{\infty}\sum_{\ell=1}^{\infty}\frac
{k^{2p}\ell^{2q}}{\left(  k^{2}+\ell^{2}\right)  ^{\alpha}}\\
&  \approx\operatorname*{const}\sum_{k=1}^{\infty}\sum_{\ell=1}^{\infty}%
\frac{\left(  k+\ell\right)  ^{2s}}{\left(  k+\ell\right)  ^{2\alpha}}\\
&  =\operatorname*{const}\sum_{n=1}^{\infty}n^{2s-2\alpha+1}<\infty
\end{align*}
if and only if $2s-2\alpha+1<-1$, that is $\alpha>1+s$. Functions with all
partial derivatives up to order $s$ square integrable are said to be in the
Sobolev space $H^{s}\left(  S\right)  $, hence we have%
\[
U\in H^{s}\left(  S\right)  \text{ if }\alpha>1+s\text{, }s\geq0.
\]
Note that this result holds even if $s$ is not integer, and in fact one can
define Sobolev spaces with non-integer exponent precisely like this, by
differentiating Fourier series. By the \emph{Sobolev embedding theorem}, in
two dimensions, functions in $H^{s}$ are continuous if $s>1$. Consequently,
$U$ given by (\ref{eq:powers_laplacian}) is almost surely continuous if
$\alpha>2$, differentiable if $\alpha>3$, etc.

The \textbf{heat kernel}
\[
K_{T}:u_{0}\mapsto u_{T}%
\]
is defined by the solution $u_{T}=u\left(  T,x\right)  $ of the heat equation%
\begin{equation}
\frac{\partial u}{\partial t}=\Delta u\text{ in }S \label{eq:heat}%
\end{equation}
with initial condition $u\left(  0,x\right)  =u_{0}$ and $u\left(  t,x\right)
=0$ on $\partial S$. Just like in elementary differential equations, this
equation has the solution%
\[
u\left(  t\right)  =e^{-t\Delta}u\left(  0\right)  ,
\]
where the exponential of the operator is defined by subtituting in the power
series,%
\[
e^{-t\Delta}=I-t\Delta+\frac{1}{2!}t^{2}\Delta^{2}-\frac{1}{3!}t^{3}\Delta
^{3}+\cdots
\]
The eigenvectors of the function of an operator remain the same and the
eigenvalues transform by substituting them into the function according to the
spectral mapping theorem. For the example when $S=\left(  0,a\right)
\times\left(  0,b\right)  $ as above, we have eigenvalues of the heat kernel
$K_{T}$,
\[
e^{-T\lambda_{k\ell}}=e^{-T\left[  \left(  \frac{k\pi}{a}\right)  ^{2}+\left(
\frac{\ell\pi}{b}\right)  ^{2}\right]  }%
\]
which go to zero faster than any $\left(  \lambda_{k\ell}\right)  ^{-\alpha}$.
Consequently, random function with heat kernel covariance has continuous
partial derivatives of all orders a.s. The use of the Green's function of the
Laplace equation (that is, $\left(  -\Delta^{-1}\right)  $ for covariance was
suggested in \cite{Kitanidis-1999-GCF}, and several functions of the Laplace
operator including the heat kernel were proposed in
\cite{Mirouze-2010-RCF,Weaver-2003-CBC}.

Finite dimensional versions of the operations with the random field and its
covariance are easily and cheaply implemented by FFT. Let
\begin{equation}
h_{x}=\frac{a}{m+1},\text{ }h_{y}=\frac{b}{n+1}. \label{eq:def-step}%
\end{equation}
Discretizing the operator $-\Delta$ by finite differences on the the uniform
mesh with the nodes%
\[
\left(  ih_{x},jh_{y}\right)  ,\quad i=1,\ldots,m,\text{ }j=1,\ldots,n.
\]
we obtain the linear operator $L_{mn}:\mathbb{R}^{m\times n}\rightarrow
\mathbb{R}^{m\times n}$, given by%
\[
L_{mn}:u\mapsto v,\quad v_{ij}=\frac{-u_{i+1,j}+2u_{ij}-u_{i-1,j}}{h_{x}^{2}%
}+\frac{-u_{i,j-1}+2u_{ij}-u_{i,j-1}}{h_{y}^{2}},
\]
where the values off the grid are taken as $u_{0j}=u_{m+1,j}=u_{i,0}%
=u_{i,n+1}=0$. The eigenvectors of $L_{mn}$ are $u_{k\ell}$, given by%
\begin{equation}
\left[  u_{k\ell}\right]  _{ij}=\sin\left(  \frac{i}{m+1}k\pi\right)
\sin\left(  \frac{j}{n+1}\ell\pi\right)  ,\quad k=1,\ldots,m,\text{ }%
\ell=1,\ldots,n, \label{eq:sin-eigvec}%
\end{equation}
and using the trigonometric identity%
\[
-\sin(A-B)+2\sin A-\sin(A+B)=2\sin A(1-\cos B)=4\sin A\sin^{2}\frac{B}{2}%
\]
we have the corresponding eigenvalues $\lambda_{k\ell}$, given by
\begin{equation}
\lambda_{mn,k\ell}=4\left(  \frac{\sin\frac{\pi}{2\left(  m+1\right)  }%
k}{\frac{a}{m+1}}\right)  ^{2}+4\left(  \frac{\sin\frac{\pi}{2\left(
n+1\right)  }\ell}{\frac{b}{n+1}}\right)  ^{2},\quad k=1,\ldots,m,\text{ }%
\ell=1,\ldots,n. \label{eq:sin-eigval}%
\end{equation}

The eigenvalues of the discretized operator approach the exact eigenvalues%
\[
\lim_{m,n\rightarrow\infty}\lambda_{mn,k\ell}=\left(  \frac{k\pi}{a}\right)
^{2}+\left(  \frac{\ell\pi}{b}\right)  ^{2},
\]
for fixed eigenvalue number $k,\ell$. The decomposition of any vector in the
basis of eigenvectors is computed by the Fast Fourier Transform (FFT).

[Matlab demo comes here]

\subsection{Perils of probability in infinite dimension}

Definitions of measure and density need to be taken seriously now - intuitive
approach is no longer good enough. Sigma-additive
\textbf{translation-invariant or rotation invariant measure finite on balls
does not exist in infinite dimension.}

No Lebesgue measure, no density in the usual sense. What happens to Bayes
theorem? Density-free formulation by integrals (Radon Nikodym theorem). No
need for reference measure. \textbf{Data likelihood need not be density}. In
fact, we have seen that it better should not be.

\begin{itemize}
\item What exactly it means a random variable, random element? The answer will
depend on the topology.

\item Can we do everything without probability densities?

\item What happens to the Bayes theorem and data assimilation?

\item How do you even integrate functions with values in infinite dimensional
spaces? The mean of a random variable is integral.

\item What are the requirements on the coefficients of a series to guarantee
convergence a.s., convergence to a continuous function a.s.?

\item Elements of a Hilbert space need not be functions, so there need not be
any concept of values at points. Then, what exactly is covariance and how does
it relate to what we know for finite dimension, and for a stochastic process?

\item How do common statistics formulas and inequalities carry over to
infinite dimension?

\item Many standard arguments in statistics are usually done term-by-term. Now
what in infinite dimension?

\item How do the laws of large numbers carry over?

\item What is white noise (a random vector with identity as the covariance) in
an infinitely dimensional space?
\end{itemize}

\subsection{Kalman filter}

Data assimilation, also known as sequential statistical estimation, or cycle
between the application of the Bayes theorem, called analysis stem, and
densities gaussian, represented by mean and covariance matrix, advances
covariance by linear model. Needs covariance matrix to evaluate the Bayes theorem.

Extended Kalman filter and variants: approximation of the covariance when the
model is not linear.

Cannot maintain covariance matrix when the system dimension is large.

Ensemble Kalman filter EnKF:\ replaces the covariance by sample covariance
computed from an ensemble of simulations. Still assumes that the distributions
are gaussian, but formulas do not depend on it, so usually used anyway.

Particle filter:\ represents the densities by a weighted ensemble, ensemble
members now called particles. Application of the Bayes theorem updates the
weights. Needs many particles to cover the locations where the densities are large.

Curse of dimensionality: when the dimension of the system state grows,
exponentially more particles are needed \cite{Bengtsson-2008-CRC}. But not
always - why? Depends on state distribution. When the effective dimension is
low, convergence is good. We will show that this happens always when the state
probability distribution is actually a probability distribution on the state
space. This will depend on the behavior of the eigenvalues of the covariance -
they need to drop off fast enough.. The assumptions in
\cite{Bengtsson-2008-CRC} that the eigenvalues are bounded away from $0$ make
the state distribution similar to white noise. Contrary to the lore, this
applies to both the EnKF and the particle filter
{\small \cite{Beezley-2009-HDA}}.

\begin{center}
\textbf{Curse of dimensionality? Not for probability measures!}
\end{center}

{%
\begin{tabular}
[c]{cc}%
\includegraphics[height=2.5in]{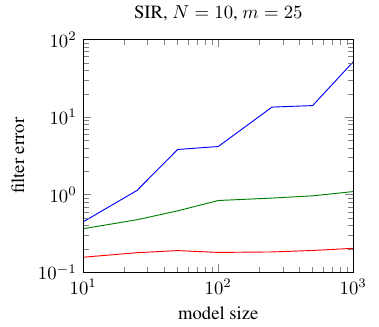} &
\includegraphics[height=2.5in]{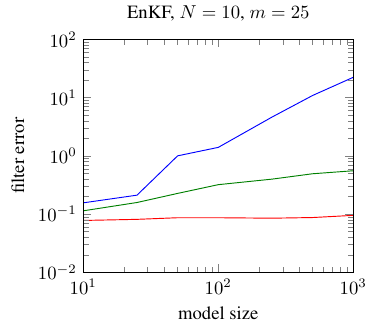}
\end{tabular}
}

{\small {\color{blue}\textbf{Constant covariance eigenvalues}} $\lambda_{n}=1$
and the {\color{green}inverse law} $\lambda_{n}=1/n$ are not probability
measures in the limit because $\sum_{n=1}^{\infty}\lambda_{n}=\infty$.
\newline{\color{red}\textbf{Inverse square law}} $\lambda_{n}=1/n^{2}$ gives a
probability measure because $\sum_{n=1}^{\infty}\lambda_{n}<\infty$. \newline
m=25 uniformly sampled data points from 1D state, N=10 ensemble members. From
\cite[Fig. 4.7]{Beezley-2009-HDA}.}

\section{Hilbert spaces}

We first review some material on Hilbert spaces to set the stage.

\subsection{A quick review of basic concepts}

The material in this section can be found in any textbook of introductory
functional analysis and many others, for example
\cite{Kreyszig-1989-IFA,Lax-2002-FA}.

A complete inner product space is called a \emph{Hilbert space}. A space is
separable if it has a countable dense subset. We will work with separable
spaces only. Let $H$ be a real, separable Hilbert space equipped with the
inner product $\left\langle \cdot,\cdot\right\rangle $ and the norm
$\left\vert \cdot\right\vert $. We use single bars because double bars
$\left\Vert \cdot\right\Vert $ will be reserved for stochastic norms. If
$x,y\in H$ and $\left\langle x,y\right\rangle =0$, then $x$ and $y$ are
orthogonal, which is denoted as $x\perp y$.

The symbol $\left\vert \cdot\right\vert $ will also denote the operator norm
on the space of continous (or, equivalently, bounded) linear operators
$\left[  H,K\right]  $, where $H$ and $K$ are Hilbert spaces, and $\left[
H\right]  $ stands for $\left[  H,H\right]  $. The space $\left[  H,K\right]
$ is equipped with the norm%
\[
\left\vert A\right\vert _{\left[  H,K\right]  }=\sup_{u\in H,u\neq0}%
\frac{\left\vert Au\right\vert _{K}}{\left\vert u\right\vert _{H}}%
\]
which makes it a Banach space (a complete normed linear space), but the norm
is not induced by an inner product, so $\left[  H,K\right]  $ is not a Hilbert
space. The Euclidean space $\mathbb{R}^{n}$ is a Hilbert space with the inner
product $\left\langle a,b\right\rangle =a^{\mathrm{T}}b$. We will use
subscripts to designate the space, such as in $\left\langle \cdot
,\cdot\right\rangle _{H}$, only if there is a danger of confusion.

The space of all linear functionals on $H$ is called the \emph{algebraic dual}
of $H$, and denoted by $H^{\#}$. The space of all continuous (or,
equivalently, bounded) linear functionals on $H$ is the \textit{dual space} of
$H$, and denoted by $H^{\prime}$. If $v\in H$, then the mapping%
\[
f:u\mapsto\left\langle u,v\right\rangle
\]
is clearly a bounded linear functional, and its norm is%
\[
\left\vert f\right\vert _{H^{\prime}}=\sup_{u\in H,u\neq0}\frac{\left\vert
\left\langle u,v\right\rangle \right\vert }{\left\vert u\right\vert
}=\left\vert v\right\vert .
\]
The \emph{Riesz representation theorem} states that every bounded linear
functional $f$ on $H$ is of this form, $f:u\mapsto\left\langle
u,v\right\rangle $ for some $v\in H$. This allows to consider the spaces $H$
and $H^{\prime}$ to be identified: for a linear functional $F$ on $H$, we use
the notation%
\begin{equation}
F\left(  u\right)  =\left\langle u,F\right\rangle , \label{eq:duality-pairing}%
\end{equation}
called \emph{duality pairing}. If $F$ is bounded, $F$ is considered to
coincide with an element of $H$ and $\left\langle \cdot,\cdot\right\rangle $
with the inner product. However, we will use the notation
(\ref{eq:duality-pairing}) for unbounded functionals also. Then $H=H^{\prime
}\subset H^{\#}$ and the duality pairing $\left\langle \cdot,\cdot
\right\rangle $ extends the inner product to $H\times H^{\#}$\emph{.}

One important consequence of the Riesz representation theorem is
representation of bounded bilinear forms by bounded operators, which we will
need to define covariance operators. For simplicity (and the application we
have in mind), consider the case of real space $H$ only.

\begin{lemma}
\label{lem:exist-op}If $a\left(  u,v\right)  $ is a bilinear form $H\times
H\rightarrow\mathbb{R}$ such that%
\[
\forall u,v\in H:\left\vert a\left(  u,v\right)  \right\vert \leq M\left\vert
u\right\vert \left\vert v\right\vert ,
\]
then there exists unique linear operator on $H$ such that%
\[
\forall u,v\in H:\left\langle u,Av\right\rangle =a\left(  u,v\right)  ,
\]
and $A\in\left[  H\right]  $ with $\left\vert A\right\vert \leq M$.
\end{lemma}

\begin{proof}
The map $u\mapsto a\left(  u,v\right)  $ for a fixed $v$ is a bounded linear
functional, so it can be written as $\left\langle u,Av\right\rangle $ for some
$Av.$
\end{proof}

Similarly, for any $A\in\left[  H,K\right]  $, where $H$ and $K$ are Hilbert
spaces, the Riesz representation theorem implies the existence of the adjoint
operators, which is defined uniquely by
\[
\left\langle u,Av\right\rangle _{K}=\left\langle A^{\ast}u,v\right\rangle
_{H},\quad\forall u\in K,v\in H.
\]
An operator $A\in\left[  H\right]  $ is called self-adjoint if $A^{\ast}=A$,
positive definite if $\left\langle Au,u\right\rangle >0$ for all $u\in H$,
$u\neq0$, and positive semidefinite if $\left\langle Au,u\right\rangle \geq0$
for all $u\in H$.

A\ \emph{complete orthonormal system} $\left\{  e_{n}\right\}  \subset H$ is a
sequence of vectors $e_{n}$ such that
\[
\left\langle e_{m},e_{n}\right\rangle =0\text{ if }m\neq n\text{,
}\left\langle e_{n},e_{n}\right\rangle =1,
\]
and the \emph{span of }$\left\{  e_{n}\right\}  $\emph{ is dense} in $H$. A
separable Hilbert space has a complete orthonormal system (in fact, infinitely
many). Given a complete orthonormal system $\left\{  e_{n}\right\}  $, any
vector $u\in H$ can be decomposed uniquely into an \emph{abstract Fourier
series}%
\begin{equation}
u=\sum_{n=1}^{\infty}c_{n}e_{n},\quad c_{n}=\left\langle u,e_{n}\right\rangle
, \label{eq:abstract-fourier}%
\end{equation}
and the \emph{Parseval equality} holds,%
\begin{equation}
\sum_{n=1}^{\infty}\left\vert c_{n}\right\vert ^{2}=\left\vert u\right\vert
^{2}. \label{eq:abstract-parseval}%
\end{equation}

In many physical systems, $\left\vert u\right\vert ^{2}$ corresponds to a
total energy of the system, and $\left\vert c_{n}\right\vert ^{2}$ is the
energy of mode $n$ (such as a vibration or turbulence mode).

The space $\left[  H\right]  $ of all bounded linear operators on an infinite
dimensional space is a Banach space and it is not separable: the diagonal
operator%
\[
\sum_{n=1}^{\infty}c_{n}e_{n}\mapsto\sum_{n=1}^{\infty}d_{n}c_{n}e_{n},\quad
d_{n}=0\text{ or }1
\]
is in $\left[  H\right]  $, but the distance between any such two operators in
$\left[  H\right]  $ is at least $1$, and there are uncountably many sequences
of zeros and ones. In fact, the space $\left[  H\right]  $ is so general that
in any Banach space can be embedded into the space $\left[  H\right]  $ for
some Hilbert space $H$. The space $\left[  H\right]  $ will not be very useful
because no special properties can hold for $\left[  H\right]  $ which would
not be true for an arbitrary Banach space. Later, we will look at some
subspaces of $\left[  H\right]  $ with some useful properties.

\subsection{Tensor product}

When $x$ and $y$ are vectors in $\mathbb{R}^{n}$, their product
$xy^{\mathrm{T}}$ is a rank-one matrix. Note that for $z\in\mathbb{R}^{n}$,%

\begin{equation}
\left(  xy^{\mathrm{T}}\right)  z=x\left(  y^{\mathrm{T}}z\right)
=x\left\langle z,y\right\rangle \label{eq:rank-one}%
\end{equation}
Therefore, the replacement for a rank-one matrix in Hilbert space is the
tensor product $x\otimes y$, which is a linear operator defined the same way
as (\ref{eq:rank-one}),%
\[
x\otimes y\in\left[  H\right]  ,\quad\left(  x\otimes y\right)
z=x\left\langle z,y\right\rangle \quad\forall z\in H
\]
or, equivalently, through the Riesz representation theorem,%
\begin{equation}
x\otimes y\in\left[  H\right]  ,\quad\left\langle u,\left(  x\otimes y\right)
v\right\rangle =\left\langle u,x\right\rangle \left\langle y,v\right\rangle
\quad\forall u,v\in H. \label{eq:def-tensor}%
\end{equation}
See, for example, \cite[page 25]{DaPrato-1992-SEI}. Then,%
\begin{equation}
\left\vert x\otimes y\right\vert =\left\vert x\right\vert \left\vert
y\right\vert \label{eq:tensor-norm}%
\end{equation}
This is clear when $y=0$. For $y\neq0$,
\[
\frac{\left\vert \left(  x\otimes y\right)  z\right\vert }{\left\vert
z\right\vert }=\frac{\left\vert x\left\langle y,z\right\rangle \right\vert
}{\left\vert z\right\vert }=\frac{\left\vert x\right\vert \left\vert
\left\langle y,z\right\rangle \right\vert }{\left\vert z\right\vert }\leq
\frac{\left\vert x\right\vert \left\vert y\right\vert \left\vert z\right\vert
}{\left\vert z\right\vert }=\left\vert x\right\vert \left\vert y\right\vert
\]
with equality when $z=y$. Continuity of the tensor product follows,
\begin{align}
\left\vert x\otimes y-w\otimes z\right\vert  &  =\left\vert \left(
x-w\right)  \otimes y+w\otimes\left(  y-z\right)  \right\vert \nonumber\\
&  \leq\left\vert z-w\right\vert \left\vert y\right\vert +\left\vert
y-z\right\vert \left\vert w\right\vert . \label{eq:tensor-continuous}%
\end{align}

If $\left\vert x\right\vert =1$, $x\otimes x$ is the orthogonal projection on
the span of $x$:%
\[
\left(  x\otimes x\right)  z=x\left\langle x,z\right\rangle =0\Leftrightarrow
z\perp x,\quad\left(  x\otimes x\right)  x=x\left\langle x,x\right\rangle =x.
\]

\subsection{Compact linear operators}

A linenar operator $A:H\rightarrow K$ is compact if for every bounded sequence
$\left\{  u_{n}\right\}  \subset H$, $\left\{  Au_{n}\right\}  $ has a
convergent subsequence in $K$. A compact operator is bounded. Linear operator
$A\in\left[  H,K\right]  $ is compact if and only if there is a sequence of
operators $A_{n}$ with finite-dimensional range which converge to $A$ in the
operator norm, $\lim_{n\rightarrow0}\left\vert A-A_{n}\right\vert =0$.

The spectrum $\sigma\left(  A\right)  $ of a linear operator $A$ is the set of
all $\lambda\in\mathbb{C}$ such that $\left(  \lambda I-A\right)  ^{-1}$ does
not exist as an operator in $\left[  H\right]  $. This concept of spectrum
generalizes eigenvalues in finite dimension. The spectrum is a closed subset
of $\mathbb{C}$. The spectrum of a linear operator can be quite general,
however the spectrum of a compact operator in $\left[  H\right]  $ consists of
isolated points, which are eigenvalues of finite multiplicity, and, unless $H$
is finite dimensional, also zero, which can be the only accumulation point of
the spectrum. The \emph{spectral theorem for selfadjoint compact operators}
guarantees that if $A=A^{\ast}$ is compact, then there is a complete
orthonormal system $\left\{  u_{n}\right\}  $ consisting of eigenvectors of
$A$,%
\[
Au_{n}=\lambda_{n}u_{n},\quad\left\langle u_{m},u_{n}\right\rangle =0\text{ if
}m\neq n\text{, }\left\vert u_{n}\right\vert =1,
\]
the eigenvalues $\lambda_{n}$ of $A$ are real, and $A$ is the infinite sum%
\begin{equation}
A=\sum_{n=1}^{\infty}\lambda_{n}u_{n}\otimes u_{n},
\label{eq:spectral-decomposition}%
\end{equation}
which converges in the operator norm. The tensor products $u_{n}\otimes u_{n}$
are called spectral projections.

The \emph{spectral decomposition} (\ref{eq:spectral-decomposition}) allows an
easy definition of functions of a compact operator: if a function $f$ is
continuous on the spectrum $\sigma\left(  A\right)  \subset\mathbb{R}$ of a
compact self-adjoint operator $A$, then%
\begin{equation}
f\left(  A\right)  =\sum_{n=1}^{\infty}f\left(  \lambda_{n}\right)
u_{n}\otimes u_{n}, \label{eq:function-of-operator}%
\end{equation}
and the action of $f\left(  A\right)  $ on a vector
\[
v=\sum_{n=1}^{\infty}c_{n}u_{n}%
\]
is obtained by multiplying the Fourier coefficients $c_{n}$ by $f\left(
\lambda_{n}\right)  $,%
\begin{equation}
f\left(  A\right)  v=\sum_{n=1}^{\infty}f\left(  \lambda_{n}\right)
u_{n}\left\langle v,u_{n}\right\rangle =\sum_{n=1}^{\infty}f\left(
\lambda_{n}\right)  c_{n}u_{n}. \label{eq:function-action}%
\end{equation}
The definition of $f\left(  A\right)  $ is consistent with the definition of
positive integer powers of $A$. More generally, when $f$ can be expanded into
a power series convergent on an open circle in $\mathbb{C}$ containing
$\sigma\left(  A\right)  $, $f\left(  A\right)  $ is obtained also by simply
substituting the operator $A$ in the power series. According to the
\emph{spectral mapping theorem}, the spectrum of $f\left(  A\right)  $ is
\[
\sigma\left(  f\left(  A\right)  \right)  =f\left(  \sigma\left(  A\right)
\right)  =\left\{  f\left(  \lambda\right)  |\lambda\in\sigma\left(  A\right)
\right\}  .
\]
This is again consistent with the behavior of eigenvalues of powers of matrices.

The definition (\ref{eq:function-of-operator}) also applies to negative powers
$A^{-\alpha}$, $\alpha>0$, of self-adjoint positive definite compact
operators, but the result is an \emph{unbounded operator defined only on the
dense subspace }$A^{\alpha}\left(  H\right)  $ \emph{of} $H$. Note that $0$
cannot be an eigenvalue of $A$, because $A$ is assumed to be positive
definite, so all $\lambda_{n}>0$. Then we define $A^{-\alpha}v$ only for
\[
v\in A^{\alpha}\left(  H\right)  =\left\{  A^{\alpha}w|w\in H\right\}
=\left\{  \left.  \sum_{n=1}^{\infty}c_{n}u_{n}\right\vert c_{n}=\lambda
_{n}^{\alpha}d_{n},\sum_{n=1}^{\infty}\left\vert d_{n}\right\vert ^{2}%
<\infty\right\}
\]
by%
\[
A^{-\alpha}v=\sum_{n=1}^{\infty}\lambda_{n}^{-\alpha}c_{n}u_{n},\quad
v=\sum_{n=1}^{\infty}c_{n}u_{n}.
\]

It is easy to see that $A^{-\alpha}$ is not a bounded operator because
\[
\frac{\left\vert A^{-\alpha}u_{n}\right\vert }{\left\vert u_{n}\right\vert
}=\lambda_{n}^{-\alpha}\rightarrow\infty\text{ as }n\rightarrow\infty
\]
because $0$ is the only accumulation point of the eigenvalues $\lambda_{n}>0$.
The space $A^{\alpha}\left(  H\right)  $ for $\alpha>0$ is a dense subspace of
$H$, because it contains all linear combinations (which are defined as having
finitely many terms) of the complete orthonormal sets of eigenvectors
$\left\{  u_{n}\right\}  $ of $A$.

If $A$ is compact, then $A^{\ast}A$ is compact self-adjoint positive
semidefinite operator. The square roots of the eigenvalues of $A^{\ast}A$ are
called \emph{singular values} of $A$, and they are denoted by $\sigma_{n}$.
With the corresponding eigenvectors $v_{n}$ of $A^{\ast}A$, we have
\[
\left(  A^{\ast}A\right)  ^{1/2}=\sum_{n=1}^{\infty}\sigma_{n}v_{n}\otimes
v_{n},\quad A^{\ast}Av_{n}=\sigma_{n}^{2}v_{n}.
\]
It is an easy exercise to verify that, just like in finite dimension,%
\[
\left\vert A\right\vert =\max\sigma_{n},
\]
and, in addition, if $A$ is self-adjoint and positive semidefinite, then the
singular values and eigenvalues are the same, $\sigma_{n}=\lambda_{n}$, and
$\left(  A^{\ast}A\right)  ^{1/2}=\left(  A^{2}\right)  ^{1/2}=A$.

\subsection{Trace class operators}

An operator $A\in\left[  H\right]  $ is trace class operator if for some
orthonormal sequence $\left\{  e_{n}\right\}  $ in $H,$%
\[
\sum_{n=1}^{\infty}\left\langle \left(  A^{\ast}A\right)  ^{1/2}e_{n}%
,e_{n}\right\rangle <\infty.
\]
The space of all trace class operators is denoted by $\operatorname*{Tr}%
\left(  H\right)  $ and it is a Banach space equipped with the norm%
\[
\left\vert A\right\vert _{\operatorname*{Tr}\left(  H\right)  }=\sum
_{n=1}^{\infty}\left\langle \left(  A^{\ast}A\right)  ^{1/2}e_{n}%
,e_{n}\right\rangle =\operatorname*{Tr}\left(  A^{\ast}A\right)  ^{1/2},
\]
where the trace of a selfadjoint operator $A\in\operatorname*{Tr}\left(
H\right)  $, $A=A^{\ast}$ is defined by%
\[
\operatorname*{Tr}A=\sum_{n=1}^{\infty}\left\langle Ae_{n},e_{n}\right\rangle
\]
The value of $\operatorname*{Tr}A$ and thus $\left\vert A\right\vert
_{\operatorname*{Tr}\left(  H\right)  }$ do not depend on the choice of the
orthonormal sequence $\left\{  e_{n}\right\}  $: if $\left\{  f_{m}\right\}  $
is another orthonormal sequence, we expand%
\[
Ae_{n}=\sum_{m=1}^{\infty}\left\langle Ae_{n},f_{m}\right\rangle f_{m}%
\]
and we get%

\begin{align*}
\sum_{n=1}^{\infty}\left\langle Ae_{n},e_{n}\right\rangle  &  =\sum
_{n=1}^{\infty}\left\langle \sum_{m=1}^{\infty}\left\langle Ae_{n}%
,f_{m}\right\rangle f_{m},e_{n}\right\rangle \\
&  =\sum_{n=1}^{\infty}\sum_{m=1}^{\infty}\left\langle Ae_{n},f_{m}%
\right\rangle \left\langle f_{m},e_{n}\right\rangle \\
&  =\sum_{m=1}^{\infty}\sum_{n=1}^{\infty}\left\langle Af_{m},e_{n}%
\right\rangle \left\langle e_{n},f_{m}\right\rangle \\
&  =\sum_{m=1}^{\infty}\left\langle \sum_{n=1}^{\infty}\left\langle
Af_{m},e_{n}\right\rangle e_{n},f_{m}\right\rangle \\
&  =\sum_{m=1}^{\infty}\left\langle Af_{m},f_{m}\right\rangle .
\end{align*}

Trace class operators are compact, therefore all points of their spectrum are
eigenvalues of finite multiplicity except possibly for the point zero, and
zero is the only possible cluster point of the spectrum. Trace class operators
will play an important role as covariance operators of random elements in a
Hilbert space.

\subsection{Hilbert-Schmidt operators}

An operator $A\in\left[  H\right]  $ is Hilbert-Schmidt operator if for an
orthonormal sequence $\left\{  e_{n}\right\}  $ in $H,$%
\[
\sum_{n=1}^{\infty}\left\langle Ae_{n},Ae_{n}\right\rangle =\sum_{n=1}%
^{\infty}\left\langle A^{\ast}Ae_{n},e_{n}\right\rangle =\sum_{n=1}^{\infty
}\left\vert Ae_{n}\right\vert ^{2}<\infty.
\]
The space of all Hilbert-Schmidt operators is denoted by $HS\left(  H\right)
$ and it is a separable Hilbert space equipped with the inner product%
\begin{equation}
\left\langle A,B\right\rangle _{HS\left(  H\right)  }=\sum_{n=1}^{\infty
}\left\langle Ae_{n},Be_{n}\right\rangle ,
\label{eq:Hilbert-Schmidt-inner-product}%
\end{equation}
and the norm%
\[
\left\vert A\right\vert _{HS\left(  H\right)  }=\left\langle A,A\right\rangle
_{HS\left(  H\right)  }^{1/2}=\left(  \sum_{n=1}^{\infty}\left\vert
Ae_{n}\right\vert ^{2}\right)  ^{1/2},
\]
which again do not depend on the choice of $\left\{  e_{n}\right\}  $, just
like the trace. Hilbert-Schmidt operators are also compact.

With $\sigma_{n}$ the singular values of $A$ and $e_{n}$ the singular vectors,
we see immediately that%
\[
\max\sigma_{n}\leq\left(  \sum_{n=1}^{\infty}\sigma_{n}^{2}\right)  ^{1/2}%
\leq\sum_{n=1}^{\infty}\sigma_{n},
\]
thus%
\begin{equation}
\left\vert A\right\vert \leq\left\vert A\right\vert _{HS\left(  H\right)
}\leq\left\vert A\right\vert _{\operatorname*{Tr}\left(  H\right)  }.
\label{eq:compare-hilbert-schmidt-norm}%
\end{equation}
In particular, every trace-class operator is also a Hilbert-Schmidt operator,
which is a bounded operator,%
\[
\left[  H\right]  \supset HS\left(  H\right)  \supset\operatorname*{Tr}\left(
H\right)  .
\]

The Hilbert-Schmidt norm of a tensor product $x\otimes x$ can be computed easily.

\begin{lemma}
For any $x\in H$,%
\begin{equation}
\left\vert x\otimes x\right\vert _{HS\left(  H\right)  }=\left\vert
x\right\vert _{H}^{2}. \label{eq:hilbert-schmidt-tensor}%
\end{equation}

\end{lemma}

\begin{proof}
Let $\left\{  e_{n}\right\}  $ be an orthonormal sequence in $H$. Then,%
\begin{align*}
\left\vert x\otimes x\right\vert _{HS\left(  H\right)  }^{2}  &  =\sum
_{n=1}^{\infty}\left\vert \left(  x\otimes x\right)  e_{k}\right\vert _{H}%
^{2}\\
&  =\sum_{n=1}^{\infty}\left\vert x\left\langle x,e_{k}\right\rangle
_{H}\right\vert _{H}^{2}\\
&  =\left\vert x\right\vert _{H}^{2}\sum_{n=1}^{\infty}\left\vert \left\langle
x,e_{k}\right\rangle \right\vert ^{2}\\
&  =\left\vert x\right\vert _{H}^{2}\left\vert x\right\vert _{H}^{2}%
\end{align*}
by Parseval's equality.
\end{proof}

In the finite-dimensional case with $H=\mathbb{R}^{m}$, $\left\langle
u,v\right\rangle =\sum_{n=1}^{m}u_{n}v_{n}$, the Hilbert-Schmidt operator norm
is the Frobenius matrix norm: With $e_{n}$ the canonical basis vectors, we
have%
\[
\left\vert A\right\vert _{HS\left(  H\right)  }^{2}=\sum_{n=1}^{m}\left\vert
Ae_{n}\right\vert ^{2}=\sum_{n=1}^{m}\left\vert A_{:,n}\right\vert ^{2}%
=\sum_{n=1}^{m}\sum_{k=1}^{m}a_{kn}^{2},
\]
where $A_{:,n}$ denotes the $n-th$ column of $A$.

The utility of the space Hilbert-Schmidt operators comes from the facts that
it is a Hilbert space, and that sample covariances are Hilbert-Schmidt operators.

See, for example \cite{DaPrato-1992-SEI,Kuo-1975-GMB,Lax-2002-FA}, and a short
introduction in \cite{Cupidon-2007-DMA} for more details on trace-class and
Hilbert-Schmidt operators.

\section{Probability on Hilbert spaces}

Because we need to question all established concepts to see what survives in
infinite dimension, we need to revisit the basics of probability theory. The
intuitive probability is not good enough here; to prove anything and navigate
safely, we need measure-theoretic probability. To start with, the concept of
probablity density (with respect to Lebesgue measure), which is ubiquitous in
finite dimensional probability, is no longer available, because there is
Lebesgue measure in infinite dimension.

\begin{theorem}
\label{thm:no-lebesgue}Let $U$ be infinitely dimensional normed linear space.
There is no measure on Borel sets of $U$ that is invariant to shifts or
invariant to rotations, and positive but finite on all balls.
\end{theorem}

\begin{proof}
We show the proof on an example. Consider the space $\ell^{2}$ of real
sequences $v=\left\{  v_{n}\right\}  $ with the norm $\left\vert v\right\vert
=\left(  v_{1}^{2}+v_{2}^{2}+\cdots\right)  ^{1/2}$, and suppose that $\mu$ is
such measure on $\ell^{2}$. The unit vectors
\begin{align*}
e_{1}  &  =\left(  1,0,0,\ldots\right)  ,\\
e_{2}  &  =\left(  0,1,0,\ldots\right)  ,
\end{align*}
satisfy $\left\vert e_{n}\right\vert =1$ and $\left\vert e_{m}-e_{n}%
\right\vert >1$, so any two open balls $B_{1/2}\left(  e_{n}\right)  =\left\{
v\in V:\left\vert v-e_{n}\right\vert <1/2\right\}  $ do not instersect, and
their union is contained in a bigger ball $B_{2}\left(  0\right)  =\left\{
v\in V:\left\vert v\right\vert <2\right\}  $ centered at zero,%
\[
B_{1/2}\left(  e_{1}\right)  \cup B_{1/2}\left(  e_{2}\right)  \cup
\cdots\subset B_{2}\left(  0\right)  .
\]
Since the measure $\mu$ is rotation (or translation) invariant, $\mu\left(
B_{1/2}\left(  e_{n}\right)  \right)  =\mu\left(  B_{1/2}\left(  e_{1}\right)
\right)  $ for all $n$. Thus,%
\[
\mu\left(  B_{1/2}\left(  e_{1}\right)  \right)  +\mu\left(  B_{1/2}\left(
e_{1}\right)  \right)  +\cdots\leq\mu\left(  B_{2}\left(  0\right)  \right)
,
\]
which is a contradiction since $\mu\left(  B_{1/2}\left(  e_{1}\right)
\right)  >0$ and $\mu\left(  B_{2}\left(  0\right)  \right)  <\infty$.
\end{proof}

\begin{remark}
Note we said \textquotedblleft normed linear"\ space. That's so that we know
what open sets and therefore Borel sets are. We also used the norm to define
the balls that served in the contradition. How can you get around the
contradiction if you need to? Change the definition of open sets so that open
balls are not measurable. That is, use a different topology. See Chapter
\ref{ch:white noise}.
\end{remark}

\subsection{Random elements}

Let $\left(  \Omega,\mathcal{F},\mu\right)  $ be a probability space, that is,
a measure space such that $\mu\left(  \Omega\right)  =1$. Here, $\mathcal{F}$
is the system of subsets $S$ of $\Omega$ such that the measure $\mu\left(
S\right)  $ is defined, and it is a $\sigma$-algebra,%
\[
\emptyset\in\mathcal{F,\quad}A\in\mathcal{F\Longrightarrow}\Omega
\mathcal{\setminus A\in F,\quad}A_{1},A_{2},\ldots\in\mathcal{F\Longrightarrow
}%
{\displaystyle\bigcup\limits_{n=1}^{\infty}}
A_{n}\in\mathcal{F}%
\]
and $\Pr$ is a measure: $\mu\left(  A\right)  \geq0$ defined for all
$A\in\mathcal{F}$, $\mu\left(  \emptyset\right)  =0$, and $\mu$ is $\sigma
$-additive:%
\[
A_{1},A_{2},\ldots\in\mathcal{F}\text{ disjoint }\Longrightarrow\mu\left(
{\displaystyle\bigcup\limits_{n=1}^{\infty}}
A_{n}\right)  =%
{\displaystyle\sum\limits_{n=1}^{\infty}}
\mu\left(  A_{n}\right)  .
\]
We emphasize that $\sigma$\emph{-additivity is required for the usual
probability theory to go through}.

\emph{Random variable} is a real-valued measurable function $X$ on $\left(
\Omega,\mathcal{F}\right)  $, that is, the inverse image $X^{-1}\left(
A\right)  =\left\{  \omega\in\Omega|X\left(  \omega\right)  \in A\right\}
\in\mathcal{F}$ for any Borel set $A\subset\mathbb{R}$.\footnote{Unlike in
Real Analysis, non-measurable functions in probability are not uncommon. When
$f$ is not measurable, it simply means that it can distinguish subsets of
$\Omega$ which are not visible by the $\sigma$-algebra $\mathcal{F}$. For
example, if you throw a dice and cannot distinguish between $2$ and $3$ (say,
those sides are smeared), any function that depends on this distinction is not
measurable.}\footnote{Borel sets are the smallest $\sigma$-algebra containing
all open sets.}. The notation%
\[
\Pr\left(  X\in A\right)  =\mu\left(  X^{-1}\left(  A\right)  \right)
=\mu\left(  \left\{  \omega\in\Omega|X\left(  \omega\right)  \in A\right\}
\right)
\]
for the probability that the value of $X$ is in $A$ is used. The mean value
(if it exists) of a random variable $X$ is the abstract (Lebesgue) integral
with respect to the measure $\mu$,
\[
E\left(  X\right)  =\int_{\Omega}X\left(  \omega\right)  \mu\left(
d\omega\right)  .
\]
To simplify notation, the measure $\mu$ in the integrals is assumed and we
write simply
\[
E\left(  X\right)  =\int_{\Omega}X\left(  \omega\right)  d\omega.
\]

Now let $H$ be a separable Hilbert space. The $\sigma$-algebra of all Borel
sets in $H$ is denoted by $\mathcal{B}\left(  H\right)  $. A random element
$X$ with values in $H$ is a measurable function $X:\Omega\rightarrow H$, that
is, such that\footnote{The usual definition of measurable function with values
in a metric space.}%
\begin{equation}
\forall B\in\mathcal{B}\left(  H\right)  :X^{-1}\left(  B\right)
\in\mathcal{F} \label{eq:measurable}%
\end{equation}
Given an $H$-valued random element $X$,%
\[
\mu_{X}\left(  B\right)  =\mu\left(  X^{-1}\left(  B\right)  \right)
=\Pr\left(  X\in B\right)
\]
is a probability measure on $\mathcal{B}\left(  H\right)  $ (commonly called
just a probability measure on $H$), called the \emph{distribution} of $X$.

\subsection{$L^{p}$ spaces}

\label{sec:Lp}Since a continuous function of a measurable function is
measurable, it follows that $\left\vert X\right\vert $, and $\left\langle
u,X\right\rangle $, for any $u\in H$, are random variables\footnote{The latter
property is called weak measurability, while measurable in the usual sense
(\ref{eq:measurable}) is called strongly measurable. When $H$ is separable,
weak measurability is equivalent to strong measurability by Pettis
measurability theorem \cite{Pettis-1938-IVS}.}. For $1\leq p<\infty$, define
$L^{p}=L^{p}\left(  \Omega,H\right)  =L^{p}\left(  \Omega,\mathcal{F}%
,\mu,H\right)  $ as the space of all $H$-valued random elements $X$ such that
the \emph{moment}%
\[
E\left(  \left\vert X\right\vert ^{p}\right)  =\int_{\Omega}\left\vert
X\left(  \omega\right)  \right\vert ^{p}d\omega<\infty
\]
As usual, we do not distinguish between elements of $L^{p}$ equal a.s. The
space $L^{p}$ equipped with the norm
\[
\left\Vert X\right\Vert _{p}=\left(  E\left(  \left\vert X\right\vert
^{p}\right)  \right)  ^{1/p}%
\]
is a Banach space, and $L^{2}\left(  \Omega,H\right)  $ equipped with the
inner product
\[
E\left(  \left\langle U,V\right\rangle \right)  =\int_{\Omega}\left\langle
U\left(  \omega\right)  ,V\left(  \omega\right)  \right\rangle d\omega
\]
is a Hilbert space\footnote{The corresponding definition of $L^{\infty}$ is
the space of all random elements $X$ such that for some $c\in\mathbb{R}$,
$\left\vert X\right\vert \leq c$ a.s., but are not concerned with the case
$p=\infty$ here.}\footnote{In the case $H=\mathbb{R}$, these are the same
spaces as the $L^{p}$ spaces of real functions of a real variable, except on
an abstract set $\Omega$, without any topology, rather than on $\mathbb{R}$.
Abstract properties of the measure and integral that do not require the
topology of $\mathbb{R}$, carry over.}.

In particular, given a Hilbert space $K$ and a probability measure $\mu$ on
$K$, the space
\[
L^{p}\left(  K,\mu\right)  =L^{p}\left(  K,\mathcal{B}\left(  K\right)
,\mu,\mathbb{R}\right)
\]
is the $L^{p}$ space of all real functions on $H$ that are measurable on $H$
with respect to the Borel sets of $H$, and such that $\int_{H}\left\vert
f\right\vert ^{p}d\mu<\infty$.

If $X\in L^{2}\left(  \Omega,H\right)  $, then by the Cauchy inequality on
$L^{2}\left(  \Omega,H\right)  $,%
\[
\left\Vert X\right\Vert _{1}=\left\vert E\left(  \left\vert X\right\vert
\right)  \right\vert \leq\left\Vert X\right\Vert _{2}%
\]
because%
\[
\left\Vert X\right\Vert _{1}=E\left(  \left\vert X\right\vert \right)
=E\left(  \left\vert X\right\vert \cdot1\right)  \leq\left\Vert X\right\Vert
_{2}\left\Vert 1\right\Vert _{2}=\left\Vert X\right\Vert _{2},
\]
so%
\[
L^{2}\left(  \Omega,H\right)  \subset L^{1}\left(  \Omega,H\right)  .
\]
More generally,
\[
p>r\geq1\Longrightarrow L^{p}\left(  \Omega,H\right)  \subset L^{r}\left(
\Omega,H\right)  .
\]

\subsection{Mean value of random element}

For $X\in L^{2}\left(  \Omega,H\right)  $, the mean value $E\left(  X\right)
$ defined by\footnote{Thus, $E\left(  U\right)  =%
{\textstyle\int\nolimits_{\Omega}}
Ud\omega$ is the Gelfand-Pettis integral \cite{Pettis-1938-IVS} (the weak
integral), not the Bochner integral (the strong integral).}%
\begin{equation}
E\left(  X\right)  \in H,\quad\left\langle v,E\left(  X\right)  \right\rangle
=E\left(  \left\langle v,X\right\rangle \right)  , \quad\forall v\in H,
\label{eq:def-mean}%
\end{equation}
exists and is unique from the Riesz representation theorem:\footnote{For every
bounded linear functional $f$ on $H$, there exists a unique $v\in H$ such that
$\left\langle u,v\right\rangle =f\left(  v\right)  $ $\forall v\in H$. In
addition, $\left\vert u\right\vert _{H}=\left\vert f\right\vert _{H^{\prime}}%
$.} Since
\[
\left\vert E\left(  \left\langle v,X\right\rangle \right)  \right\vert \leq
E\left(  \left\vert v\right\vert \left\vert X\right\vert \right)
\leq\left\vert v\right\vert E\left(  \left\vert X\right\vert \right)  ,
\]
by the Cauchy inequality on $H$, the mapping $v\mapsto E\left(  \left\langle
v,X\right\rangle \right)  $ is a bounded linear functional on $H$, therefore
there exists unique
\[
u\in H\text{ such that }\left\langle v,u\right\rangle =E\left(  \left\langle
v,X\right\rangle \right)  \text{ for all }v\in H.
\]
Note that when $H=\mathbb{R}^{n}$ and $v=e_{k}$, the $k$-th coordinate vector,
then $\left\langle v,E\left(  X\right)  \right\rangle $ is simply entry $k$ of
the vector $E\left(  X\right)  $.

For \emph{operator-valued random element} $U\in L^{1}\left(  \Omega,\left[
H\right]  \right)  $, the mean value $E\left(  U\right)  $, defined by%
\begin{equation}
E\left(  U\right)  \in\left[  H\right]  ,\quad\left\langle u,E\left(
U\right)  v\right\rangle =E\left(  \left\langle u,Uv\right\rangle \right)
\quad\forall u,v\in H, \label{eq:def-operator-mean}%
\end{equation}
exists in $\left[  H\right]  $ and it is unique, from Lemma \ref{lem:exist-op}
and the Cauchy inequality in $H$,%
\[
\left\vert E\left(  \left\langle u,Uv\right\rangle \right)  \right\vert
\leq\left\vert E\left(  \left\vert u\right\vert \left\vert U\right\vert
\left\vert v\right\vert \right)  \right\vert =\left\vert u\right\vert
\left\vert v\right\vert E\left(  \left\vert U\right\vert \right)  .
\]

From standard properties of the abstract integral \cite{Lang-1993-RFA},%
\[
\left\vert E\left(  X\right)  \right\vert \leq E\left(  \left\vert
X\right\vert \right)  =\left\Vert X\right\Vert _{1}.
\]

A random variable $X$ and its distribution are called \emph{centered} if
$E\left(  X\right)  =0$.

A random element $X$ is a constant if $X\left(  \omega\right)  $ is the same
for all $\omega\in\Omega$. A constant random element with values in $H$ can be
considered as simply a (non-random, deterministic) element of $H$, and we can
write $X\in H$. If $X$ is a constant, then%
\begin{equation}
\left\Vert X\right\Vert _{p}=\left(  E\left(  \left\vert X\right\vert
^{p}\right)  \right)  ^{1/p}=\left\vert X\right\vert ,\quad1\leq p<\infty.
\label{eq:const-norm}%
\end{equation}
In particular, a constant $H$-valued random element is in all spaces $L^{p}$,
$p\geq1$.

\subsection{Mean of tensor product}

The tensor product with a constant and the mean commute: If $X,Y\in
L^{1}\left(  \Omega,H\right)  $, $x,y\in H,$ then%

\begin{equation}
E\left(  X\otimes y\right)  =E\left(  X\right)  \otimes y,
\label{eq:tensor-mean-Xy}%
\end{equation}
and%

\begin{equation}
E\left(  x\otimes Y\right)  =x\otimes E\left(  Y\right)
\label{eq:tensor-mean-xY}%
\end{equation}

Indeed, if and $u,v\in H$, then%
\begin{align*}
\left\langle u,E\left(  X\otimes y\right)  v\right\rangle  &  =E\left(
\left\langle u,\left(  X\otimes y\right)  v\right\rangle \right) \\
&  =E\left(  \left\langle u,X\right\rangle \left\langle y,v\right\rangle
\right) \\
&  =\left\langle u,E\left(  X\right)  \right\rangle \left\langle
y,v\right\rangle \\
&  =\left\langle u,\left(  E\left(  X\right)  \otimes y\right)
v\right\rangle
\end{align*}
by the definition (\ref{eq:def-operator-mean}) of the mean of a random
mapping, the definition (\ref{eq:def-tensor}) of tensor product, the
definition (\ref{eq:def-mean}) of mean of a random vector, and again the the
definition (\ref{eq:def-tensor}) of tensor product. This shows
(\ref{eq:tensor-mean-Xy}). The proof of (\ref{eq:tensor-mean-xY}) is similar.
See also Lemma \ref{lem:exist-op}.

\subsection{Covariance}

The covariance of random elements $X$, $Y$ is the linear operator
$\operatorname*{Cov}\left(  X,Y\right)  $, defined by\footnote{The existence
and uniqueness of $\operatorname*{Cov}\left(  X,Y\right)  $ follows again from
Riesz representation theorem. See \cite[Ch. 1]{DaPrato-2006-IIA} for more
details. However we used the Riesz representation theorem in the definition of
tensor product and the mean of the tensor product... it is exactly the same
but we need it several times so we did it just once.}%
\begin{equation}
\left\langle u,\operatorname*{Cov}\left(  X,Y\right)  v\right\rangle =E\left(
\left\langle u,X-E\left(  X\right)  \right\rangle \left\langle v,Y-E\left(
Y\right)  \right\rangle \right)  \quad\forall u,v\in H,
\label{eq:def-covariance}%
\end{equation}

$\operatorname*{Cov}\left(  X\right)  $ stands for $\operatorname*{Cov}\left(
X,X\right)  $. Clearly, if $\operatorname*{Cov}\left(  X\right)  $ exists, it
is self-adjoint and positive semidefinite.

\begin{theorem}
\label{thm:cov-existence}If $X,Y\in L^{2}\left(  \Omega,H\right)  $,
$\operatorname*{Cov}\left(  X,Y\right)  \in\left[  H\right]  $ exists and%
\begin{align}
\operatorname*{Cov}\left(  X,Y\right)   &  =E\left(  \left(  X-E\left(
X\right)  \right)  \otimes\left(  Y-E\left(  Y\right)  \right)  \right)
\nonumber\\
&  =E\left(  X\otimes Y\right)  -E\left(  X\right)  \otimes E\left(  Y\right)
. \label{eq:covariance2}%
\end{align}

\end{theorem}

\begin{proof}
Writing%
\[
\left\langle u,X-E\left(  X\right)  \right\rangle \left\langle v,Y-E\left(
Y\right)  \right\rangle =\left\langle u,\left(  X-E\left(  X\right)  \right)
\otimes\left(  Y-E\left(  Y\right)  \right)  v\right\rangle
\]
we have from the definition of the mean of a linear operator
(\ref{eq:def-operator-mean}),
\begin{align*}
\left\langle u,\operatorname*{Cov}\left(  X,Y\right)  v\right\rangle  &
=E\left(  \left\langle u,\left(  X-E\left(  X\right)  \right)  \otimes\left(
Y-E\left(  Y\right)  \right)  v\right\rangle \right) \\
&  =\left\langle u,E\left(  \left(  X-E\left(  X\right)  \right)
\otimes\left(  Y-E\left(  Y\right)  \right)  \right)  v\right\rangle
\end{align*}
and by the Cauchy inequality and triangle inequality%
\begin{align*}
\left\vert E\left(  \left\langle u,X-E\left(  X\right)  \right\rangle
\left\langle v,Y-E\left(  Y\right)  \right\rangle \right)  \right\vert  &
\leq\left\vert E\left(  \left\langle u,X-E\left(  X\right)  \right\rangle
\right)  \right\vert \left\vert E\left(  \left\langle v,Y-E\left(  Y\right)
\right\rangle \right)  \right\vert \\
&  \leq\left\vert E\left(  \left\vert u\right\vert \left\vert X-E\left(
X\right)  \right\vert \right)  \right\vert \left\vert E\left(  \left\vert
v\right\vert \left\vert Y-E\left(  Y\right)  \right\vert \right)  \right\vert
\\
&  \leq\left\vert u\right\vert \left\Vert X-E\left(  X\right)  \right\Vert
_{2}\left\vert v\right\vert \left\Vert Y-E\left(  Y\right)  \right\Vert _{2}\\
&  \leq4\left\vert u\right\vert \left\vert v\right\vert \left\Vert
X\right\Vert \left\Vert Y\right\Vert _{2}%
\end{align*}
for any $u,v\in H$. Therefore, $\operatorname*{Cov}\left(  X,Y\right)  $
exists and $\operatorname*{Cov}\left(  X,Y\right)  $ $\in L\left(  H\right)  $
by Lemma \ref{lem:exist-op} (corollary of Riesz representation theorem). To
prove (\ref{eq:covariance2}),%
\begin{align*}
\operatorname*{Cov}\left(  X,Y\right)   &  =E\left(  \left(  X-E\left(
X\right)  \right)  \otimes\left(  Y-E\left(  Y\right)  \right)  \right) \\
&  =E\left(  X\otimes Y-X\otimes E\left(  Y\right)  -E\left(  X\right)
\otimes Y+E\left(  X\right)  \otimes E\left(  Y\right)  \right) \\
&  =E\left(  X\otimes Y\right)  -E\left(  X\right)  \otimes E\left(  Y\right)
-E\left(  X\right)  \otimes E\left(  Y\right)  +E\left(  X\right)  \otimes
E\left(  Y\right) \\
&  =E\left(  X\otimes Y\right)  -E\left(  X\right)  \otimes E\left(  Y\right)
,
\end{align*}
using (\ref{eq:tensor-mean-Xy}) and (\ref{eq:tensor-mean-xY}).
\end{proof}

\begin{theorem}
Suppose that $X\in L^{1}\left(  \Omega,H\right)  $. Then $X\in L^{2}\left(
\Omega,H\right)  $ if and only if $\operatorname*{Cov}\left(  X\right)  $
exists and $\operatorname*{Tr}\operatorname*{Cov}\left(  X\right)  <\infty$.
In fact,%
\[
\left\Vert X-E\left(  X\right)  \right\Vert _{L^{2}\left(  \Omega,H\right)
}^{2}=\operatorname*{Tr}\operatorname*{Cov}\left(  X\right)  .
\]

\end{theorem}

\begin{proof}
We follow \cite[proof of Theorem 2.1]{Kuo-1975-GMB}. Without loss of
generality, assume $E\left(  X\right)  =0$. Suppose for the moment that
$C=\operatorname*{Cov}\left(  X\right)  $ exists. Then%
\begin{equation}
\left\langle u,Cv\right\rangle =E\left(  \left\langle X,u\right\rangle
\left\langle X,v\right\rangle \right)  \quad\forall u,v\in H. \label{eq:cov-0}%
\end{equation}
Choose a complete orthonormal set $\left\{  e_{n}\right\}  $. For a fixed
$\omega$, decompose $X\left(  \omega\right)  $ in the abstract Fourier series%
\[
X\left(  \omega\right)  =\sum_{n=1}^{\infty}\,\left\langle X\left(
\omega\right)  ,e_{n}\right\rangle e_{n}%
\]
convergent in the norm of $H$. By Parseval equality,
\[
\left\vert X\left(  \omega\right)  \right\vert ^{2}=\sum_{n=1}^{\infty
}\left\vert \,\left\langle X\left(  \omega\right)  ,e_{n}\right\rangle
\right\vert ^{2}.
\]
Because this is a series of nonegative random variables (=real-valued
measurable functions of $\omega$), by the monotone convergence theorem,%
\begin{equation}
\int_{\Omega}\left\vert X\left(  \omega\right)  \right\vert ^{2}d\omega
=\sum_{n=1}^{\infty}\int_{\Omega}\left\vert \,\left\langle X\left(
\omega\right)  ,e_{n}\right\rangle \right\vert ^{2}d\omega=\sum_{n=1}^{\infty
}\left\langle Ce_{n},e_{n}\right\rangle =\operatorname*{Tr}C \label{eq:L2-tr}%
\end{equation}
because $\left\langle Ce_{n},e_{n}\right\rangle =\lambda_{n}$.

Let $\operatorname*{Tr}C<\infty$. Then $X\in L^{2}\left(  \Omega,H\right)  $
from (\ref{eq:L2-tr}).

Let $X\in L^{2}\left(  \Omega,H\right)  $. By Theorem \ref{thm:cov-existence},
$C$ exists, and $\operatorname*{Tr}C<\infty$ from (\ref{eq:L2-tr}).
\end{proof}

\begin{remark}
When $H=\mathbb{R}^{n}$ and $u=e_{k}$, $v=e_{\ell}$, and the linear operator
$\operatorname*{Cov}\left(  X,Y\right)  $ is identified with its matrix, then
$\left\langle u,\operatorname*{Cov}\left(  X,Y\right)  v\right\rangle $ is
simply the entry $k$, $\ell$ of the covariance matrix. In the case when $n=1$,
$\operatorname*{Cov}\left(  X\right)  $ is also called the variance of the
random variable $X$ and denoted by $\operatorname*{Var}\left(  X\right)  $.
\end{remark}

\subsection{Convergence of random elements}

There are many types of convergence of random variables. With many types of
convergence in infinitely dimensional spaces, there are many, many more. But
here just replace the absolute value by the norm in $H$.

\emph{Convergence in norm almost surely (a.s.)} is defined as%
\begin{align*}
X_{k}\rightarrow X\text{ in }H\text{ a.s.}  &  \Leftrightarrow\lim
_{k\rightarrow\infty}\left\vert X_{k}\left(  \omega\right)  -X\left(
\omega\right)  \right\vert =0\quad\forall\omega\in\Omega\setminus
\Theta,\text{\quad}\mu\left(  \Theta\right)  =0\\
&  \Leftrightarrow\Pr\left(  \lim_{k\rightarrow\infty}\left\vert X_{k}\left(
\omega\right)  -X\left(  \omega\right)  \right\vert =0\right)  =1
\end{align*}

Since $L^{p}\left(  \Omega,H\right)  $ is a normed space, we have the usual
notion of \emph{convergence in the }$L^{p}$\emph{ norm},%
\[
X_{k}\rightarrow X\text{ in }L^{p}\Longleftrightarrow X\in L^{p}\left(
\Omega,H\right)  \wedge\left\Vert X_{k}-X\right\Vert _{p}\rightarrow0.
\]

\emph{Convergence in distribution} of random elements, denoted by
$X_{k}\Longrightarrow X$, is defined as weak convergence of their
distributions,%
\begin{equation}
\lim_{k\rightarrow\infty}%
{\displaystyle\int\limits_{H}}
\phi d\mu_{X_{k}}=%
{\displaystyle\int\limits_{H}}
\phi d\mu_{X}\quad\forall\phi\in C_{b}\left(  H\right)  ,
\label{eq:def-weak-conv}%
\end{equation}
where $C_{b}\left(  H\right)  $ is the space of all continuous bounded
functions on $H$.

\emph{Convergence in probability} of random elements is defined as%
\[
\lim_{k\rightarrow\infty}\Pr\left(  \left\vert X_{k}-X\right\vert
>\varepsilon\right)  =0\quad\forall\varepsilon>0.
\]

Convergence a.s. implies convergence in probability, which implies convergence
in distribution. Convergence in distribution to a constant implies convergence
in probability \cite[Lemma 4.7]{Kallenberg-2002-FMP}.\footnote{In fact, in any
metric space.}

Convergence in $L^{p}$ implies convergence a.s. and in probability, but the
converse is in general not true. However, we have the converse under an
additional boundedness assumption. This result is a strengthening of the
Lebesgue dominated convergence theorem.

\begin{lemma}
[Uniform integrability]\label{eq:lem-uniform-integrability}If $\left\{
X_{k}\right\}  $ is a bounded sequence in $L^{p}\left(  \Omega,H\right)  $ and
$X_{k}\rightarrow X$ in probability, then $\left\Vert X_{k}-X\right\Vert
_{q}\rightarrow0$ for all $1\leq q<p$.
\end{lemma}

\begin{proof}
It is well known that for random variables, the lemma follows from uniform
integrability \cite[p. 338]{Billingsley-1995-PM}. For random elements $X_{k}$,
$X$ on $H$, consider the random variables $U_{k}=\left\vert X_{k}-X\right\vert
$.
\end{proof}

\begin{lemma}
[{Slutsky's theorem, \cite[page 254, Corollary 2]{Chow-1988-PT}}]If $\left\{
X_{k}\right\}  $, $\left\{  Y_{k}\right\}  $are random variables,
$X_{k}\Longrightarrow X$, $Y_{k}\Longrightarrow Y$, and $Y$ is constant, then
$X_{k}Y_{k}\Longrightarrow XY$ and $X_{k}+Y_{k}\Longrightarrow X+Y$.
\end{lemma}

\begin{lemma}
[{Slutsky's theorem, \cite[Theorem 18.8, page 161]{Jacod-2003-PE}}]If
$\left\{  X_{k}\right\}  $, $\left\{  Y_{k}\right\}  $are random elements in
$\mathbb{R}^{m}$, $X_{k}\Longrightarrow X$, $Y_{k}\Longrightarrow Y$, and $Y$
is constant, $X_{k}+Y_{k}\Longrightarrow X+Y$.
\end{lemma}

\begin{lemma}
[{Continuous mapping theorem, \cite[Theorem 2.3]{vanderVaart-2000-AS}}]If
$X_{n}$ are random elements with values on a metric space $M$, and $X_{k},X$
random elements with values in $M$, $f$ function from $M$ to another metric
space, the probability that $X$ attains value where $f$ is discontinuous is
zero, then $X_{k}\rightarrow X$ implies $f\left(  X_{k}\right)  \rightarrow
f\left(  X\right)  $. This is true for convergence in distribution,
convergence in probability, and convergence almost surely.
\end{lemma}

\subsection{Karhunen-Lo\`{e}ve explansion and random orthonormal series}

\label{sec:smooth}

Suppose $U\in L^{2}\left(  \Omega,H\right)  $ with the covariance operator
$C=\operatorname*{Cov}\left(  X\right)  $. Since $C$ is compact self-adjoint
operator, there exists a complete orthonormal system $\left\{  u_{n}\right\}
$ of eigenvectors of $U$:%
\[
Cu_{n}=\lambda_{n}u_{n}.
\]
Assume for the moment that $E\left(  U\right)  =0$ and for a fixed $\omega
\in\Omega$, expand $U$ in an abstract Fourier series using the orthonormal
system $\left\{  u_{n}\right\}  $:%
\begin{equation}
U\left(  \omega\right)  =%
{\textstyle\sum\nolimits_{n=1}^{\infty}}
\theta_{n}\left(  \omega\right)  u_{n},\text{ where }\theta_{n}\left(
\omega\right)  =\left\langle U\left(  \omega\right)  ,u_{n}\right\rangle .
\label{eq:U-fourier-expansion}%
\end{equation}
So the sum (\ref{eq:U-fourier-expansion}) converges in $H$ a.s. Compute:%
\begin{align*}
E\left(  \theta_{n}\right)   &  =E\left(  \left\langle U,u_{m}\right\rangle
\right)  =\left\langle E\left(  U\right)  ,u_{m}\right\rangle =0\\
E\left(  \theta_{m}\theta_{n}\right)   &  =E\left(  \left\langle
U,u_{m}\right\rangle \left\langle U,u_{n}\right\rangle \right)  =\left\langle
u_{m},Cu_{n}\right\rangle =\left\{
\begin{array}
[c]{c}%
\lambda_{n}\text{ if }m=n\\
0\text{ if }m\neq0
\end{array}
\right.
\end{align*}

Back to the of general $E\left(  U\right)  $ and substituting $\theta
_{n}=\lambda_{n}^{1/2}\xi_{n}$ (and choosing suitable $\xi_{n}$ if
$\lambda_{n}=0$) we have $E\left(  \xi_{m}\xi_{n}\right)  =\delta_{mn}$ and
(\ref{eq:U-fourier-expansion})\ becomes the \emph{Karhunen-Lo\`{e}ve
expansion,}%
\begin{equation}
U=E\left(  U\right)  +%
{\textstyle\sum\nolimits_{n=1}^{\infty}}
\lambda_{n}^{1/2}\xi_{n}u_{n}\text{,\quad}E\left(  \xi_{n}\right)  =0.
\label{eq:KL-expansion}%
\end{equation}
convergent a.s. in $H$ and \emph{double-orthogonal}:%
\[
\left\langle u_{m},u_{n}\right\rangle _{H}=\delta_{mn},\quad\left\langle
\xi_{m},\xi_{n}\right\rangle _{L^{2}\left(  \Omega\right)  }=\delta_{mn}%
\]
because $\left\langle \xi_{m},\xi_{n}\right\rangle _{L^{2}\left(
\Omega\right)  }=E\left(  \xi_{m}\xi_{n}\right)  $. See \cite{Loeve-1963-PT}
for further details.

Karhunen-Lo\`{e}ve explansion is used in practice in many ways. Here are some:

\begin{enumerate}
\item \emph{Data analysis}: Estimate the covariance from data (realizations of
$U$) by sample covariance, and compute the eigenvalues and eigenvectors to
represent the random element $U$. This is known as \emph{Principal Component
Analysis (PCA), Karhunen-Lo\`{e}ve transform (KLT), or Proper Orthogonal
Decomposition (POD)}. The components with several largest eigenvalues
$\lambda_{n}$ are responsible for most of the variance in the random element
$U$. (The practical computation is done by SVD of the data minus sample mean,
which is much less expensive and less prone to numerical errors than computing
the sample covariance first.)

\item Prescribe the covariance so that it has suitable eigenvectors, and use
the Karhunen-Lo\`{e}ve expansion to \emph{generate the random element} $U$
(rather, the first few terms to generate a version of $U$ in finite
dimension). This is the method of random Fourier series in Section
\ref{sec:laplace}, where the covariance was chosen so that it has
trigonometric functions as its eigenvectors.

\item Use the first few terms Karhunen-Lo\`{e}ve expansion of the to build
random coefficients and assumed form of the solution (called trial space in
variational methods) of \emph{stochastic partial differential equations}. The
solution is then found numerically as a \emph{deterministic function of a
small number random variables} $\xi_{1},\ldots,\xi_{n}$. This is the
foundation of methods such as \emph{stochastic Galerkin}
\cite{Babuska-2002-EES}, \emph{stochastic collocation} \cite{Ganis-2008-SCM}
and \emph{polynomial chaos} \cite{Xiu-2010-NMS}.
\end{enumerate}

\section{Inequalities}

\subsection{Cauchy-Schwarz inequality}

In the Hilbert space $H$, the Cauchy-Schwarz inequality reads%
\begin{equation}
\left\vert \left\langle X,Y\right\rangle \right\vert \leq\left\vert
X\right\vert \left\vert Y\right\vert . \label{eq:Cauchy-Schwarz-H}%
\end{equation}

For random variables $X$ and $Y$, the Cauchy-Schwarz inequality in
$L^{2}\left(  \Omega,\mathbb{R}\right)  $,%
\[
\left\vert \int_{\Omega}XYd\omega\right\vert ^{2}\leq\int_{\Omega}\left\vert
X\right\vert ^{2}d\omega\int_{\Omega}\left\vert Y\right\vert ^{2}d\omega
\]
becomes%
\begin{equation}
\left\vert E\left(  XY\right)  \right\vert \leq\left\Vert X\right\Vert
_{2}\left\Vert Y\right\Vert _{2}. \label{eq:cauchy-random-variables}%
\end{equation}
This is a special case (with $H=\mathbb{R}$) of Cauchy-Schwarz inequality in
the Hilbert space $L^{2}\left(  \Omega,H\right)  $ below.

\begin{lemma}
If $X,Y\in L^{2}\left(  \Omega,H\right)  $, then
\begin{align}
\left\vert E\left(  \left\langle X,Y\right\rangle \right)  \right\vert  &
\leq\left\Vert X\right\Vert _{2}\left\Vert Y\right\Vert _{2}%
,\label{eq:cauchy-random-elements}\\
\left\Vert X\otimes Y\right\Vert _{1}  &  \leq\left\Vert X\right\Vert
_{2}\left\Vert Y\right\Vert _{2}, \label{eq:mean-tensor-bound}%
\end{align}
and%
\begin{equation}
\left\Vert X\right\Vert _{1}\leq\left\Vert X\right\Vert _{2}.
\label{eq:L12-comparison}%
\end{equation}

\end{lemma}

\begin{proof}
Inequality (\ref{eq:cauchy-random-elements}) is Cauchy-Schwarz inequality in
the Hilbert space $L^{2}\left(  \Omega,H\right)  $. We have used it few times
already, this is for completeness only. Inequality (\ref{eq:mean-tensor-bound}%
) follows from (\ref{eq:tensor-norm}),%
\[
\left\vert E\left(  X\otimes Y\right)  \right\vert \leq E\left(  \left\vert
X\otimes Y\right\vert \right)  =E\left(  \left\vert X\right\vert \left\vert
Y\right\vert \right)  \leq\left\Vert X\right\Vert _{2},
\]
using Cauchy-Schwarz inequality for the random variables $\left\vert
X\right\vert $ and $\left\vert Y\right\vert $. From
(\ref{eq:cauchy-random-elements}),%
\[
E\left(  \left\vert X\right\vert \right)  =E\left(  1\left\vert X\right\vert
\right)  \leq E\left(  1\right)  ^{1/2}E\left(  \left\vert X\right\vert
^{2}\right)  ^{1/2}=\left\Vert X\right\Vert _{2},
\]
which yields (\ref{eq:L12-comparison}).
\end{proof}

\begin{lemma}
For any $U$, $V\in L^{p}\left(  \Omega,H\right)  $,%
\begin{align}
\left\Vert U\otimes V\right\Vert _{p}  &  =\left\Vert \left\vert U\right\vert
\left\vert V\right\vert \right\Vert _{p}=\left(  E\left(  \left\vert
U\right\vert ^{p}\left\vert V\right\vert ^{p}\right)  \right)  ^{1/p}%
\label{eq:tensor-high-cauchy}\\
&  \leq E\left(  \left\vert U\right\vert ^{2p}\right)  ^{1/2p}E\left(
\left\vert U\right\vert ^{2p}\right)  ^{1/2p}=\left\Vert U\right\Vert
_{2p}\left\Vert U\right\Vert _{2p}.\nonumber
\end{align}

\end{lemma}

\begin{proof}
The proof follows from Cauchy-Schwarz inequality for the random variables
$\left\vert U\right\vert ^{p}$ and $\left\vert V\right\vert ^{p}$.
\end{proof}

\subsection{H\"{o}lder's inequality}

Let $1<p,q<\infty$, and $1/p+1/q=1$. For random variables $X$ and $Y$,
H\"{o}lder's inequality for integrals,%
\[
\left\vert \int_{\Omega}XYd\omega\right\vert \leq\left(  \int_{\Omega
}\left\vert X\right\vert ^{p}d\omega\right)  ^{1/p}\left(  \int_{\Omega
}\left\vert Y\right\vert ^{p}d\omega\right)  ^{1/q},
\]
becomes%
\begin{equation}
\left\vert E\left(  XY\right)  \right\vert \leq E\left(  \left\vert
X\right\vert ^{p}\right)  ^{1/p}E\left(  \left\vert Y\right\vert ^{q}\right)
^{1/q}. \label{eq:Holder-scalar}%
\end{equation}
Estimates for $H$-valued random elements follow.

\begin{lemma}
\label{lem:holder}If $1<p,q<\infty$, $1/p+1/q=1$, and $X\in L^{p}\left(
\Omega,H\right)  $, $Y\in L^{q}\left(  \Omega,H\right)  $, then%
\begin{align}
E\left(  \left\vert X\right\vert \left\vert Y\right\vert \right)   &
\leq\left\Vert X\right\Vert _{p}\left\Vert Y\right\Vert _{q}%
,\label{eq:holder-norm}\\
\left\Vert \left\langle X,Y\right\rangle \right\Vert _{1}  &  \leq\left\Vert
X\right\Vert _{p}\left\Vert Y\right\Vert _{q}, \label{eq:holder-inner-product}%
\\
\left\Vert X\otimes Y\right\Vert _{1}  &  \leq\left\Vert X\right\Vert
_{p}\left\Vert Y\right\Vert _{q}, \label{eq:holder-tensor}%
\end{align}

\end{lemma}

\begin{proof}
Inequality (\ref{eq:holder-norm}) is H\"{o}lder's inequality
(\ref{eq:Holder-scalar}) applied to the random variables $\left\vert
X\right\vert $ and $\left\vert Y\right\vert $. (\ref{eq:holder-inner-product})
follows from (\ref{eq:holder-norm}) by Cauchy-Schwarz inequality
(\ref{eq:Cauchy-Schwarz-H}) in $H$,%
\[
E\left(  \left\vert \left\langle X,Y\right\rangle \right\vert \right)  \leq
E\left(  \left\vert X\right\vert \left\vert Y\right\vert \right)  .
\]
Finally, (\ref{eq:tensor-norm}),%
\[
E\left(  \left\vert X\otimes Y\right\vert \right)  \leq E\left(  \left\vert
X\right\vert \left\vert Y\right\vert \right)
\]
and (\ref{eq:holder-norm}) give (\ref{eq:holder-tensor}).
\end{proof}

Lemma \ref{lem:holder} provides bounds on the $L^{1}\left(  \Omega,H\right)  $
norm. Bounds on higher norms are similar.

\begin{lemma}
\label{lem:holder-high}If $1<p,q<\infty$, $1/p+1/q=1$, $s\geq1$, and $X\in
L^{ps}\left(  \Omega,H\right)  $, $Y\in L^{qs}\left(  \Omega,H\right)  $, then%
\begin{align}
\left\Vert \left\vert X\right\vert \left\vert Y\right\vert \right\Vert _{s}
&  \leq\left\Vert X\right\Vert _{ps}\left\Vert Y\right\Vert _{qs}%
,\label{eq:holder-norm-high}\\
\left\Vert \left\langle X,Y\right\rangle \right\Vert _{s}  &  \leq\left\Vert
X\right\Vert _{ps}\left\Vert Y\right\Vert _{qs},\label{eq:holder-inner-high}\\
\left\Vert X\otimes Y\right\Vert _{s}  &  \leq\left\Vert X\right\Vert
_{ps}\left\Vert Y\right\Vert _{qs}. \label{eq:holder-tensor-high}%
\end{align}

\end{lemma}

\begin{proof}
By H\"{o}lder inequality for the random variables $\left\vert X\right\vert
^{s}$ and $\left\vert Y\right\vert ^{s}$,%
\[
E\left(  \left\vert X\right\vert ^{s}\left\vert Y\right\vert ^{s}\right)  \leq
E\left(  \left\vert X\right\vert ^{sp}\right)  ^{1/p}E\left(  \left\vert
Y\right\vert ^{sq}\right)  ^{1/q}.
\]
The proofs of (\ref{eq:holder-inner-high}) and (\ref{eq:holder-tensor-high})
carry over from the proof of Lemma \ref{lem:holder}.
\end{proof}

We now generalize (\ref{eq:L12-comparison}) to arbitrary $L^{p}$ spaces.

\begin{lemma}
If $1\leq s<t$ and $X\in L^{t}\left(  \Omega,H\right)  $, then%
\begin{equation}
\left\Vert X\right\Vert _{s}\leq\left\Vert X\right\Vert _{t}.
\label{eq:Lp-comparison}%
\end{equation}

\end{lemma}

\begin{proof}
Choose $p>1$ so that $t=ps$, apply H\"{o}lder inequality for the random
variables $\left\vert X\right\vert ^{s}$ and $1$,
\[
E\left(  \left\vert X\right\vert ^{s}1^{s}\right)  \leq E\left(  \left\vert
X\right\vert ^{sp}\right)  ^{1/p}E\left(  1^{sq}\right)  ^{1/q}%
\]
and note that $\left\Vert 1\right\Vert _{qs}=1$ since $1$ is constant.
\end{proof}

\subsection{Chebychev's inequality}

The generalized Chebychev's inequality for a nonnegative random variable $U$
states that if $f$ is a measurable, nonnegative, nondecreasing function on
$\left(  0,+\infty\right)  $, then%
\[
\Pr\left(  f\left(  U\right)  >t\right)  \leq\frac{1}{f\left(  t\right)
}E\left(  f\left(  U\right)  \right)  .
\]
For $X\in L^{p}$, the choices $f\left(  t\right)  =t$ and $U=\left\vert
X\right\vert ^{p}$ yield%
\[
\Pr\left(  \left\vert X\right\vert ^{p}>t\right)  \leq\frac{1}{t}E\left(
\left\vert X\right\vert ^{p}\right)  ,\quad\forall t>0,
\]
or, equivalently, by the substitution $t=\theta^{p}$,%
\begin{equation}
\Pr\left(  \left\vert X\right\vert >\theta\right)  \leq\left(  \frac
{\left\Vert X\right\Vert _{p}}{\theta}\right)  ^{p},\quad\forall\theta>0,\quad
p\geq1. \label{eq:chebychev-p}%
\end{equation}

\subsection{Marcinkiewicz-Zygmund inequality}

\begin{theorem}
[{\cite{Marcinkiewicz-1937-FI}, \cite[p. 367]{Chow-1988-PT}}]If $1\leq
p<+\infty$ and $X_{k}$, $k=1,\ldots,n$, are independent random variables such
that $E\left(  X_{k}\right)  =0$ and $E\left(  \left\vert X_{k}\right\vert
^{p}\right)  <+\infty$, then%
\begin{equation}
A_{p}E\left(  \left(  \sum_{k=1}^{n}\left\vert X_{k}\right\vert ^{2}\right)
^{p/2}\right)  \leq E\left(  \left\vert \sum_{k=1}^{n}X_{k}\right\vert
^{p}\right)  \leq B_{p}E\left(  \left(  \sum_{k=1}^{n}\left\vert
X_{k}\right\vert ^{2}\right)  ^{p/2}\right)  \label{eq:M-Z-inequality}%
\end{equation}
where $A_{p}$ and $B_{p}$ are positive constants, which depend only on $p$.
\end{theorem}

Marcinkiewicz-Zygmund inequality in terms of norms becomes%
\[
A_{p}\left\Vert \left(  \sum_{k=1}^{n}\left\vert X_{k}\right\vert ^{2}\right)
^{1/2}\right\Vert _{p}\leq\left\Vert \sum_{k=1}^{n}X_{k}\right\Vert _{p}\leq
B_{p}\left\Vert \left(  \sum_{k=1}^{n}\left\vert X_{k}\right\vert ^{2}\right)
^{1/2}\right\Vert _{p},
\]
still for random variables only. For $p=2$, the inequality in a Hilbert space
holds with $A_{2}=B_{2}=1,$%
\[
E\left(  \left\vert \sum_{k=1}^{n}X_{k}\right\vert ^{2}\right)  =E\left(
\sum_{i=1}^{n}\sum_{j=1}^{n}\left\langle X_{i},X_{j}\right\rangle \right)
=E\left(  \sum_{i=1}^{n}\left\vert X_{i}\right\vert ^{2}\right)  ,
\]
since $E\left(  X_{k}\right)  =0$, which is the well-known property%
\[
\operatorname{Var}\left(  \sum_{k=1}^{n}X_{k}\right)  =\sum_{i=1}%
^{n}\operatorname{Var}\left(  X_{i}\right)  .
\]

However, an extension from $\mathbb{R}$ to Banach spaces relies on the theory
of geometry of Banach spaces. The Marcinkiewicz-Zygmund inequality is not
valid in Banach spaces in general, and in fact it defines a special type of
Banach spaces. We will concentrate on the upper bound, cf. Corollary
\ref{cor:MZB} below.

\begin{definition}
\label{def:type-p}Let $1\leq p\leq2$. A separable Banach space $U$ is said to
be of Rademacher type $p$ if there exists $C$ such that for every $n$ and for
all $x_{1},\ldots,x_{n}\in U$ and for every sequence $r_{i}$ of real
independent random variables with $\Pr\left(  r_{i}=\pm1/2\right)  =1/2$,%
\begin{equation}
E\left(  \left\vert \sum_{i=1}^{n}r_{i}x_{i}\right\vert \right)  \leq C\left(
\sum_{i=1}^{n}\left\vert x_{i}\right\vert ^{p}\right)  ^{1/p}.
\label{eq:def-type-p}%
\end{equation}

\end{definition}

\begin{remark}
[{\cite[p. 158]{Araujo-1980-CLT}}]Rademacher type is defined for $p\leq2$ only
because the only Banach space that satifies (\ref{eq:def-type-p}) with $p>2$
is the trivial space $\left\{  0\right\}  $.
\end{remark}

\begin{example}
Every separable Banach space is of type $1$. By the triangle inequality,%
\[
E\left(  \left\vert \sum_{i=1}^{n}r_{i}x_{i}\right\vert \right)  \leq E\left(
\sum_{i=1}^{n}\left\vert r_{i}\right\vert \left\vert x_{i}\right\vert \right)
=\frac{1}{2}\sum_{i=1}^{n}\left\vert x_{i}\right\vert .
\]

\end{example}

\begin{example}
The space $l^{1}$ is not of type $p$ for any $p>1$. Consider $x_{i}=\left(
0,\ldots,0,1,0,\ldots\right)  $ with the $1$ in the $i$-the place. Then,%
\[
E\left(  \left\vert \sum_{i=1}^{n}r_{i}x_{i}\right\vert \right)  =E\left(
\sum_{i=1}^{n}\left\vert r_{i}\right\vert \right)  =\frac{n}{2},\text{but
}\left(  \sum_{i=1}^{n}\left\vert x_{i}\right\vert ^{p}\right)  ^{1/p}%
=n^{1/p}.
\]

\end{example}

\begin{example}
The space $\operatorname*{Tr}H$ of trace class operators on a Hilbert space
$H$ is is not of type $p$ for any $p>1$, because it contains a copy of $l^{1}%
$, namely, diagonal matrices.
\end{example}

\begin{example}
[{\cite[p. 159]{Araujo-1980-CLT}}]A separable Hilbert space $H$ is of
Rademacher type 2. Let $x_{1},\ldots,x_{n}\in H$. Then
\begin{align*}
E\left(  \left\vert \sum_{i=1}^{n}r_{i}x_{i}\right\vert \right)  ^{2}  &  \leq
E\left(  \left\vert \sum_{i=1}^{n}r_{i}x_{i}\right\vert ^{2}\right)  =E\left(
\sum_{i=1}^{n}\sum_{j=1}^{n}r_{i}r_{j}\left(  x_{i},x_{j}\right)  \right) \\
&  =E\left(  \sum_{i=1}^{n}r_{i}^{2}\left\vert x_{i}\right\vert ^{2}\right)
=\frac{1}{4}\left(  \sum_{i=1}^{n}\left\vert x_{i}\right\vert ^{2}\right)  ,
\end{align*}
by the independence of $r_{i}$.
\end{example}

\begin{remark}
It follows that for any $\sigma$-finite measure, $L^{2}$ is type 2 since it is
a separable Hilbert space. In addition, $L^{p}$, $p>2$, is also type 2
\cite[p. 160]{Araujo-1980-CLT}, and $L^{p}$, $1<p<2$ is type $p$
(\cite[Exercise 1, p. 202]{Araujo-1980-CLT},
\cite{Hoffmann-Jorgensen-1974-SIB}).
\end{remark}

\begin{proposition}
[{\cite[page 120, Proposition 2.1]{Woyczynski-1980-MLL}, \cite[Theorem 7.2
(2)]{Araujo-1980-CLT} for $q=1$}]\label{prop:type-p}Let $1\leq p\leq2$ and
$q\geq1$. Banach space $U$ is of Rademacher type $p$ if and only if there
exists constant $C$ such that for every $n$ and for any sequence $X_{i}$ of
independent random elements in $U$ with $E\left(  X_{i}\right)  =0,$%
\begin{equation}
E\left(  \left\vert \sum_{i=1}^{n}X_{i}\right\vert ^{q}\right)  \leq CE\left(
\sum_{i=1}^{n}\left\vert X_{i}\right\vert ^{p}\right)  ^{q/p} \label{eq:MZ-q}%
\end{equation}

\end{proposition}

\begin{corollary}
\label{cor:MZB}The upper bound in Marcinkiewicz-Zygmund inequality
(\ref{eq:M-Z-inequality}) holds in a separable Banach space $U$ if and only if
$U$ is of Rademacher type $2$.
\end{corollary}

\begin{corollary}
[Marcinkiewicz-Zygmund inequality in Hilbert space]\label{cor:MZH} If $1\leq
q<+\infty$ and $X_{k}$, $k=1,\ldots,n$, are independent random elements in a
separable Hilbert space $H$ such that $E\left(  X_{k}\right)  =0$ and
$E\left(  \left\vert X_{k}\right\vert ^{q}\right)  <+\infty$, then%
\begin{equation}
E\left(  \left\vert \sum_{i=1}^{n}X_{i}\right\vert ^{q}\right)  \leq
B_{p}E\left(  \sum_{i=1}^{n}\left\vert x_{i}\right\vert ^{2}\right)  ^{q/2},
\label{eq:M-Z-inequality-Hilbert}%
\end{equation}
where $B_{p}$ depends on $p$ only.
\end{corollary}

\begin{proof}
Hilbert space is of Rademacher type 2, so it is enough to set $p=2$ in
(\ref{eq:MZ-q}). Since every separable Hilbert space $H$ is isometric to
$l^{2}$, the constant $B_{p}$ is determined on $l^{2}$ and it does not depend
on the particular space $H$.
\end{proof}

\begin{remark}
Strangely, even if Rademacher type is a property of the Banach space, and so
one would expect a formulation as an inequality that does not involve
probability, no such characterization seems to be known. (The probabilistic
formulation does rely on the norm only, but in a complicated manner: the norm
defines Borel measures and thus random variables.) However, Banach spaces of
Rademacher of type $p$ can be characterized as not containing subspaces
isomorphic to $l^{q}$, $q>p$, in a (rather complicated) approximate sense
\cite{Ledoux-1991-PBS}.
\end{remark}

\section{Properties of covariance}

\subsection{Independent random elements}

We now generalize the well-known property that indepent random variables are
uncorrelated. Note that the variables need to be in $L^{2}$ to guarantee that
the correlation is defined.

\begin{lemma}
\label{lem:uncorrelated}If $X,$ $Y\in L^{2}\left(  \Omega,H\right)  $ are
independent, then $\operatorname*{Cov}\left(  X,Y\right)  =0.$
\end{lemma}

\begin{proof}
Adding constants to $X$ and $Y$ does not change independence, so without loss
of generality assume that $E\left(  X\right)  =E\left(  Y\right)  =0$. Let
$u,v\in H$. From the Cauchy inequality in $H$, $\left\vert \left\langle
u,X\left(  \omega\right)  \right\rangle \right\vert ^{2}\leq\left\vert
u\right\vert \left\vert X\left(  \omega\right)  \right\vert $, it follows that
the random variable $\left\langle u,X\right\rangle \in L^{2}\left(
\Omega,\mathbb{R}\right)  $. Similarly, $\left\langle u,X\right\rangle \in
L^{2}\left(  \Omega,\mathbb{R}\right)  $. Since $\left\langle v,X\right\rangle
$ and $\left\langle v,Y\right\rangle $ are independent,%
\[
E\left(  \left\langle u,X\right\rangle \left\langle v,Y\right\rangle \right)
=E\left(  \left\langle u,X\right\rangle \right)  E\left(  \left\langle
v,Y\right\rangle \right)  =\left\langle u,E\left(  X\right)  \right\rangle
\left\langle v,E\left(  Y\right)  \right\rangle =0,
\]
using the definition of the mean (\ref{eq:def-mean}).
\end{proof}

\subsection{Bounds on covariance}

The following estimate generalizes a well-known inequality for random vectors
in $\mathbb{R}^{n}$. Taking advantage or tensor products and properties of the
norm, we can carry over a straightforward proof for $\mathbb{R}^{n}$.

\begin{lemma}
If $X$, $Y\in L^{2}\left(  \Omega,H\right)  $, then
\[
\left\vert \operatorname*{Cov}\left(  X,Y\right)  \right\vert \leq E\left(
\left\vert X\right\vert \left\vert Y\right\vert \right)  +\left\vert E\left(
X\right)  \right\vert \left\vert E\left(  Y\right)  \right\vert .
\]
In particular,
\begin{equation}
\left\vert \operatorname*{Cov}\left(  X,Y\right)  \right\vert \leq2\left\Vert
X\right\Vert _{2}\left\Vert Y\right\Vert _{2}. \label{eq:cov-bound}%
\end{equation}
and, if, in addition, $E\left(  X\right)  =E\left(  Y\right)  =0$, then%
\begin{equation}
\left\vert \operatorname*{Cov}\left(  X,Y\right)  \right\vert \leq\left\Vert
X\right\Vert _{2}\left\Vert Y\right\Vert _{2}. \label{eq:cov-bound-zero-mean}%
\end{equation}

\end{lemma}

\begin{proof}
From (\ref{eq:covariance2}), the triangle inequality, the property of the
integral $\left\vert E\left(  U\right)  \right\vert \leq E\left(  \left\vert
U\right\vert \right)  $, the equality $\left\vert x\otimes y\right\vert
=\left\vert x\right\vert \left\vert y\right\vert $, and Cauchy-Schwarz
inequality,%
\begin{align*}
\left\vert \operatorname*{Cov}\left(  X,Y\right)  \right\vert  &  =\left\vert
E\left(  X\otimes Y\right)  -E\left(  X\right)  \otimes E\left(  Y\right)
\right\vert \\
&  \leq\left\vert E\left(  X\otimes Y\right)  \right\vert +\left\vert E\left(
X\right)  \otimes E\left(  Y\right)  \right\vert \\
&  \leq E\left(  \left\vert X\otimes Y\right\vert \right)  +\left\vert
E\left(  X\right)  \otimes E\left(  Y\right)  \right\vert \\
&  \leq E\left(  \left\vert X\right\vert \left\vert Y\right\vert \right)
+\left\vert E\left(  X\right)  \right\vert \left\vert E\left(  Y\right)
\right\vert .
\end{align*}
The rest follows from Cauchy-Schwarz inequality, $E\left(  \left\vert
X\right\vert \left\vert Y\right\vert \right)  \leq\left\Vert X\right\Vert
_{2}\left\Vert Y\right\Vert _{2}$, $\left\vert E\left(  X\right)  \right\vert
\leq E\left(  \left\vert X\right\vert 1\right)  \leq\left\Vert X\right\Vert
_{2}$.
\end{proof}

Note that the same proof, using H\"{o}lder's inequality instead of
Cauchy-Schwarz inequality, yields%
\begin{equation}
\left\vert \operatorname*{Cov}\left(  X,Y\right)  \right\vert \leq2\left\Vert
X\right\Vert _{p}\left\Vert Y\right\Vert _{q},\quad p,q>1,\quad\frac{1}%
{p}+\frac{1}{q}=1. \label{eq:cov-bound-Lp}%
\end{equation}

Here is continuity of covariance.

\begin{lemma}
If $X$, $Y\in L^{2}\left(  \Omega,H\right)  $, then%
\begin{equation}
\left\vert \operatorname*{Cov}(X,X)-\operatorname*{Cov}(Y,Y)\right\vert
\leq2\left\Vert X-Y\right\Vert _{2}\left(  \left\Vert Y\right\Vert
_{2}+\left\Vert X\right\Vert _{2}\right)  . \label{eq:cov-cont}%
\end{equation}

\end{lemma}

\begin{proof}
Since $\operatorname*{Cov}$ is bilinear,%
\[
\operatorname*{Cov}(X,X)-\operatorname*{Cov}(Y,Y)=\operatorname*{Cov}%
(X-Y,Y)+\operatorname*{Cov}(X,X-Y),
\]
which gives
\[
\left\vert \operatorname*{Cov}(X,X)-\operatorname*{Cov}(Y,Y)\right\vert
\leq2\left\Vert X-Y\right\Vert _{2}\left\Vert Y\right\Vert _{2}+2\left\Vert
X\right\Vert _{2}\left\Vert X-Y\right\Vert _{2}%
\]
from (\ref{eq:cov-bound}).
\end{proof}

\subsection{Sample mean and covariance}

Given $X_{k}\in H$, $k=1,\ldots,n$, the sample mean is defined by%
\begin{equation}
E_{n}\left(  X_{k}\right)  =\frac{1}{n}\sum_{k=1}^{n}X_{k}. \label{eq:def-En}%
\end{equation}
Given also $Y_{k}$, $k=1,\ldots,n$, the sample covariance $C_{n}\left(
X_{k},Y_{k}\right)  $ is defined as
\[
C_{n}(X_{k},Y_{k})=E_{n}\left(  X_{k}-E_{n}\left(  X_{k}\right)  \right)
\otimes\left(  Y_{k}-E_{n}\left(  Y_{k}\right)  \right)  ,
\]
Again, we write $C_{n}\left(  X_{k}\right)  $ for $C_{n}\left(  X_{k}%
,X_{k}\right)  $. We will use the sample mean notation also with more general
terms, for example%
\[
E_{n}\left(  \left\vert X_{k}\right\vert ^{2}\right)  =\frac{1}{n}\sum
_{k=1}^{n}\left\vert X_{k}\right\vert ^{2}.
\]

We first estimate sample covariance pointwise.

\begin{lemma}
If $X_{k},Y_{k}\in H$, $k=1,\ldots,n$, then%
\begin{align}
C_{n}(X_{k},Y_{k})  &  =E_{n}\left(  X_{k}\otimes Y_{k}\right)  -E_{n}\left(
X_{k}\right)  \otimes E_{n}\left(  Y_{k}\right)  \label{eq:sample-cov-tensor}%
\\
\left\vert C_{n}\left(  X_{k},Y_{k}\right)  \right\vert  &  \leq2E_{n}\left(
\left\vert X_{k}\right\vert ^{2}\right)  ^{1/2}E_{n}\left(  \left\vert
Y_{k}\right\vert ^{2}\right)  ^{1/2},\label{eq:sample-cov-bound-pointwise}\\
\left\vert C_{n}(X_{k})-C_{n}(Y_{k})\right\vert  &  \leq2E_{n}\left(
\left\vert X_{k}-Y_{k}\right\vert ^{2}\right)  ^{1/2}\left(  E_{n}\left(
\left\vert X_{k}\right\vert ^{2}\right)  ^{1/2}+E_{n}\left(  \left\vert
Y_{k}\right\vert ^{2}\right)  ^{1/2}\right)  .
\label{eq:sample-cov-cont-pointwise}%
\end{align}

\end{lemma}

\begin{proof}
These are same as the properties of covariance (\ref{eq:covariance2}),
(\ref{eq:cov-bound}), (\ref{eq:cov-cont}), and the proofs are exactly the
same. Alternatively, one can consider random variables $X$ and $Y$ that attain
the values of $X_{k}$, $Y_{k}$, $k=1,\ldots,n$, respectively, with equal
probability, and apply (\ref{eq:covariance2}), (\ref{eq:cov-bound}), and
(\ref{eq:cov-cont}).
\end{proof}

If $X_{k}$ and $Y_{k}$ are random variables, then the sample mean is
$H$-valued random element, while the sample covariance is $\left[  H\right]
$-valued random element. We estimate these random elements in $L^{p}$ norms.

\begin{lemma}
\label{lem:bound-sample-covariance}For all $p\geq1$, if $X_{k}\in L^{p}\left(
\Omega,H\right)  $, $k=1,\ldots,n$, are identically distributed and $Y_{k}\in
L^{p}\left(  \Omega,H\right)  $, $k=1,\ldots,n$, are identically distributed,
then
\begin{align}
\left\Vert E_{n}\left(  X_{k}\right)  \right\Vert _{p}  &  \leq\left\Vert
X_{1}\right\Vert _{p},\label{eq:bound-sample-mean}\\
\left\Vert C_{n}\left(  X_{k},Y_{k}\right)  \right\Vert _{p}  &
\leq2\left\Vert X_{1}\right\Vert _{2p}\left\Vert Y_{1}\right\Vert _{2p}.
\label{eq:bound-sample-covariance}%
\end{align}

\end{lemma}

\begin{proof}
Inequality (\ref{eq:bound-sample-mean}) follows immediately from the triangle
inequality,%
\begin{equation}
\left\Vert E_{n}\left(  X_{k}\right)  \right\Vert _{p}=\left\Vert \frac{1}%
{n}\sum_{k=1}^{n}X_{k}\right\Vert _{p}\leq\frac{1}{n}\sum_{k=1}^{n}\left\Vert
X_{k}\right\Vert _{p}=\left\Vert X_{k}\right\Vert _{p}. \label{eq:Lp-mean}%
\end{equation}
Then, by the same argument as in (\ref{eq:Lp-mean}) and from
(\ref{eq:tensor-high-cauchy}),%
\[
\left\Vert E_{n}\left(  X_{k}\otimes Y_{k}\right)  \right\Vert _{p}=\left\Vert
\frac{1}{n}\sum_{k=1}^{n}X_{k}\otimes Y_{k}\right\Vert _{p}\leq\left\Vert
X_{k}\otimes Y_{k}\right\Vert _{p}\leq\left\Vert X_{k}\right\Vert
_{2p}\left\Vert Y_{k}\right\Vert _{2p}%
\]
and by using (\ref{eq:tensor-high-cauchy}) again and then (\ref{eq:Lp-mean}),%
\[
\left\Vert E_{n}\left(  X_{k}\right)  \otimes E_{n}\left(  Y_{k}\right)
\right\Vert _{p}\leq\left\Vert E\left(  X_{1}\right)  \right\Vert
_{2p}\left\Vert E\left(  Y_{1}\right)  \right\Vert _{2p}\leq\left\Vert
X_{1}\right\Vert _{2p}\left\Vert Y_{1}\right\Vert _{2p},
\]
which gives%
\begin{align*}
\left\Vert C_{n}\left(  X_{k},Y_{k}\right)  \right\Vert _{p}  &
\leq\left\Vert E_{n}\left(  X_{k}\otimes Y_{k}\right)  \right\Vert
_{p}+\left\Vert E_{n}\left(  X_{k}\right)  \otimes E_{n}\left(  Y_{k}\right)
\right\Vert _{p}\\
&  \leq\left\Vert X_{1}\right\Vert _{2p}\left\Vert Y_{1}\right\Vert
_{2p}+\left\Vert X_{1}\right\Vert _{2p}\left\Vert Y_{1}\right\Vert _{2p},
\end{align*}
and completes the proof of (\ref{eq:bound-sample-covariance}).
\end{proof}

\begin{lemma}
\label{lem:sample-cov-norm-cont}If $X_{k}\in L^{4}\left(  \Omega,H\right)  $,
$k=1,\ldots,n$, are identically distributed and $Y_{k}\in L^{4}\left(
\Omega,H\right)  $, $k=1,\ldots,n$, are identically distributed, then
\begin{equation}
\left\Vert C_{n}(X_{k})-C_{n}(Y_{k})\right\Vert _{2}\leq\sqrt{8}\max
_{k}\left\Vert X_{k}-Y_{k}\right\Vert _{4}\sqrt{\left\Vert X_{k}\right\Vert
_{4}^{2}+\left\Vert Y_{k}\right\Vert _{4}^{2}} \label{eq:sample-cov-norm-cont}%
\end{equation}

\end{lemma}

\begin{proof}
From (\ref{eq:sample-cov-cont-pointwise}) and the inequality $\left(
a+b\right)  ^{2}\leq2\left(  a^{2}+b^{2}\right)  ,$ we have%
\[
\left\vert C_{n}(X_{k})-C_{n}(Y_{k})\right\vert ^{2}\leq8E_{n}\left(
\left\vert X_{k}-Y_{k}\right\vert ^{2}\right)  \left(  E_{n}\left(  \left\vert
X_{k}\right\vert ^{2}\right)  +E_{n}\left(  \left\vert Y_{k}\right\vert
^{2}\right)  \right)  .
\]
Integrating and using the Cauchy-Schwarz inequality for random variables for
each term yields%
\begin{align}
E\left(  \left\vert C_{n}(X_{k})-C_{n}(Y_{k})\right\vert ^{2}\right)   &
\leq8E\left(  E_{n}\left(  \left\vert X_{k}-Y_{k}\right\vert ^{2}\right)
\left[  E_{n}\left(  \left\vert X_{k}\right\vert ^{2}\right)  +E_{n}\left(
\left\vert Y_{k}\right\vert ^{2}\right)  \right]  \right)
\label{eq:mean-cov-diff-s}\\
&  \leq8\left\Vert E_{n}\left(  \left\vert X_{k}-Y_{k}\right\vert ^{2}\right)
\right\Vert _{2}\left[  \left\Vert E_{n}\left(  \left\vert X_{k}\right\vert
^{2}\right)  \right\Vert _{2}+\left\Vert E_{n}\left(  \left\vert
X_{k}\right\vert ^{2}\right)  \right\Vert _{2}\right]  ,\nonumber
\end{align}
where by Cauchy-Schwarz inequality,%
\begin{equation}
\left\Vert E_{n}\left(  \left\vert X_{k}\right\vert ^{2}\right)  \right\Vert
_{2}\leq\max_{k=1:n}\left\Vert X_{k}\right\Vert _{4}^{2}
\label{eq:est-mean-squares}%
\end{equation}
since%
\begin{align*}
\left\Vert E_{n}\left(  \left\vert X_{k}\right\vert ^{2}\right)  \right\Vert
_{2}^{2}  &  =E\left(  \left(  \frac{1}{n}\sum_{k=1}^{n}\left\vert
X_{k}\right\vert ^{2}\right)  ^{2}\right) \\
&  =\frac{1}{n^{2}}\sum_{k=1}^{n}\sum_{\ell=1}^{n}E\left(  \left\vert
X_{k}\right\vert ^{2}\left\vert X_{\ell}\right\vert ^{2}\right) \\
&  \leq\frac{1}{n^{2}}\sum_{k=1}^{n}\sum_{\ell=1}^{n}E\left(  \left\vert
X_{k}\right\vert ^{4}\right)  ^{1/2}E\left(  \left\vert X_{\ell}\right\vert
^{4}\right)  ^{1/2}\\
&  \leq\max_{k=1:n}E\left(  \left\vert X_{k}\right\vert ^{4}\right)
=\max_{k=1:n}\left\Vert X_{k}\right\Vert _{4}^{4},
\end{align*}
and similarly for the other terms. Now taking the square root of
(\ref{eq:mean-cov-diff-s})\ and using (\ref{eq:est-mean-squares})\ three times
gives the desired estimate (\ref{eq:sample-cov-norm-cont}).
\end{proof}

\section{Laws of large numbers}

\subsection{$L^{2}$ and weak law of large numbers}

The following theorem generalizes the weak law of large numbers to the Hilbert
space setting. The usual proof caries over.

\begin{lemma}
\label{lem:WLLN}Let $X_{k}\in L^{2}\left(  \Omega,H\right)  $ be i.i.d. Then
\begin{equation}
\left\Vert E_{n}\left(  X_{k}\right)  -E\left(  X_{1}\right)  \right\Vert
_{2}\leq\frac{1}{\sqrt{n}}\left\Vert X_{1}-E\left(  X_{1}\right)  \right\Vert
_{2}\leq\frac{2}{\sqrt{n}}\left\Vert X_{1}\right\Vert _{2} \label{eq:L2-LLN}%
\end{equation}
and $E_{n}\left(  X_{k}\right)  \Rightarrow E\left(  X_{1}\right)  $ as
$n\rightarrow\infty$.
\end{lemma}

\begin{proof}
First, without loss of generality, let $E\left(  X_{1}\right)  =0$. Since
$X_{k}$ are independent, they are uncorrelated (Lemma \ref{lem:uncorrelated}),
and
\begin{align*}
\left\Vert E_{n}\left(  X_{k}\right)  \right\Vert _{2}^{2}  &  =E\left(
\left\vert \left(  \frac{1}{N}\sum_{k=1}^{n}X_{k}\right)  \right\vert
^{2}\right)  =\frac{1}{n^{2}}\sum_{k=1}^{n}\sum_{\ell=1}^{n}E\left(
\left\langle X_{k},X_{\ell}\right\rangle \right) \\
&  =\frac{1}{n}E\left(  \left\vert X_{1}\right\vert ^{2}\right)  =\frac{1}%
{n}\left\Vert X_{1}\right\Vert _{2}^{2}.
\end{align*}
For general $E\left(  X_{1}\right)  $, from $E_{n}\left(  X_{k}-E\left(
X_{1}\right)  \right)  =E_{n}\left(  X_{k}\right)  -E\left(  X_{1}\right)  $
and from the triangle inequality,
\[
\left\Vert E_{n}\left(  X_{k}\right)  -E\left(  X_{1}\right)  \right\Vert
_{2}=\frac{1}{\sqrt{n}}\left\Vert X_{1}-E\left(  X_{1}\right)  \right\Vert
_{2}\leq\frac{1}{\sqrt{n}}\left(  \left\Vert X_{1}\right\Vert _{2}+\left\Vert
E\left(  X_{1}\right)  \right\Vert _{2}\right)  .
\]
Since $E\left(  X_{1}\right)  $ is a constant and by properties of the
integral and monotonicity $L^{p}$ norms,%
\[
\left\Vert E\left(  X_{1}\right)  \right\Vert _{2}=\left\vert E\left(
X_{1}\right)  \right\vert \leq E\left(  \left\vert X_{1}\right\vert \right)
=\left\Vert X_{1}\right\Vert _{1}\leq\left\Vert X_{1}\right\Vert _{2},
\]
which gives (\ref{eq:L2-LLN}). Now by Chebyschev's inequality for the random
variable $\left\vert E_{n}\left(  X_{k}\right)  -E\left(  X_{1}\right)
\right\vert $, for any $\varepsilon>0$,%
\[
\Pr\left(  \left\vert E_{n}\left(  X_{k}\right)  -E\left(  X_{1}\right)
\right\vert >\varepsilon\right)  \leq\left(  \frac{\left\Vert E_{n}\left(
X_{k}\right)  -E\left(  X_{1}\right)  \right\Vert _{2}}{\varepsilon}\right)
^{2}=\frac{4\left\Vert X_{1}\right\Vert _{2}^{2}}{n\varepsilon^{2}}%
\rightarrow0,
\]
thus $\ E_{n}\left(  X_{k}\right)  \rightarrow E\left(  X_{1}\right)  $ in
probability. The result now follows from the fact that for random elements in
a metric space, convergence in probability to a constant implies convergence
in distribution \cite[Lemma 4.7]{Kallenberg-2002-FMP}.
\end{proof}

\subsection{$L^{p}$ law of large numbers}

Marcinkiewicz-Zygmund inequality allows to prove a variant of the weak law of
large numbers in stronger $L_{p}$ norms. The following statement is
essentially the same as \cite[Corollary 2, page 368]{Chow-1988-PT}, just with
more detail, and using the Marzinkiewicz-Zygmund inequality in Hilbert spaces.

\begin{lemma}
\label{lem:Lp-large-numbers} Let $H$ be a Hilbert space, $X_{i}\in
L^{p}\left(  \Omega,H\right)  $ be i.i.d., $p\geq2$, and $E_{n}\left(
X_{k}\right)  =\frac{1}{n}\sum_{k=1}^{n}X_{k}$. Then,%
\begin{equation}
\left\Vert E_{n}\left(  X_{k}\right)  -E\left(  X_{1}\right)  \right\Vert
_{p}\leq\frac{C_{p}}{\sqrt{n}}\left\Vert X_{1}-E\left(  X_{1}\right)
\right\Vert _{p}\leq\frac{2C_{p}}{\sqrt{n}}\left\Vert X_{1}\right\Vert _{p},
\label{eq:Lp-LLN}%
\end{equation}
where $C_{p}$ depends on $p$ only.
\end{lemma}

\begin{proof}
If $p=2$, the statement becomes (\ref{eq:L2-LLN}). Let $p>2$, and, without
loss of generality, assume first that $E\left(  X_{1}\right)  =0$. By
H\"{o}lder's inequality,%
\begin{equation}
\sum_{k=1}^{n}\left\vert X_{k}\right\vert ^{2}=\sum_{k=1}^{n}1\left(
\left\vert X_{k}\right\vert ^{2}\right)  \leq n^{\left(  p-2\right)
/p}\left(  \sum_{k=1}^{n}\left(  \left\vert X_{k}\right\vert ^{2}\right)
^{p/2}\right)  ^{2/p}=n^{\left(  p-2\right)  /p}\left(  \sum_{k=1}%
^{n}\left\vert X_{k}\right\vert ^{p}\right)  ^{2/p}, \label{eq:MZH}%
\end{equation}
thus, using Marcinkiewicz-Zygmund inequality (\ref{eq:M-Z-inequality-Hilbert}%
),%
\begin{align*}
E\left(  \left\vert \sum_{k=1}^{n}X_{k}\right\vert ^{p}\right)   &  \leq
B_{p}E\left(  \left(  \sum_{k=1}^{n}\left\vert X_{k}\right\vert ^{2}\right)
^{p/2}\right) \\
&  \leq B_{p}n^{p/2-1}E\left(  \sum_{k=1}^{n}\left\vert X_{k}\right\vert
^{p}\right) \\
&  =B_{p}n^{p/2}E\left(  \left\vert X_{1}\right\vert ^{p}\right)
\end{align*}
because $\frac{p}{2}\frac{p-2}{p}=\frac{p}{2}-1$ and $X_{k}$ are identically
distributed. Consequently,%
\[
\left\Vert \sum_{k=1}^{n}X_{k}\right\Vert _{p}\leq B_{p}^{1/p}n^{1/2}%
\left\Vert X_{1}\right\Vert _{p},
\]
and the first inequality in (\ref{eq:Lp-LLN}) follows. The rest follows from
the triangle inequality.
\end{proof}

\subsection{Convergence of sample covariance}

Recall that from (\ref{eq:sample-cov-tensor}),%
\[
C_{n}(X_{k},Y_{k})=E_{n}\left(  X_{k}\otimes Y_{k}\right)  -E_{n}\left(
X_{k}\right)  \otimes E_{n}\left(  Y_{k}\right)  ,
\]
while%
\[
\operatorname*{Cov}\left(  X_{1},Y_{1}\right)  =E\left(  X_{1}\otimes
Y_{1}\right)  -E\left(  X_{1}\right)  \otimes E\left(  Y_{1}\right)  .
\]
We need convergence of $C_{n}(X_{k},Y_{k})$ to $\operatorname*{Cov}%
(X_{1},Y_{1})$ as $n\rightarrow\infty$, in some sense, i.e., the laws of large
numbers for the sample covariance.

Since covariance is in the Hilbert-Schmidt space $HS\left(  H\right)  $, which
is a Hilbert space, we can use laws of large numbers in the space of
Hilbert-Schmidt operators.

\begin{lemma}
\label{lem:LpLLN-sample-cov}If $p\geq2$, $H$ is Hilbert space, and $X_{k}\in
L^{2p}\left(  \Omega,H\right)  $ are i.i.d. Then
\end{lemma}

\begin{equation}
\left(  E\left\vert C_{n}\left(  X_{k}\right)  -\operatorname*{Cov}\left(
X_{1}\right)  \right\vert _{HS\left(  H\right)  }^{p}\right)  ^{1/p}%
\leq\left(  \frac{2C_{p}}{\sqrt{n}}+\frac{4C_{2p}^{2}}{n}\right)  \left\Vert
X_{1}\right\Vert _{2p}^{2}. \label{eq:LpLLN-sample-cov}%
\end{equation}
where $C_{r}$ is a constant which depends on $r$ only; in particular,
$C_{2}=1.$

\begin{proof}
Without loss of generality, let $E\left(  X_{1}\right)  =0$. Then%
\begin{equation}
C_{n}\left(  X_{k}\otimes Y_{k}\right)  -\operatorname*{Cov}\left(
X_{1}\otimes Y_{1}\right)  =\left(  E_{n}\left(  X_{k}\otimes Y_{k}\right)
-E\left(  X_{1}\otimes Y_{1}\right)  \right)  +E_{n}\left(  X_{k}\right)
\otimes E_{n}\left(  Y_{k}\right)  . \label{eq:cov-err}%
\end{equation}
The first term in (\ref{eq:cov-err}) is estimated by%
\begin{align*}
\left(  E\left\vert E_{n}\left(  X_{k}\otimes X_{k}\right)  -E\left(
X_{1}\otimes X_{1}\right)  \right\vert _{HS\left(  H\right)  }^{p}\right)
^{1/p}  &  \leq\frac{2C_{p}}{\sqrt{n}}\left(  E\left\vert X_{1}\otimes
X_{1}\right\vert _{HS\left(  H\right)  }^{p}\right)  ^{1/p}\\
&  =\frac{2C_{p}}{\sqrt{n}}\left(  E\left\vert X_{1}\right\vert _{H}%
^{2p}\right)  ^{1/p}=\frac{2C_{p}}{\sqrt{n}}\left\Vert X\right\Vert _{2p}^{2}%
\end{align*}
from the $L^{p}$ law of large numbers (\ref{eq:Lp-LLN}) in the Hilbert space
$HS\left(  H\right)  $, and computing the Hilbert-Schmidt norm of tensor
product by (\ref{eq:hilbert-schmidt-tensor}).

For the second term in (\ref{eq:cov-err}), first note that%
\[
\left\vert E_{n}\left(  X_{k}\right)  \otimes E_{n}\left(  X_{k}\right)
\right\vert _{HS\left(  H\right)  }=\left\vert E_{n}\left(  X_{k}\right)
\right\vert _{H}^{2}.
\]
Then, (\ref{eq:tensor-norm}) and the $L^{p}$ law of large numbers in $H$ yield%
\[
\left(  E\left\vert E_{n}\left(  X_{k}\right)  \otimes E_{n}\left(
X_{k}\right)  \right\vert _{HS\left(  H\right)  }^{p}\right)  ^{1/p}=\left(
E\left\vert E_{n}\left(  X_{k}\right)  \right\vert _{H}^{2p}\right)
^{1/p}=\left\Vert E_{n}\left(  X\right)  -0\right\Vert _{2p}^{2}\leq\left(
\frac{2C_{2p}}{\sqrt{n}}\left\Vert X_{1}\right\Vert _{2p}\right)  ^{2}.
\]

The proof is concluded by the use of the triangle inequality.
\end{proof}

Lemma \ref{lem:LpLLN-sample-cov} proves convergence of the sample covariance
in the Hilbert-Schmidt norm, which is weaker than the trace norm, which would
be the natural norm for covariances. However, the space $\operatorname*{Tr}%
\left(  H\right)  $ of trace class operators is not a Hilbert space.

\begin{corollary}
\label{cor:LpLLN-sample-cov-operator-norm}Using the inequality between the
operator norm and the Hilbert-Schmidt norm
(\ref{eq:compare-hilbert-schmidt-norm}), we have from
(\ref{eq:LpLLN-sample-cov}) that
\[
\left\Vert C_{n}\left(  X_{k}\right)  -\operatorname*{Cov}\left(
X_{1}\right)  \right\Vert _{p}\leq\left(  \frac{2C_{p}}{\sqrt{n}}%
+\frac{4C_{2p}^{2}}{n}\right)  \left\Vert X_{1}\right\Vert _{2p}^{2}.
\]

\end{corollary}

\begin{remark}
The argument developed in Lemma \ref{lem:Lp-large-numbers}, Lemma
\ref{lem:LpLLN-sample-cov}, and Corollary
\ref{cor:LpLLN-sample-cov-operator-norm} is essentially a reformulation the
proof of Lemma 3.3 in \cite{LeGland-2011-LSA}. We have extended the argument
to infinite dimension and made explicit the fact that the bounds are
independent of the dimension of the space.
\end{remark}

\subsection{Strong laws of large numbers}

\begin{theorem}
[{\cite[Theorem 7.9]{Ledoux-1991-PBS}}]Let $0<p<2$, and $\left(  X_{i}\right)
$ be i.i.d. random elements with values in Banach space $B$, and $S_{n}%
=X_{1}+\cdots+X_{n}$. Then%
\[
\frac{S_{n}}{n^{1/p}}\rightarrow0\text{ almost surely}%
\]
if and only if%
\[
E\left(  \left\vert X_{1}\right\vert ^{p}\right)  <\infty\text{ and }%
\frac{S_{n}}{n^{1/p}}\rightarrow0\text{ in probability.}%
\]

\end{theorem}

\begin{theorem}
[{\cite[Corollary 7.10]{Ledoux-1991-PBS}}]Let $\left(  X_{i}\right)  $ be
i.i.d. random elements with values in Banach space $B$ with $E\left(
\left\vert X_{1}\right\vert \right)  <\infty$, and $E_{n}=\frac{1}{n}\left(
X_{1}+\cdots+X_{n}\right)  $. Then%
\[
E_{n}\rightarrow E\left(  X\right)  \text{ almost surely.}%
\]

\end{theorem}

\section{Gaussian measures}

Given a measure $\mu$ on a Hilbert space $H$, its Fourier transform is the
function%
\[
\widehat{\mu}:h\in H\mapsto\int_{H}e^{-i\left\langle h,u\right\rangle }%
\mu\left(  du\right)  .
\]
More generally, if $U$ is a linear topological space and $\mu$ a measure on
the dual space $U^{\prime}$ (the space for all continous linear functionals on
$U$), the Fourier transform of $\mu$ is the function%
\[
\widehat{\mu}:v\in U\mapsto\int_{U^{\prime}}e^{-i\left\langle v,u\right\rangle
}\mu\left(  du\right)  ,
\]
where $\left\langle v,u\right\rangle =u\left(  v\right)  $.

\begin{proposition}
Given a separable Hilbert space $H$, $a\in H$, and $Q$ symmetric positive
semidefinite linear operator on $H$ of trace class, then there exists unique
probability measure $\mu$ on $H$ such that%
\begin{equation}
\widehat{\mu}\left(  h\right)  =e^{i\left\langle a,h\right\rangle -\frac{1}%
{2}\left\langle Qh,h\right\rangle }. \label{eq:gaussian-fourier-transform}%
\end{equation}
The measure $\mu$ has mean $a$ and covariance $Q$. If $X$ is a random element
with values in $H$ and with distribution $\mu$, and $u\in H$, then
\[
\left\langle u,X\right\rangle \sim N\left(  \left\langle a,u\right\rangle
,\left\langle u,Qu\right\rangle \right)  .
\]

\end{proposition}

The Gaussian measure satisfying (\ref{eq:gaussian-fourier-transform}) is
denoted by $N\left(  a,Q\right)  $.

\begin{proposition}
If $\mu$ is a Gaussian measure on a Hilbert space $H$ and $1\leq p<\infty$,
then%
\begin{equation}
M_{p}\left(  \mu\right)  =\int_{H}\left\vert x\right\vert ^{p}\mu\left(
dx\right)  <\infty. \label{eq:gaussian-moments}%
\end{equation}
That is, Gaussian measure has finite moments of all orders.
\end{proposition}

The space $L^{p}\left(  H,\mu\right)  $ for a Gaussian measure $\mu$ has some
addional properties.

\begin{proposition}
\label{prop:LpH}If $H$ is a Hilbert space, $\mu$ a centered Gaussian measure,
and $1\leq p<\infty$. Then $H^{\prime}\subset L^{p}\left(  H,\mu\right)  $ and
the embedding $H^{\prime}\rightarrow L^{p}\left(  H,\mu\right)  $ is
continous. If $H$ is infinitely dimensional, the $H^{\prime}$ and the
$L^{p}\left(  H,\mu\right)  $ norm are not equivalent on $H^{\prime}$.
\end{proposition}

\begin{proof}
Let $v\in H^{\prime}$. Then%
\begin{align}
\left\vert v\right\vert _{L^{p}\left(  H,\mu\right)  }  &  =\left(  \int%
_{H}\left\vert \left\langle v,u\right\rangle \right\vert ^{p}\mu\left(
du\right)  \right)  ^{1/p}\leq\left(  \int_{H}\left\vert v\right\vert
_{H^{\prime}}^{p}\left\vert u\right\vert _{H}^{p}\mu\left(  du\right)
\right)  ^{1/p}=\label{eq:LpH-cont}\\
&  =\left\vert v\right\vert _{H^{\prime}}^{p}\left(  \int_{H}\left\vert
u\right\vert _{H}^{p}\mu\left(  du\right)  \right)  ^{1/p}=\left\vert
v\right\vert _{H^{\prime}}M_{p}\left(  \mu\right)  ^{1/p},
\end{align}
with $M_{p}\left(  \mu\right)  <\infty$ by (\ref{eq:gaussian-moments}), which
proves that $H^{\prime}\subset L^{p}\left(  H,\mu\right)  $ with continous
embedding. To show that the opposite inequality does not hold, let $1\leq
p<\infty$ and $\left\{  e_{k}\right\}  $ be an orthonormal basis of
eigenvectors of the covariance of $\mu$ with the corresponding eigenvalues
$\lambda_{k}$, and define linear functionals $v_{n}:H\rightarrow\mathbb{R}$ by%
\[
v_{n}:u=%
{\textstyle\sum\nolimits_{k=1}^{\infty}}
u_{k}e_{k}\mapsto u_{n}.
\]
Then%
\[
\left\vert v_{n}\right\vert _{L^{p}\left(  H,\mu\right)  }=c_{n}\left(
\int_{-\infty}^{+\infty}\left\vert x\right\vert ^{p}e^{-x^{2}/2\lambda_{n}%
}dx\right)  ^{1/p}=C_{p}\lambda_{n}^{1/p}\rightarrow0\text{ as }%
n\rightarrow\infty,
\]
for some constant $C_{p}$, while%
\[
\left\vert v_{n}\right\vert _{H^{\prime}}=\sup\left\{  u_{n}:%
{\textstyle\sum\nolimits_{k=1}^{\infty}}
u_{k}^{2}=1\right\}  =1.
\]

\end{proof}

Note that the standard definition of the dual norm gives,%
\begin{equation}
\left\vert u\right\vert _{L^{p}\left(  H,\mu\right)  ^{\prime}}=\sup\left\{
\left\vert \left\langle u,v\right\rangle \right\vert :v\in H,\text{ }\int%
_{H}\left\vert \left\langle x,v\right\rangle \right\vert ^{p}\mu\left(
dx\right)  \leq1\right\}  ,\quad u\in H. \label{eq:def-dual-Lp-norm}%
\end{equation}

\begin{lemma}
The space $L^{p}\left(  H,\mu\right)  ^{\prime}\cap H$ is continously embedded
and dense in $H.$
\end{lemma}

\begin{proof}
From (\ref{eq:LpH-cont}), for any $u\in H$%
\begin{align*}
\left\vert u\right\vert _{H}  &  =\sup\left\{  \left\vert \left\langle
u,v\right\rangle \right\vert :v\in H^{\prime},\text{ }\left\vert v\right\vert
_{H^{\prime}}\leq1\right\} \\
&  \leq\sup\left\{  \left\vert \left\langle u,v\right\rangle \right\vert :v\in
H^{\prime},\text{ }\left\vert v\right\vert _{L^{p}\left(  H,\mu\right)  }\leq
C\right\} \\
&  =C\sup\left\{  \left\vert \left\langle u,v\right\rangle \right\vert :v\in
H^{\prime},\text{ }\left\vert v\right\vert _{L^{p}\left(  H,\mu\right)  }%
\leq1\right\} \\
&  \leq C\left\vert u\right\vert _{L^{p}\left(  H,\mu\right)  ^{\prime}}.
\end{align*}

\end{proof}

\section{White noise}

\label{ch:white noise}

\subsection{Classical definition of white noise stochastic process}

The classical concept of Gaussian white noise is a random processes, that is a
collection of random variables $F\left(  t\right)  :\Omega\rightarrow$
$\mathbb{R}$, $t\in\left[  0,1\right]  $, such that
\begin{align*}
E\left(  F\left(  t\right)  \right)   &  =0\text{ for all }t\in\left[
0,1\right]  ,\\
\operatorname*{Cov}\left(  F\left(  r\right)  ,F\left(  s\right)  \right)   &
=E\left(  F\left(  r\right)  F\left(  s\right)  \right)  =\left\{
\begin{array}
[c]{c}%
0\text{ if }r\neq s,\\
1\text{ if }r=s.
\end{array}
\right.
\end{align*}
Alternatively, $F$ is thought of as a random element whose values are real
function on $\left[  0,1\right]  $, called paths. White noise $F$ is gaussian
if the distribution of $F\left(  t\right)  $ is gaussian for all $t$, that is,
$F\left(  t\right)  \sim N\left(  0,1\right)  $. Similarly, white noise vector
$F$ in $\mathbb{R}^{n}$ is a random vector with zero mean and covariance equal
to the identity $I$, and gaussian white noise if $F\sim N(0,I)$. However, for
random elements $F$ with values in an infinitely dimensional Hilbert space,
this is not possible, because the covariance of a random element must be a
trace class operator, while $\operatorname*{Tr}I=+\infty$. Thus, a white noise
random element with values in a Hilbert space cannot exist, and something must
give way.

\subsection{White noise as a weak random variable}

One natural generalization of white noise in $\mathbb{R}^{n}$ to a Hilbert
space $H$ is a weak random variable. The space of all (not necessarily
bounded) linear functionals on $H$ is called the algebraic dual of $H$ and
denoted by $H^{\#}$. The value of linear functional $F\in H^{\#}$ at $u\in H$
is denoted by $\left\langle u,F\right\rangle $. Recall that random variable
means a measurable function on some probability space $\left(  \Omega
,\mathcal{F}\right)  $.

\begin{definition}
\label{def:weak-random-variable} Let $H$ be a Hilbert space. A map
$F:\Omega\rightarrow H^{\#}$ is a weak random variable if for every $u\in H$,
$\left\langle u,F\right\rangle $ is a random variable.
\end{definition}

The definitions of mean and covariance remain the same as in
(\ref{eq:def-mean})\ and (\ref{eq:def-covariance}).

\begin{definition}
For a weak random variable $F:\Omega\rightarrow H^{\#}$, the mean value
$E\left(  X\right)  $ (if it exists) is defined by%
\[
E\left(  F\right)  \in H,\quad\left\langle v,E\left(  F\right)  \right\rangle
=E\left(  \left\langle v,F\right\rangle \right)  \quad\forall v\in H
\]
and its covariance (if it exists) is the linear operator $\operatorname*{Cov}%
\left(  F\right)  $, defined by%
\begin{equation}
\left\langle u,\operatorname*{Cov}\left(  F\right)  v\right\rangle =E\left(
\left\langle u,F-E\left(  F\right)  \right\rangle \left\langle v,F-E\left(
F\right)  \right\rangle \right)  \quad\forall u,v\in H,
\label{eq:def-weak-covariance}%
\end{equation}

\end{definition}

\begin{definition}
A weak random variable $F$ is gaussian if for any $v_{1},\ldots,v_{n}$, the
random vector $\left(  \left\langle v_{1},F\right\rangle ,\ldots,\left\langle
v_{n},F\right\rangle \right)  $ is gaussian.
\end{definition}

\begin{definition}
\label{def:white-noise}White noise on a Hilbert space $H$ is a gaussian weak
random variable $F$ such that $E\left(  F\right)  =0$ and $\operatorname*{Cov}%
\left(  F\right)  =I.$
\end{definition}

\begin{remark}
It follows from Definition \ref{def:white-noise} that white noise $F$
satisfies%
\begin{equation}
\left\langle u,\operatorname*{Cov}\left(  F\right)  v\right\rangle =E\left(
\left\langle u,F\right\rangle \left\langle v,F\right\rangle \right)
=\left\langle u,v\right\rangle \quad\forall u,v\in H.
\label{eq:def-white-noise}%
\end{equation}
Condition (\ref{eq:def-white-noise}) can be stated equivalently that the map%
\[
u\in H\mapsto\left\langle u,F\left(  \cdot\right)  \right\rangle \in
L^{2}\left(  \Omega\right)
\]
is isometry.
\end{remark}

The definition of noise as a weak random variable means that the noise $F$
does not need to be measurable, only averages of the form $\left\langle
u,F\right\rangle $ are required to be measurable.

\begin{remark}
It follows immediately from Definition \ref{def:white-noise} that Gaussian
white noise $F$ satisfies $E\left(  \left\langle v,F\right\rangle \left\langle
u,F\right\rangle \right)  =\left\langle u,v\right\rangle $ and taking $u=v$,
we have%
\begin{equation}
E\left(  \left\vert \left\langle u,F\right\rangle \right\vert ^{2}\right)
=1\quad\forall u\in H,\left\vert u\right\vert =1. \label{eq:white-noise-avg}%
\end{equation}
While, in this sense, the norm of the random linear functional $u\mapsto
\left\langle u,F\right\rangle $ is one in squared average, and
(\ref{eq:white-noise-avg}) looks similar to the definition of the norm of a
functional,%
\[
\left\Vert F\left(  \omega\right)  \right\Vert _{H^{\prime}}=\sup_{u\in
H,\left\Vert u\right\Vert =1}\left\vert \left\langle u,F\left(  \omega\right)
\right\rangle \right\vert ,
\]
(\ref{eq:white-noise-avg}) does not say that linear functionals $u\mapsto
\left\langle u,F\left(  \omega\right)  \right\rangle $ for a fixed $\omega$
are bounded, i.e., (\ref{eq:white-noise-avg}) does not necessarily imply that%
\begin{equation}
\sup_{u\in H,\left\vert u\right\vert =1}\left\vert \left\langle u,F\left(
\omega\right)  \right\rangle \right\vert ^{2}<\infty\quad\forall\omega
\in\Omega. \label{eq:white-noise-sup}%
\end{equation}

(It would be good to find an example where (\ref{eq:white-noise-sup})\ is false.)
\end{remark}

\begin{remark}
However, we do not know yet if such weak random variable exists. In
\cite{Balakrishnan-1976-AFA}, white noise is constructed using a theory of
finitely additive measures instead of the standard measure theory. However, we
will proceed as in \cite{DaPrato-2006-IIA}, just with few more details, and
then interpret the result in the sense of Definition \ref{def:white-noise}.
\end{remark}

\subsection{Whitening of a random vector and white noise mapping}

A classical way how to construct a white noise vector in finite dimension is
by \textquotedblleft whitening" an existing random vector. Given a random
vector $X\sim N\left(  0,Q\right)  $ on $\mathbb{R}^{n}$, whitening of the
vector $X$ is defined as the random vector $Q^{-1/2}X$. Clearly,%
\begin{align*}
\operatorname*{Cov}\left(  Q^{-1/2}X\right)   &  =E\left(  Q^{-1/2}X\left(
Q^{-1/2}X\right)  ^{\mathrm{T}}\right)  =E\left(  Q^{-1/2}XX^{\mathrm{T}%
}Q^{-1/2}\right) \\
&  =Q^{-1/2}E\left(  XX^{\mathrm{T}}\right)  Q^{-1/2}=Q^{-1/2}QQ^{-1/2}=I
\end{align*}
so $Q^{-1/2}X$ is indeed white noise.

Suppose that $X$ is a random element with distribution $\mu$, where
$\mu=N\left(  0,Q\right)  $ is a Gaussian measure on an infinitely dimensional
Hilbert space $H$ with nondegenerate covariance operator $Q$, that is, $\ker
Q=\left\{  0\right\}  $. Since $Q$ is self-adjoint compact operator, there is
a complete orthonormal sequence $\left\{  v_{k}\right\}  $ consisting of the
eigenvectors of $Q$ with eigenvalues $\lambda_{1}\geq\lambda_{2}\geq\cdots>0$,
and $Q$ has the spectral decomposition%
\[
Q=\sum_{k=1}^{\infty}\lambda_{k}v_{k}\otimes v_{k}.
\]
Any real power of $Q$ is then defined by
\[
Q^{\theta}=\sum_{k=1}^{\infty}\lambda_{k}^{\theta}v_{k}\otimes v_{k}%
,\quad\theta\in\mathbb{R}.
\]
Since $Q$ is of trace class, $\sum_{k=1}^{\infty}\lambda_{k}<\infty$, and,
consequently, $\lim_{k\rightarrow\infty}\lambda_{k}=0$.

Unfortunately, the construction of white noise by whitening does not carry
over immediately to an infinitely dimensional Hilbert space because $Q^{-1/2}$
is an unbounded operator, defined only on the subspace%
\[
\mathcal{D}\left(  Q^{-1/2}\right)  =Q^{1/2}\left(  H\right)  =\left\{
\sum_{k=1}^{\infty}\lambda_{k}^{1/2}c_{k}v_{k}:\sum_{k=1}^{\infty}c_{k}%
^{2}<\infty\right\}  \subset H,
\]
which is called the \emph{Cameron-Martin space} of the measure $\mu$. A
definition of white noise as $Q^{-1/2}X$, $X\sim\mu$, then runs into a
difficulty: it is not true that $Q^{-1/2}$ is defined $\mu$-a.e. in $H$; in
fact, exactly the opposite is true \cite[Proposition 1.27]{DaPrato-2006-IIA},%
\[
\mu\left(  Q^{1/2}\left(  H\right)  \right)  =0.
\]
Consequently, $Q^{-1/2}X$ is not a random element, because it is not defined
a.s. Furthermore, the operator $Q^{-1/2}$ cannot be extended to all of $H$ by
continuity using the $H$-norm: The Cameron-Martin space $Q^{1/2}\left(
H\right)  $ is dense in $H$ but the operator $Q^{-1/2}:$ $Q^{1/2}\left(
H\right)  \rightarrow H$ is unbounded, that is, not continuous in the $H$ norm.

\subsection{White noise mapping}

Extending $Q^{-1/2}$ by continuity is possible if one can find some other
Banach space $V\supset H$ such that $Q^{-1/2}$ is continuous from $\left(
Q^{1/2}\left(  H\right)  ,\left\Vert \cdot\right\Vert \right)  $ to $\left(
V,\left\Vert \cdot\right\Vert _{V}\right)  $. Choose $V=L^{2}\left(
H,\mu\right)  $. This is a space of real-valued functions on $H$, not
necessarily linear, with the inner product%
\[
\left\langle F,G\right\rangle _{L^{2}\left(  H,\mu\right)  }=\int_{H}%
FGd\mu=\int_{H}F\left(  u\right)  G\left(  u\right)  \mu\left(  du\right)  .
\]
An element $u\in H$ is identified with the bounded linear functional
$x\mapsto\left\langle x,u\right\rangle $ in $H^{\prime}$, and then
$H=H^{\prime}$ by the Riesz representation theorem. Since
\begin{equation}
\left\vert u\right\vert _{L^{2}\left(  H,\mu\right)  }=\left(  \int%
_{H}\left\vert \left\langle x,u\right\rangle \right\vert ^{2}\mu\left(
dx\right)  \right)  ^{1/2}\leq\operatorname*{const}\left\vert u\right\vert
_{H}\quad\forall u\in H \label{eq:dual-in-L2}%
\end{equation}
by Proposition \ref{prop:LpH}, it holds that%
\[
H=H^{\prime}\subset L^{2}\left(  H,\mu\right)  .
\]
with continuous inclusion. The $L^{2}\left(  H,\mu\right)  $ inner product
becomes for $u$,
\[
v\in H\left\langle u,v\right\rangle _{L^{2}\left(  H,\mu\right)  }=\int%
_{H}\left\langle x,u\right\rangle \left\langle x,v\right\rangle \mu\left(
dx\right)  =\left\langle Qu,v\right\rangle
\]
by the definition of covariance (\ref{eq:def-weak-covariance}). Now $Q^{-1/2}$
is continuous, in fact isometry from $Q^{1/2}\left(  H\right)  $ to
$L^{2}\left(  H,\mu\right)  $:
\begin{align*}
\left\langle Q^{-1/2}w,Q^{-1/2}z\right\rangle _{L^{2}\left(  H,\mu\right)  }
&  =\int_{H}\left\langle x,Q^{-1/2}w\right\rangle \left\langle x,Q^{-1/2}%
z\right\rangle \mu\left(  dx\right) \\
&  =\left\langle Q^{-1/2}w,QQ^{-1/2}z\right\rangle =\left\langle
w,z\right\rangle .
\end{align*}
Hence, $Q^{-1/2}$ can be uniquely extended to a continuous mapping (which is
also isometry)%
\begin{equation}
W:H\rightarrow L^{2}\left(  H,\mu\right)  ,\quad\left\langle
Ww,Wz\right\rangle _{L^{2}\left(  H,\mu\right)  }=\left\langle
w,z\right\rangle , \label{eq:white-noise-isometry}%
\end{equation}
called the \emph{white noise mapping }\cite[p. 23]{DaPrato-2006-IIA}.

Since the $L^{2}\left(  H,\mu\right)  $ limit of a sequence of linear
functionals is clearly a linear functional, $W\left(  z\right)  $ are still
linear functionals on $H$, so $W\left(  z\right)  \in H^{\#}\cap L^{2}\left(
H,\mu\right)  $. However, for $z\in H\setminus Q^{1/2}\left(  H\right)  $,
they are not necessarily bounded linear functionals, because the $L^{2}\left(
H,\mu\right)  $ limit of a sequence of bounded linear functionals on $H$ does
not need to be a bounded functional.

\begin{proposition}
Let $\mu$ be a centered Gaussian measure on Hilbert space $H$, $W:H\rightarrow
H^{\#}\cap L^{2}\left(  H,\mu\right)  $ the white noise mapping, and
$X:\Omega\rightarrow H$ a random element with distribution $\mu$. Define
$F:\Omega\rightarrow H^{\#}\cap L^{2}\left(  H,\mu\right)  $ by $F\left(
\omega\right)  =WX\left(  \omega\right)  $. Then $\left\langle
u,F\right\rangle $ is a Gaussian random variable for every $u\in H$, and
\begin{equation}
E\left(  \left\langle v,F\right\rangle \left\langle u,F\right\rangle \right)
=\left\langle u,v\right\rangle ,\quad\forall u,v\in H.
\label{eq:white-noise-covariance}%
\end{equation}

\end{proposition}

\begin{proof}
We already know that for a fixed $\omega$, $F\left(  \omega\right)  $ is a
linear functional, and the mapping $WX\left(  \omega\right)  $ is linear
because the mapping $u\in H\mapsto$ $W_{u}\in$ $L^{2}\left(  H,\mu\right)  $
is linear, being a continuous extension a linear mapping on a dense subspace
of $H$. By substitution from (\ref{eq:white-noise-isometry}),%
\begin{align*}
E\left(  \left\langle v,F\right\rangle \left\langle u,F\right\rangle \right)
&  =\int_{\Omega}\left\langle v,F\left(  \omega\right)  \right\rangle
\left\langle u,F\left(  \omega\right)  \right\rangle d\omega\\
&  =\int_{\Omega}\left\langle v,WX\left(  \omega\right)  \right\rangle
\left\langle u,WX\left(  \omega\right)  \right\rangle d\omega\\
&  =\int_{H}\left\langle v,Wx\right\rangle \left\langle u,Wx\right\rangle
\mu\left(  dx\right) \\
&  =\int_{H}\left\langle x,Wv\right\rangle \left\langle x,Wu\right\rangle
\mu\left(  dx\right) \\
&  =\left\langle Wv,QWu\right\rangle \\
&  =\left\langle Q^{-1/2}v,QQ^{-1/2}u\right\rangle \\
&  =\left\langle v,u\right\rangle
\end{align*}
for all $u,v\in Q^{1/2}\left(  H\right)  $. Because $Q^{1/2}\left(  H\right)
$ is dense in $H$, the covariance property (\ref{eq:white-noise-covariance})
follows by continuity.
\end{proof}

\section{Convergence analysis of the EnKF in the large ensemble limit}

\subsection{Optimal statistical interpolation}

Consider a stochastic model with the state $V\sim N\left(  \mu,Q\right)  $ in
$\mathbb{R}^{n}$, and a given observation (i.e., data) vector $d$. The
$n\times n$ matrix $Q$ is called the background covariance. The correspondence
between the state and the data is given by a linear observation operator $H$
and data error covariance $R$: the data is assumed to be distributed as $d\sim
N\left(  Hv,R\right)  $, given the value of the state $V$ equal to $v$. That
is, $Hv$ is what the data would be if there were no errors and the truth were
$v$, and $R$ is the covariance of the normally distributed data error. Then
the probability density of the state is%
\[
p_{V}\left(  v\right)  \propto e^{-\frac{1}{2}\left(  v-\mu\right)
^{\mathrm{T}}Q^{-1}\left(  v-\mu\right)  }%
\]
where $\propto$ means proportional, and the data likelihood is%
\[
p\left(  d|v\right)  \propto e^{-\frac{1}{2}\left(  Hv-d\right)  ^{\mathrm{T}%
}R^{-1}\left(  Hv-d\right)  }.
\]

From the Bayes theorem, the probability density of the model state after the
data is taken into account, called the analysis $V^{a}$, is%
\[
p_{V^{a}}(v)\propto p(d|v)p_{V}(v)\propto e^{-\frac{1}{2}\left[  \left(
v-\mu\right)  ^{\mathrm{T}}Q^{-1}\left(  v-\mu\right)  +\left(  Hv-d\right)
^{\mathrm{T}}R^{-1}\left(  Hv-d\right)  \right]  }%
\]
Note that the exponent involves a quadratic function of $v$, which we can
write in the form%
\[
\left(  v-\mu\right)  ^{\mathrm{T}}Q^{-1}\left(  v-\mu\right)  +\left(
Hv-d\right)  ^{\mathrm{T}}R^{-1}\left(  Hv-d\right)  =\left(  v-\mu
^{a}\right)  ^{\mathrm{T}}\left(  Q^{a}\right)  ^{-1}\left(  v-\mu^{a}\right)
\]
and compare the quadratic and linear terms:%
\begin{align*}
v^{\mathrm{T}}Q^{-1}v+v^{\mathrm{T}}H^{\mathrm{T}}R^{-1}Hv  &  =v^{\mathrm{T}%
}\left(  Q^{a}\right)  ^{-1}v\\
-2v^{\mathrm{T}}Q^{-1}\mu-2v^{\mathrm{T}}H^{\mathrm{T}}R^{-1}d  &
=-2v^{\mathrm{T}}\left(  Q^{a}\right)  ^{-1}\mu^{a}%
\end{align*}
which gives%
\begin{align}
Q^{-1}+H^{\mathrm{T}}R^{-1}H  &  =\left(  Q^{a}\right)  ^{-1}%
\label{eq:osi-cov}\\
Q^{-1}\mu+H^{\mathrm{T}}R^{-1}d  &  =\left(  Q^{a}\right)  ^{-1}\mu^{a}
\label{eq:osi-mean}%
\end{align}
which gives%
\begin{align*}
Q^{a}  &  =\left(  Q^{-1}+H^{\mathrm{T}}R^{-1}H\right)  ^{-1}\\
\mu^{a}  &  =Q^{a}\left(  Q^{-1}\mu+H^{\mathrm{T}}R^{-1}d\right)
\end{align*}
Using the Sherman-Morrison-Woodbury formula
\begin{equation}
\left(  A+UCV\right)  ^{-1}=A^{-1}-A^{-1}U\left(  C^{-1}+VA^{-1}U\right)
^{-1}VA^{-1}, \label{eq:SMW}%
\end{equation}
we have%
\begin{align}
Q^{a}  &  =\left(  \underbrace{Q^{-1}}_{A}+\underbrace{H^{\mathrm{T}}}%
_{U}\underbrace{R^{-1}}_{C}\underbrace{H}_{V}\right)  ^{-1}%
\label{eq:kf-covariance}\\
&  =Q-QH^{\mathrm{T}}\left(  R+HQH^{\mathrm{T}}\right)  ^{-1}HQ\nonumber\\
&  =\left(  I-QH^{\mathrm{T}}\left(  R+HQH^{\mathrm{T}}\right)  ^{-1}H\right)
Q\\
&  =(I-KH)Q, \label{eq:kf-gain-cov}%
\end{align}
where $K$ is the Kalman gain matrix, given by%
\begin{equation}
K=QH^{\mathrm{T}}(HQH^{\mathrm{T}}+R)^{-1}, \label{eq:kf-gain}%
\end{equation}
and%
\begin{align*}
\mu^{a}  &  =Q^{a}\left(  Q^{-1}\mu+H^{\mathrm{T}}R^{-1}d\right) \\
&  =(I-KH)Q\left(  Q^{-1}\mu+H^{\mathrm{T}}R^{-1}d\right) \\
&  =(I-KH)\mu+\left(  I-QH^{\mathrm{T}}\left(  R+HQH^{\mathrm{T}}\right)
^{-1}H\right)  QH^{\mathrm{T}}R^{-1}d\\
&  =(I-KH)\mu+Kd\\
&  =\mu-K\left(  H\mu-d\right)  ,
\end{align*}
since%
\begin{align*}
\left(  I-QH^{\mathrm{T}}\left(  R+HQH^{\mathrm{T}}\right)  ^{-1}H\right)
QH^{\mathrm{T}}R^{-1}  &  =QH^{\mathrm{T}}\left(  I-\left(  R+HQH^{\mathrm{T}%
}\right)  ^{-1}HQH^{\mathrm{T}}\right)  R^{-1}\\
&  =QH^{\mathrm{T}}\left(  R+HQH^{\mathrm{T}}\right)  ^{-1}\left(  \left(
R+HQH^{\mathrm{T}}\right)  -HQH^{\mathrm{T}}\right)  R^{-1}\\
&  =QH^{\mathrm{T}}(HQH^{\mathrm{T}}+R)^{-1}=K
\end{align*}

That is, the analysis distribution is%
\[
V^{a}\sim N\left(  \mu^{a},Q^{a}\right)
\]
where $\mu^{a}$ is given by the least squares estimate $v=\mu^{a}$,
\begin{equation}
\left(  v-\mu\right)  ^{\mathrm{T}}Q^{-1}\left(  v-\mu\right)  +\left(
Hv-d\right)  ^{\mathrm{T}}R^{-1}\left(  Hv-d\right)  \rightarrow\min_{v}
\label{eq:kf-opt}%
\end{equation}
with the optimality conditions%
\[
Q^{-1}\left(  v-\mu\right)  +H^{\mathrm{T}}R^{-1}\left(  Hv-d\right)  =0
\]
which gives
\begin{equation}
\mu^{a}=\left(  Q^{-1}+H^{\mathrm{T}}R^{-1}H\right)  ^{-1}\left(  Q^{-1}%
\mu+H^{\mathrm{T}}R^{-1}d\right)  . \label{eq:kf-mean-opt}%
\end{equation}
and%
\begin{equation}
\mu^{a}=\mu+K(d-H\mu). \label{eq:kf-mean}%
\end{equation}

When only the mean is advanced by the model and the background covariance is
obtained in other ways (typically by an expert judgement involving data
analysis), the method (\ref{eq:kf-mean})-(\ref{eq:kf-gain}) is called optimal
statistical interpolation \cite{Kalnay-2003-AMD}.

\subsection{Kalman filter}

Here is how a linear model transform state with normal distribution. Suppose
$u\propto N\left(  \mu_{k}^{a},Q_{k}^{a}\right)  $ is the analysis in step $k$
with linear model $Au+f$ acting on it. Then the forecast in the next step
$k+1$ is
\[
Au+f\propto N\left(  A\mu_{k}^{a}+f,AQ_{k}^{a}A^{\top}\right)
\]
The output of the model is the forecast for the next cycle (time step)
\[
u\propto N\left(  \mu_{k+1}^{f},Q_{k+1}^{f}\right)  ,\quad\mu_{k+1}^{f}%
=A\mu_{k}^{a}+f,\ Q_{k+1}^{f}=AQ_{k}^{a}A^{\top}%
\]
This is how the original filter was formulated for a linear model
\cite{Kalman-1960-NAL}. The model can of course be different from step to
step, $A_{k}u+f_{k}$. When the model is nonlinear, we use the first order
approximation%
\[
\mathcal{A}\left(  u\right)  \approx\mathcal{A}\left(  u_{k}\right)  +J\left(
u_{k}\right)  \left(  u-u_{k}\right)
\]
and take
\[
u\propto N\left(  \mu_{k+1}^{f},Q_{k+1}^{f}\right)  ,\quad\mu_{k+1}%
^{f}=\mathcal{A}\left(  \mu_{k}^{a}\right)  ,\ Q_{k+1}^{f}=J\left(
u_{k}\right)  Q_{k}^{a}J\left(  u_{k}\right)  ^{\top}+D,
\]
where $D>0$ is regularization (oftem diagonal boosting) boosting matrix. These
are the equations of the extended Kalman filter. The extension to nonlinear
case \cite{McGee-1985-DKF} made many applications possible, and is still a
de-facto standard in navigation including GPS.

For practical use, regulariation is important - we add a positive definite
matrix $D$ to the covariance, usually diagonal. This prevents the degeneration
of the Kalman filter covariance to zero, because in optimal statistical
interpolation, the variance always decreases. The regularization
\textquotedblleft accounts for\textquotedblright\ errors in the model vs. the
reality, which the data come from.

\subsection{The ensemble Kalman filter}

To avoid manipulating the covariance matrix and to allow to use the model as a
black box, the ensemble Kalman filter approximates the state distribution
$N\left(  \mu,Q\right)  $ by a collection of of random vectors $X_{k}%
\in\mathbb{R}^{n}$,
\begin{equation}
X=[X_{1},\ldots,X_{N}] \label{eq:ensemble}%
\end{equation}
called an ensemble. Note that in general the ensemble is \emph{not} a sample
(i.i.d.~set of random elements) from $N\left(  \mu,Q\right)  $; the analysis
step (\ref{eq:enkf-analysis}) below breaks the independence of the ensemble
members. It is convenient to operate with an ensemble as a matrix with the
columns $X_{k}$. The analysis step of the ensemble Kalman filter (EnKF)
consists of creating the perturbed data ensemble,%
\begin{equation}
D=[D_{1},\ldots,D_{N}],\quad D_{k}\sim N\left(  d,R\right)  ,
\label{eq:data-perturbation}%
\end{equation}
independent of $X$, and the analysis ensemble $X^{a}$, by%
\begin{equation}
X^{a}=X+K_{N}(D-HX), \label{eq:enkf-analysis}%
\end{equation}
where $K_{N}$ is the approximate Kalman gain matrix, given by%
\begin{equation}
K_{N}=Q_{N}H^{\mathrm{T}}(HQ_{N}H^{\mathrm{T}}+R)^{-1}, \label{eq:enkf-gain}%
\end{equation}
and $Q_{N}$ is the sample covariance computed from the ensemble $X$,%
\begin{equation}
Q_{N}=\frac{1}{N-1}\sum_{k=1}^{N}A_{k}A_{k}^{\mathrm{T}},\quad A_{k}=\left(
X_{k}-\overline{X}\right)  ,\quad\overline{X}=\frac{1}{N}\sum_{k=1}^{N}X_{k}.
\label{eq:sample-cov}%
\end{equation}

For practical application, note that $\overline{X}=\frac{1}{N}Xe$, where $e$
is the column vector of all ones of length $N$, and then
\[
HQ_{N}H^{\mathrm{T}}=\frac{1}{N-1}BB^{\mathrm{T}},
\]
where
\[
B=HX-HX\frac{ee^{\mathrm{T}}}{N}%
\]
and use the Sherman-Morrison-Woodbury formula to (\ref{eq:enkf-gain})\ to get%
\begin{equation}
(HQ_{N}H^{\mathrm{T}}+R)^{-1}=R^{-1}\left[  I-\frac{1}{N-1}B\left(
I+\frac{B^{\mathrm{T}}R^{-1}B}{N-1}\right)  ^{-1}B_{k}^{\mathrm{T}}R_{k}%
^{-1}\right]
\end{equation}

We can write the ensemble covariance as%
\begin{align*}
Q_{N}  &  =\frac{1}{N-1}\left(  X-X\frac{ee^{\mathrm{T}}}{N}\right)  \left(
X-X\frac{ee^{\mathrm{T}}}{N}\right)  ^{\mathrm{T}}\\
&  =X\frac{1}{N-1}\left(  I-\frac{ee^{\mathrm{T}}}{N}\right)  X^{\mathrm{T}},
\end{align*}
and the analysis ensemble (\ref{eq:enkf-analysis}) in the transformation form
as linear combination of the forecast forecast ensemble,%
\begin{equation}
X^{a}=XT,\quad T=I+\frac{1}{N-1}\left(  I-\frac{ee^{\mathrm{T}}}{N}\right)
X^{\mathrm{T}}H^{\mathrm{T}}(HQ_{N}H^{\mathrm{T}}+R)^{-1}(D-HX).
\label{eq:enkf-transform}%
\end{equation}

After the analysis step (\ref{eq:enkf-analysis}), each member of the ensemble
is advanced by the model (\ref{eq:model}) independently,%
\begin{equation}
X^{f}=\mathcal{M}\left(  X^{a}\right)  , \label{eq:enkf-model}%
\end{equation}
then we use $X^{f}$ in place of $X$ above, and the process repeats.

This is the version of EnKF from \cite{Burgers-1998-ASE}, distinguished by the
use of the randomized data (\ref{eq:data-perturbation}). See
\cite{Evensen-2009-DAE} for futher details and variants. The reason for the
data perturbation (\ref{eq:data-perturbation}) is to make the sample
covariance of the analysis what it should be.

\begin{lemma}
\label{lem:enkf-cov}Suppose $U\sim N\left(  \mu,Q\right)  $ is a random vector
with values in $\mathbb{R}^{n}$, and $D\sim N\left(  d,R\right)  $ is a random
vectors with values in $\mathbb{R}^{m}$, $U$ and $D$ are independent, and
$H\in\mathbb{R}^{m\times n}$ is a matrix. Define%
\begin{equation}
U^{a}=U+K(D-HU),\quad K=QH^{\mathrm{T}}(HQH^{\mathrm{T}}+R)^{-1}.
\label{eq:enkf-analysis-rv}%
\end{equation}
Then $U^{a}\sim N\left(  \mu^{a},Q^{a}\right)  $, where%
\begin{equation}
\mu^{a}=\mu+K(d-H\mu),\quad Q^{a}=\left(  I-KH\right)  Q
\label{eq:enkf-analysis-mean-cov-rv}%
\end{equation}

\end{lemma}

\begin{proof}
Since the $D$ is gaussian and $K$ and $H$ are non-random, $U^{a}$ is also
gaussian. Gaussian distribution is uniquely specified by its mean and
covariance. Taking the mean in (\ref{eq:enkf-analysis-rv}) shows that $\mu
^{a}=E\left(  U_{k}^{a}\right)  $. Using the independence of $U$ and $D$, the
computation from \cite{Burgers-1998-ASE}
\begin{align*}
E\left\langle U^{a}-\mu^{a},U^{a}-\mu^{a}\right\rangle  &  =E\left\langle
U-\mu+K(\left(  D-d\right)  -H\left(  U-\mu\right)  ),U-\mu+K(\left(
D-d\right)  -H\left(  U-\mu\right)  )\right\rangle \\
&  =E\left\langle \left(  I-KH\right)  \left(  U-\mu\right)  ),\left(
I-KH\right)  \left(  U-\mu\right)  )\right\rangle +E\left\langle K\left(
D-d\right)  ,K\left(  D-d\right)  \right\rangle \\
&  =\left(  I-KH\right)  Q\left(  I-KH\right)  ^{\top}+KRK^{\top}\\
&  =Q-KHQ-QH^{\top}K^{\top}+KHQH^{\top}K^{\top}+KRK^{\top}\\
&  =\left(  I-KH\right)  Q-QH^{\top}K^{\top}+K\left(  HQH^{\top}+R\right)
K^{\top}\\
&  =\left(  I-KH\right)  Q-QH^{\top}K^{\top}+QH(HQH^{\mathrm{T}}%
+R)^{-1}\left(  HQH^{\top}+R\right)  K^{\top}\\
&  =\left(  I-KH\right)  Q
\end{align*}
shows that $\operatorname{Cov}U^{a}=Q^{a}$ from
(\ref{eq:enkf-analysis-mean-cov-rv}).
\end{proof}

Applying Theorem

\begin{corollary}
\label{cor:enkf-cov} Given the forecast enemble
\[
U=[U_{1},\ldots,U_{N}],\quad U_{k}\sim N\left(  \mu,Q\right)  \text{ i.i.d.,}%
\]
and the perturbed data ensemble
\[
D=[D_{1},\ldots,D_{N}],\quad D_{k}\sim N\left(  d,R\right)  ,
\]
independent of $U$, define $U^{a}$ by the EnKF analysis formula with the exact
covariance,
\begin{equation}
U^{a}=U+K(D-HU).\quad K=QH^{\mathrm{T}}(HQH^{\mathrm{T}}+R)^{-1}
\label{eq:def-Ua}%
\end{equation}
Then
\begin{equation}
U^{a}=[U_{1}^{a},\ldots,U_{N}^{a}],\quad U_{k}^{a}\sim N\left(  \mu^{a}%
,Q^{a}\right)  \text{ i.i.d}. \label{eq:Ua}%
\end{equation}

\end{corollary}

\subsection{Finite exchangeable sequences}

An $N$-tuple of random elements $\left[  X_{1},\ldots,X_{N}\right]  $ with
values in $H\times\ldots\times H=H^{N}$ is exchangeable if their joint
distribution is invariant to a permutation of the indices; that is, for any
permutation $\pi$ of the numbers $1,\ldots,N$ and any Borel set $B\subset
H^{N}$,%
\[
\Pr\left(  \left[  X_{\pi\left(  1\right)  },\ldots,X_{\pi\left(  N\right)
}\right]  \in B\right)  =\Pr\left(  \left[  X_{1},\ldots,X_{N}\right]  \in
B\right)  .
\]

Clearly, an i.i.d. sequence is exchangeable. The following lemmas are almost obvious.

\begin{lemma}
\label{lem:ex-sum}If random elements $Y_{1},\ldots,Y_{N}$ are exchangeable,
$Z_{1},\ldots,Z_{N}$ are exchangeable, and $Y_{1},\ldots,Y_{N}$ are
independent of $Z_{1},\ldots,Z_{N}$, then $Y_{1}+Z_{1},\ldots,Y_{N}+Z_{N}$ are exchangeable.
\end{lemma}


\begin{lemma}
[{\cite[Lemma 1]{Mandel-2011-CEK}}]\label{lem:ex-fun}If random elements
$Z_{1},\ldots,Z_{N}$ are exchangeable and $Y_{k}=F\left(  Z_{1},\ldots
,Z_{N},Z_{k}\right)  $, where $F$ is measurable, and permutation invariant in
the first $N$ arguments, then $Y_{1},\ldots,Y_{N}$ are also exchangeable.
\end{lemma}

\subsection{Convergence of the EnKF in the large ensemble limit}

Here is a simplified presentation of the proof from \cite{Mandel-2011-CEK}.
Clearly, the hope is that by a suitable law of large numbers, $Q_{N}%
\rightarrow Q$, as $N\rightarrow\infty$, then $K_{N}\rightarrow K$, and, in
some sense, the EnKF analysis step (\ref{eq:enkf-analysis}) is asymptotically
correct. The use of Slutsky's theorem for showing that $K_{N}\rightarrow K$
was suggested in \cite{Furrer-2007-EHP}. However, the actual formulation of
the result and its proof were done only later in
\cite{LeGland-2011-LSA,Mandel-2011-CEK}. We follow \cite{Mandel-2011-CEK}
here, with some simplifications discovered since \cite{Mandel-2011-CEK} was written.

Here we need to introduce the analysis cycle index $m$, which we were avoiding
so far in order not to complicate the notation. We will drop it later whenever
possible. The data assimilation process starts with an initial ensemble
$X^{\left(  0\right)  }$, which we assume to be gaussian i.i.d., and it
proceeds through a sequence of analysis cycle. First advance the initial
ensemble by the model (\ref{eq:enkf-model}), $X^{\left(  1\right)
}=\mathcal{M}^{\left(  1\right)  }\left(  X^{\left(  0\right)  }\right)  $.
The analysis $X^{\left(  1\right)  ,a}$ is obtained by the use of
(\ref{eq:enkf-analysis}) with $X^{\left(  1\right)  }$ in place of $X$. In
each cycle $m>1$, the ensemble is advanced by the model (\ref{eq:enkf-model}),
$X^{\left(  m\right)  }=\mathcal{M}^{\left(  m\right)  }\left(  X^{\left(
m-1\right)  ,a}\right)  $ and the analysis step (\ref{eq:enkf-analysis}) is
applied with $X^{\left(  m\right)  }$ in place of $X$, giving $X^{\left(
m\right)  ,a}$. The data $d^{\left(  m\right)  }$ is given in each step, and
the observation matrix $H^{\left(  m\right)  }$ and data covariance
$R^{\left(  m\right)  }$ can also change from step to step.

To maintain gaussian distribution at least in the limit as $N\rightarrow
\infty$, we assume that the model in each step $m$ is linear, as in
(\ref{eq:model}), with $A=A^{\left(  m\right)  }$ and $b=b^{\left(  m\right)
}$.

Instead of proving properties of the ensembles $X^{\left(  m\right)  }$ alone,
we want to show that they approach the ensembles $U^{\left(  m\right)  }$,
obtained by using the exact state covariance $Q^{\left(  m\right)  }$ in each
step. Thus, we put $U^{\left(  0\right)  }=X^{\left(  0\right)  }$ (which is
i.i.d. gaussian), and define $U^{\left(  m\right)  }$ in the same way as
$X^{\left(  m\right)  }$, except that we use the exact state covariance in
every step following (\ref{eq:def-Ua}), instead of the sample covariance from
the ensemble $X^{\left(  m\right)  }$.

We will need the property that for each $m$, the ordered pairs $\left[
X_{k}^{\left(  m\right)  };U_{k}^{\left(  m\right)  }\right]  $ are
identically distributed for all $k=1,\ldots,m$. We find it convenient to write
the pairs vertically as
\begin{equation}
\left[  X_{k}^{\left(  m\right)  };U_{k}^{\left(  m\right)  }\right]  =\left[
\begin{array}
[c]{c}%
X_{k}^{\left(  m\right)  }\\
U_{k}^{\left(  m\right)  }%
\end{array}
\right]  ,\quad k=1,\ldots,N, \label{eq:stacked}%
\end{equation}
and prove a stronger property, namely that they are exchangeable. A set of of
random elements is called \emph{exchangeable} if their joint distribution does
not depend on a permutation of its arguments. Clearly, i.i.d. random elements
are exchangeable.

When we know that (\ref{eq:stacked}) are identically distributed, we can
formulate the convergence of the ensemble members in terms of the convergence
of just one ensemble member, say $X_{1}\rightarrow U_{1}$ (in a suitable
sense) as the number of ensemble members $N\rightarrow\infty$. We will measure
the differences in $L^{p}$ norms. If $W$ is a random element (either vector or
matrix), let $\left\vert W\right\vert $ be the usual Euclidean norm (for
vectors) or spectral norm (for a matrix). For $1\leq p<\infty$, denote%
\[
\left\Vert W\right\Vert _{p}=\left(  E\left(  \left\vert W\right\vert
^{p}\right)  \right)  ^{1/p}.
\]
The space $L^{p}$ (of vectors or matrices) consists of all random elements $W$
such that the moment $E\left(  \left\vert W\right\vert ^{p}\right)  <\infty$.

\begin{theorem}
Suppose that $U_{k}^{\left(  0\right)  }$ and $D_{k}^{\left(  m\right)  }$,
$k=1,\ldots,N$ are taken from the beginning of fixed infinite sequences of
random elements. Then, for each $m$, (\ref{eq:stacked}) are exchangeable, and
$X_{1}^{\left(  m\right)  }\rightarrow U_{1}^{\left(  m\right)  }$ in $L^{p}$
as $N\rightarrow\infty$, for any $1\leq p<\infty$.
\end{theorem}

\begin{proof}
The statement is true for $m=0$ because $U^{\left(  0\right)  }=X^{\left(
0\right)  }$ are i.i.d., and $U_{1}^{\left(  0\right)  }$ is gaussian, so it
has finite moments of all orders, thus $U_{1}^{\left(  0\right)  }\in L^{p}$
for all $1\leq p<\infty$. Assume that the statement is true for some $m>0$. To
simplify the notation, we drop the index $m$.

We first prove that $\left[  X_{k}^{a};U_{k}^{a}\right]  $, $k=1,\ldots,N$ are
exchangeable. Write (\ref{eq:enkf-analysis}) and (\ref{eq:def-Ua}) as%
\begin{align*}
\left[
\begin{array}
[c]{c}%
X^{a}\\
U^{a}%
\end{array}
\right]   &  =\left[
\begin{array}
[c]{c}%
X+K_{N}(D-HX)\\
U+K(D-HU)
\end{array}
\right] \\
&  =\left[
\begin{array}
[c]{c}%
X\\
U
\end{array}
\right]  +\left[
\begin{array}
[c]{cc}%
K_{N} & 0\\
0 & K
\end{array}
\right]  \left(  \left[
\begin{array}
[c]{c}%
D\\
D
\end{array}
\right]  -\left[
\begin{array}
[c]{cc}%
H & 0\\
0 & H
\end{array}
\right]  \left[
\begin{array}
[c]{c}%
X\\
U
\end{array}
\right]  \right)  .
\end{align*}
Now
\[
\left[
\begin{array}
[c]{c}%
D\\
D
\end{array}
\right]  -\left[
\begin{array}
[c]{cc}%
H & 0\\
0 & H
\end{array}
\right]  \left[
\begin{array}
[c]{c}%
X\\
U
\end{array}
\right]
\]
has exchangeable columns by Lemma \ref{lem:ex-sum}, and $K_{N}$ is a function
of the sample covariance (\ref{eq:sample-cov}), which is permulation
invariant, so $\left[  X_{k}^{a};U_{k}^{a}\right]  $, $k=1,\ldots,N$, are
exchangeable by Lemma \ref{lem:ex-fun}.

Next we give an a-priori bound on the $L^{p}$ norms of $U_{1}^{a}$ and
$X_{1}^{a}$. Since $U_{1}^{a}$ is gaussian, it has all moments. For $X_{1}%
^{a}$, note that $R$ is positive definite and $Q_{N}$ positive semidefinite,
so we have%
\[
\left\vert (HQ_{N}H^{\mathrm{T}}+R)^{-1}\right\vert \leq\left\vert
R^{-1}\right\vert .
\]
From Lemma \ref{lem:bound-sample-covariance},
\[
\left\Vert Q_{N}\right\Vert _{p}\leq2\left\Vert X_{1}\right\Vert _{2p}^{2},
\]
thus%
\[
\left\Vert Q_{N}H^{\mathrm{T}}(HQ_{N}H^{\mathrm{T}}+R)^{-1}\right\Vert
_{p}\leq2\left\Vert X_{1}\right\Vert _{2p}^{2}\left\vert R^{-1}\right\vert
\left\vert H\right\vert ,
\]
hence from (\ref{eq:enkf-analysis}), using the triangle inequality and Cauchy
inequality,
\begin{align*}
\left\Vert X_{1}^{a}\right\Vert _{p}  &  =\left\Vert X_{1}+K_{N}%
(D-HX_{1})\right\Vert _{p}\leq\left\Vert X_{1}\right\Vert _{p}+\left\Vert
K_{N}\right\Vert _{2p}\left\Vert D-HX_{1}\right\Vert _{2p}\\
&  \leq\left\Vert X_{1}\right\Vert _{p}+\left\Vert Q_{N}H^{\mathrm{T}}%
(HQ_{N}H^{\mathrm{T}}+R)^{-1}\right\Vert _{2p}\left(  \left\Vert
D_{k}\right\Vert _{2p}+\left\vert H\right\vert \left\Vert X_{1}\right\Vert
_{2p}\right) \\
&  \leq\left\Vert X_{1}\right\Vert _{p}+2\left\Vert X_{1}\right\Vert _{2p}%
^{2}\left\vert R^{-1}\right\vert \left\vert H\right\vert \left(  \left\Vert
D_{1}\right\Vert _{2p}+\left\vert H\right\vert \left\Vert X_{1}\right\Vert
_{2p}\right)  .
\end{align*}

We now add the subscript $N$ to ensemble members to indicate explicitly they
come from an ensemble of size $N$. Thus, $X_{1,N}$ is the first element in the
ensemble. The key to the convergence proof is an estimate the difference
between the sample covariance $Q_{N}$ computed from the ensemble $X$ and the
exact covariance $Q$. Denote by $C_{N}\left(  Y\right)  $ the sample
covariance computed from $Y=\left[  Y_{1},\ldots,Y_{N}\right]  $. From Lemma
\ref{lem:LpLLN-sample-cov}, since $U_{k}$ are i.i.d., we have the weak law of
large numbers in $L^{2}$,%
\[
\left\Vert C_{N}\left(  U\right)  -Q\right\Vert _{2}=E\left(  \left\vert
C_{N}\left(  U\right)  -Q\right\vert ^{2}\right)  ^{1/2}\leq\left(  \frac
{2}{\sqrt{N}}+\frac{\operatorname*{const}}{N}\right)  \left\Vert
U_{1}\right\Vert _{4}^{2}.
\]
From Lemma \ref{lem:sample-cov-norm-cont} and the fact that $X_{k}$ and
$U_{k}$ are identically distributed, we have the continuity of the sample
covariance,
\[
\left\Vert C_{N}(X)-C_{N}(U)\right\Vert _{2}\leq\sqrt{8}\left\Vert
X_{1,N}-U_{1}\right\Vert _{4}\sqrt{\left\Vert X_{1,N}\right\Vert _{4}%
^{2}+\left\Vert U_{1}\right\Vert _{4}^{2}}%
\]
Thus, by the triangle inequality , and since $L^{2}$ convergence implies
convergence in probability,%
\[
X_{1,N}\rightarrow U_{1}\text{ in }L^{4}\Longrightarrow C_{N}(X)\rightarrow
Q\text{ in }L^{2}\Longrightarrow C_{N}(X)\xrightarrow{\mathrm{P}}Q\text{.}%
\]

Since $R$ is positive definite, the mapping $S\mapsto K=SH^{\mathrm{T}%
}(HSH^{\mathrm{T}}+R)^{-1}$ is continuous at any symmetric positive
semidefinite matrix $S$, thus by the continuous mapping theorem for
convergence in probability \cite[Theorem 2.3(ii)]{vanderVaart-2000-AS},%
\begin{equation}
X_{1,N}\rightarrow U_{1}\text{ in }L^{4}\Longrightarrow K_{N}%
(X)\xrightarrow{\mathrm{P}}K\text{.} \label{eq:convergence-gain}%
\end{equation}
So if $X_{1,N}\rightarrow U_{1}$ in $L^{4}$, we have also $X_{1,N}%
\xrightarrow{\mathrm{P}}U_{1}$, and from (\ref{eq:kf-gain}) and
(\ref{eq:def-Ua}), and the standard properties of convergence in probability,
for each $k=1,\ldots,N$,%
\[
X_{1,N}\rightarrow U_{1}\text{ in }L^{4}\Longrightarrow X_{1,N}^{a}%
=X_{1}-K_{N}\left(  D_{1}-HX_{1}\right)  \xrightarrow{\mathrm{P}}U^{a}%
_{1}=U_{1}-K\left(  D-HU_{1}\right)  .
\]

Now we leverage convergence in probability to convergence in $L^{p}$, which
will complete the induction step. By Lemma \ref{eq:lem-uniform-integrability}
(uniform integrability), $X_{1,N}^{a}\xrightarrow{\mathrm{P}}U_{1}^{a}$ as
$N\rightarrow\infty$ and $\left\Vert X_{1,N}^{a}\right\Vert _{q}$ bounded
independently of $N$ implies that $X_{1,N}^{a}\rightarrow U_{1}^{a}$ in
$L^{p}$ for all $1\leq p<q$.
\end{proof}

\begin{remark}
[Exchangeability and identical distributions]We use that $\left[  X_{k}%
^{a};U_{k}^{a}\right]  $, $k=1,\ldots,N$ are identically distributed in
several places. In the $L^{p}$ estimates, we would otherwise need estimates of
the form $\left\Vert X_{k}^{a}-U_{k}^{a}\right\Vert _{p}\leq a_{N}%
\rightarrow0$ as $N\rightarrow\infty$, uniformly in $k$; however, the uniform
integrability estimate is nonconstructive and it does not give an explicit
bound on the $L^{p}$ convergence, let alone uniform in $k$. Using that
$X_{k}^{a}$, $k=1,\ldots,N$ are identically distributed and $U_{k}^{a}$,
$k=1,\ldots,N$ are identically distributed does not seem to be enough, because
this says nothing about the differences $X_{k}^{a}-U_{k}^{a}$. We do not need
exchangeability, we use it only to show that $\left[  X_{k}^{a};U_{k}%
^{a}\right]  $, $k=1,\ldots,N$ are identically distributed.
\end{remark}

\begin{remark}
[Extension to Hilbert space]The arguments carry over to gaussian measures on a
Hilbert spaces, with the data covariance $R$ bounded below (so that $R^{-1}$
is a bounded operator), except in the case when the data space is infinite
dimensional because then $R$ cannot be the covariance of a~probability
measure. See~\cite{Kasanicky-2017-EKF,Kasanicky-2017-WBD} for proofs. The use
of a data covariance bounded below and unbounded above was suggested in
particle filters~\cite{Robinson-2018-IPF}.
\end{remark}

\section{The ETKF}

We now summarize the \textquotedblleft Ensemble Transform Kalman
Filter\textquotedblright\ (ETKF) \cite{Bishop-2001-ASE,Wei-2006-ETK}. The ETKF
is a type of square root filter \cite{Livings-2008-UES,Tippett-2003-ESR}. We
follow the descripton of ETKF from \cite[Sec. 2.2.1 -- 2.3.3]{Hunt-2007-EDA}. .

\subsection{Base scheme}

Given an ensemble
\[
X=[X_{1},\ldots,X_{N}],
\]
define the ensemble mean, matrix of deviates, the sample covariance
\begin{equation}
\overline{X}=\frac{1}{N}\sum_{i=1}^{N}X_{i}\quad A=\left[  X_{1}-\overline
{X},\ldots,X_{N}-\overline{X}\right]  =X\left(  I-\frac{ee^{\mathrm{T}}}%
{N}\right)  ,\label{eq:def-A}%
\end{equation}%
\begin{equation}
P=\frac{1}{N-1}AA^{\mathrm{T}},\label{eq:def-P}%
\end{equation}
and the ensemble space%
\[
S=\operatorname*{Range}A=\operatorname*{Range}P,
\]
since for any matrix,%
\[
\operatorname*{Range}AA^{\mathrm{T}}=\operatorname*{Range}A,
\]
see (\ref{eq:AAT}). For a given data vector $d$, the ETKF first solves the
problem,%
\begin{equation}
\left\Vert \delta x\right\Vert _{P^{-1}}^{2}+\left\Vert d-\mathcal{H}\left(
\overline{X}+\delta x\right)  \right\Vert _{R^{-1}}^{2}\rightarrow\min_{\delta
x\in S},\label{eq:LETKF-obj-x}%
\end{equation}
and set the analysis mean
\begin{equation}
\overline{X}^{a}=\overline{X}+\delta x.\label{eq:letkf-analysis-mean}%
\end{equation}
Here,%
\[
\left\Vert \delta x\right\Vert _{P^{-1}}^{2}=\left\langle P^{-1}\delta
x,\delta x\right\rangle ,
\]
which is well defined if $A\ $is full rank, since $\delta x\in
S=\operatorname*{Range}P=\operatorname{dom}P^{-1}$.

Since $\delta x\in S=\operatorname*{Range}A$, we can look for $\delta x$ in
the form $\delta x=Aw$. Substituting in (\ref{eq:letkf-analysis-mean}), we
have%
\begin{equation}
\overline{X}^{a}=\overline{X}+Aw \label{eq:letkf-subst}%
\end{equation}
and that $P^{-1}=\left(  \frac{AA^{\mathrm{T}}}{N-1}\right)  ^{-1}=\left(
N-1\right)  \left(  AA^{\mathrm{T}}\right)  ^{-1}$, (\ref{eq:LETKF-obj-x})
becomes%
\begin{align}
&  \left\Vert Aw\right\Vert _{P^{-1}}^{2}+\left\Vert d-\mathcal{H}\left(
\overline{X}+Aw\right)  \right\Vert _{R^{-1}}^{2}\nonumber\\
&  =\left(  Aw\right)  ^{T}\left(  N-1\right)  \left(  AA^{\mathrm{T}}\right)
^{-1}Aw+\left\Vert d-\mathcal{H}\left(  \overline{X}+Aw\right)  \right\Vert
_{R^{-1}}^{2}\\
&  =\left(  N-1\right)  w^{T}\underbrace{A\left(  AA^{\mathrm{T}}\right)
^{-1}A}_{\Pi}w+\left\Vert d-\mathcal{H}\left(  \overline{X}+Aw\right)
\right\Vert _{R^{-1}}^{2}\\
&  =\left(  N-1\right)  w^{\mathrm{T}}\Pi w+\left\Vert d-\mathcal{H}\left(
\overline{X}+Aw\right)  \right\Vert _{R^{-1}}^{2}\rightarrow\min_{w}
\label{eq:letkf-min}%
\end{align}
where
\[
\Pi=A^{\mathrm{T}}\left(  AA^{\mathrm{T}}\right)  ^{-1}A
\]
is orthogonal projection, and
\[
\ker\Pi=\ker A\perp\operatorname*{Range}\Pi=\operatorname*{Range}A^{\top}.
\]
Using the orthogonal decomposition
\[
w^{\mathrm{T}}w=\left(  \Pi w\right)  ^{\mathrm{T}}\Pi w+\left(  \left(
I-\Pi\right)  w\right)  ^{\mathrm{T}}\left(  I-\Pi\right)  w\geq\left(  \Pi
w\right)  ^{\mathrm{T}}\Pi w,
\]
(\ref{eq:letkf-min}) is in turn equivalent to%
\begin{equation}
\left(  N-1\right)  w^{\mathrm{T}}w+\left\Vert d-\mathcal{H}\left(
\overline{X}+Aw\right)  \right\Vert _{R^{-1}}^{2}\rightarrow\min_{w},
\label{eq:letkf-min-w}%
\end{equation}
since the minimum in (\ref{eq:letkf-min-w}) occurs at $w$ orthogonal to
$\operatorname*{Null}\left(  A\right)  $ and then the objective functions in
(\ref{eq:letkf-min-w}) and (\ref{eq:letkf-min}) equal.

\subsection{Formulation in the ensemble space}

Define%
\begin{equation}
Y_{i}=\mathcal{H}\left(  X_{i}\right)  ,\quad\overline{Y}=\frac{1}{N}%
\sum_{i=1}^{N}Y_{i},\quad B=\left[  Y_{1}-\overline{Y},\ldots,Y_{N}%
-\overline{Y}\right]  \label{eq:def-B}%
\end{equation}
and consider the approximation%
\begin{equation}
\mathcal{H}\left(  \overline{X}+Aw\right)  \approx\overline{Y}+Bw,
\label{eq:letkf-approx}%
\end{equation}
which is exact when the map $\mathcal{H}$ is affine:\ if $\mathcal{H}\left(
x\right)  =Hx+h$, then%
\begin{align*}
\mathcal{H}\left(  \overline{X}\right)   &  =H\left(  \frac{1}{N}\sum
_{i=1}^{N}X_{i}\right)  +h=\frac{1}{N}\sum_{i=1}^{N}HX_{i}+h\\
&  =\frac{1}{N}\sum_{i=1}^{N}\left(  HX_{i}+h\right)  =\frac{1}{N}\sum
_{i=1}^{N}\mathcal{H}\left(  X_{i}\right)  =\frac{1}{N}\sum_{i=1}^{N}%
Y_{i}=\overline{Y,}%
\end{align*}
thus%
\begin{align*}
\mathcal{H}\left(  \overline{X}+Aw\right)   &  =H\left(  \overline{X}%
+\sum_{i=1}^{N}\left(  X_{i}-\overline{X}\right)  w_{i}\right)  +h\\
&  =H\overline{X}+h+\sum_{i=1}^{N}\left(  HX_{i}+h-H\overline{X}-h\right)
w_{i}\\
&  =\overline{Y}+\sum_{i=1}^{N}\left(  Y_{i}-\overline{Y}\right)  w_{i}\\
&  =\overline{Y}+Bw.
\end{align*}

Now, (\ref{eq:letkf-min-w}) can we written using only quantities in the
observation space,%
\begin{equation}
\left(  N-1\right)  w^{\mathrm{T}}w+\underbrace{\left\Vert d-\left(
\overline{Y}+Bw\right)  \right\Vert _{R^{-1}}^{2}}_{\left\Vert d-\mathcal{H}%
\left(  \overline{X}+Aw\right)  \right\Vert _{R^{-1}}^{2}}\rightarrow\min_{w}.
\label{eq:etkf-min-obs}%
\end{equation}
Since (\ref{eq:etkf-min-obs}) is the same as the least squares form
(\ref{eq:kf-opt})\footnote{$\left(  v-\mu\right)  ^{\mathrm{T}}Q^{-1}\left(
v-\mu\right)  +\left\Vert Hv-d\right\Vert _{R^{-1}}^{2}\rightarrow\min_{v}$}
of the Kalman filter, with $\left(  N-1\right)  I$ for the inverse forecast
covariance $Q^{-1}$, $B$ playing the role of the observation operator $H$,
$d-\overline{Y}$ for data $d$, and $\mu=0$, (\ref{eq:osi-cov}%
)\footnote{$Q^{-1}+H^{\mathrm{T}}R^{-1}H=\left(  Q^{a}\right)  ^{-1}$} and
(\ref{eq:osi-mean})\footnote{$Q^{-1}\mu+H^{\mathrm{T}}R^{-1}d=\left(
Q^{a}\right)  ^{-1}\mu^{a}$} give the covariance $\widetilde{Q}^{a}$ and the
mean $w^{a}$ for $w$ as%
\begin{align}
\widetilde{Q}^{a}  &  =\left(  \left(  N-1\right)  I+B^{\mathrm{T}}%
R^{-1}B\right)  ^{-1}\label{eq:letkf-Qa}\\
w^{a}  &  =\widetilde{Q}^{a}B^{\mathrm{T}}R^{-1}\left(  d-\overline{Y}\right)
. \label{eq:letkf-wa}%
\end{align}
From the substitution (\ref{eq:letkf-subst}), we then have the analysis mean
and covariance%
\begin{equation}
\overline{X}^{a}=\overline{X}+Aw^{a},\quad Q^{a}=A\widetilde{Q}^{a}%
A^{\mathrm{T}}. \label{eq:letkf-mean-cov}%
\end{equation}
To create the analysis ensemble with this mean and covariance $\frac
{A^{a}A^{aT}}{N-1}=Q^{a}$, we choose%
\[
A^{a}=AW,\quad\frac{AWW^{T}A^{T}}{N-1}=Q^{a}=A\widetilde{Q}^{a}A^{\mathrm{T}}%
\]
i.e.,%
\begin{equation}
X_{i}^{a}=\overline{X}^{a}+A_{i}^{a},\quad A^{a}=AW,\quad WW^{\mathrm{T}%
}=\left(  N-1\right)  \widetilde{Q}^{a}=\left(  I+\frac{B^{\mathrm{T}}R^{-1}%
B}{N-1}\right)  ^{-1}. \label{eq:letkf-final}%
\end{equation}

\begin{remark}
Note that $B\approx O\left(  N\right)  $, $B^{\mathrm{T}}R^{-1}B\approx
O\left(  N^{2}\right)  $, so $\left(  I+\frac{B^{\mathrm{T}}R^{-1}B}%
{N-1}\right)  ^{-1}\approx O\left(  N^{-1}\right)  $, and since also $A\approx
O\left(  N\right)  $, we have $AW\approx O\left(  1\right)  $ so it has a
chance to converge as $N\rightarrow\infty$.
\end{remark}

\begin{theorem}
[Unbiased square root filter \cite{Livings-2008-UES}]\bigskip The sample mean
of the ensemble $X^{a}$ from (\ref{eq:letkf-final})\ is $\overline{X}^{a}$ and
its sample covariance is ${Q}^{a}$. In particular, $A^{a}{A^{a}}^{T}$ depends
only on $AA^{T}$ and not on $A$ itself.
\end{theorem}

\begin{proof}
To show that $\overline{X}^{a}=\frac{1}{N}\sum_{i=1}^{N}X_{i}^{a}$, it is
sufficient to show that $A^{a}e=0$, where $e$ is vector length $N$ of all
ones. By (\ref{eq:def-B}), $Be=0$, hence $e$ is eigenvector of $\left(
N-1\right)  \widetilde{Q}^{a}$ with eigenvalue $1$. Then the same is true of
$W$, and we have%
\[
A^{a}e=AWe=Ae=0
\]
by the definition of $A$ in (\ref{eq:def-A}). The sample covariance of the
analysis ensemble is%
\begin{align*}
\frac{1}{N-1}\sum_{k=1}^{N}\left(  X_{k}^{a}-\overline{X}^{a}\right)  \left(
X_{k}^{a}-\overline{X}^{a}\right)  ^{\mathrm{T}}  &  =\frac{1}{N-1}%
A^{a}\left(  A^{a}\right)  ^{\mathrm{T}}=\frac{1}{N-1}AWW^{\mathrm{T}%
}A^{\mathrm{T}}\\
&  =\frac{1}{N-1}A\left(  N-1\right)  \widetilde{Q}^{a}A^{\mathrm{T}%
}=A\widetilde{Q}^{a}A^{\mathrm{T}}=Q^{a},
\end{align*}
as desired in (\ref{eq:letkf-mean-cov}).
\end{proof}

\begin{theorem}
Each member of the the analysis\ ensemble $X^{a}$ is a linear combination of
the forecast ensemble $X$.
\end{theorem}

\begin{proof}
Expanding the sample covariance, we can rewrite (\ref{eq:letkf-Qa}%
)--(\ref{eq:letkf-final}) as%
\[
X^{a}=X\frac{ee^{\mathrm{T}}}{N}+X\left(  I-\frac{ee^{\mathrm{T}}}{N}\right)
\left(  w^{a}+\left(  \widetilde{P}^{a}\right)  ^{1/2}\right)  ,
\]
where%
\begin{align}
\widetilde{P}^{a}  &  =\left(  I+\frac{B^{\mathrm{T}}R^{-1}B}{N-1}\right)
^{-1},\quad\label{eq:Pa}\\
B  &  =\left[  \mathcal{H}\left(  X_{1}\right)  ,\ldots,\mathcal{H}\left(
X_{N}\right)  \right]  \left(  I-\frac{ee^{\mathrm{T}}}{N}\right) \nonumber\\
w^{a}  &  =\frac{1}{N-1}\widetilde{P}^{a}B^{\mathrm{T}}R^{-1}\left(
d-\overline{Y}\right) \nonumber
\end{align}
which shows that
\[
X^{a}=XT,\quad T=\frac{ee^{\mathrm{T}}}{N}+\left(  I-\frac{ee^{\mathrm{T}}}%
{N}\right)  \left(  w^{a}+\left(  \widetilde{P}^{a}\right)  ^{1/2}\right)  .
\]

\end{proof}

\begin{remark}
The matrix
\[
\left(  N-1\right)  Q^{a}=\widetilde{P}^{a}=\left(  I+\frac{B^{\mathrm{T}%
}R^{-1}B}{N-1}\right)  ^{-1}%
\]
from (\ref{eq:letkf-final},\ref{eq:Pa}) is size $N$, the number of ensemble
members, and it is not clear what convergence as $N\rightarrow\infty$ may be
taking place. Using the Sherman-Morrison-Woodbury formula (\ref{eq:SMW})
\[
\left(  A+UCV\right)  ^{-1}=A^{-1}-A^{-1}U\left(  C^{-1}+VA^{-1}U\right)
^{-1}VA^{-1},
\]
with $C=R^{-1}/\left(  N-1\right)  $, $U=B^{\mathrm{T}}$, $V=B$, we can write%
\begin{align}
\left(  I+\frac{B^{\mathrm{T}}R^{-1}B}{N-1}\right)  ^{-1}  &  =\left(
I+B^{\mathrm{T}}\frac{R^{-1}}{N-1}B\right)  ^{-1}\nonumber\\
&  =I-B^{\mathrm{T}}\left(  \left(  N-1\right)  R+BB^{\mathrm{T}}\right)
^{-1}B\nonumber\\
&  =I-B^{\mathrm{T}}\left(  \left(  N-1\right)  \left(  R+\frac{1}%
{N-1}BB^{\mathrm{T}}\right)  \right)  ^{-1}B\nonumber\\
&  =I-\frac{1}{N-1}B^{\mathrm{T}}\left(  R+S\right)  ^{-1}B,
\label{eq:after-smw}%
\end{align}
where $S$ is the size of the number of data points, and%
\[
S=\frac{BB^{\mathrm{T}}}{N-1}\rightarrow\operatorname*{cov}Y\text{ as
}N\rightarrow\infty
\]
The matrix in (\ref{eq:after-smw}) is a type of regularized orthogonal
projection; when $R\rightarrow0$ and $B$ is full rank, (\ref{eq:after-smw})
becomes%
\[
\lim_{R\rightarrow0}I-\frac{1}{N-1}B^{\mathrm{T}}\left(  R+S\right)
^{-1}B=I-\frac{1}{N-1}B^{\mathrm{T}}\left(  \frac{BB^{\mathrm{T}}}%
{N-1}\right)  ^{-1}B=I-B^{\mathrm{T}}\left(  BB^{\mathrm{T}}\right)
^{-1}B=P_{\operatorname{null}B}%
\]
A more refined argument is needed when the ensemble members do not span the
whole state space.
\end{remark}

\begin{remark}
Analysis of the LETKF in \cite{Kwiatkowski-2015-CSR}, limited to linear model,
is formulated in terms of the transformations of the ensemble mean and
covariance by the model and by the analysis step. The ensemble itself is not used.
\end{remark}

To summarize:

\begin{algorithm}
[ETKF]Given an ensemble
\[
X=[X_{1},\ldots,X_{N}],
\]
observation operator%
\[
\mathcal{H}\left(  x\right)  =d-Hx
\]
and data covariance $R$, compute:
\begin{align}
Y_{i} &  =\mathcal{H}\left(  X_{i}\right)  ,\quad i=1,\ldots,N\quad\\
\overline{Y} &  =\frac{1}{N}\sum_{i=1}^{N}Y_{i},\quad B=\left[  Y_{1}%
-\overline{Y},\ldots,Y_{N}-\overline{Y}\right]
\end{align}%
\begin{equation}
\overline{X}=\frac{1}{N}\sum_{i=1}^{N}X_{i}\quad A=\left[  X_{1}-\overline
{X},\ldots,X_{N}-\overline{X}\right]  =X\left(  I-\frac{ee^{\mathrm{T}}}%
{N}\right)  ,
\end{equation}%
\begin{align}
\widetilde{Q}^{a} &  =\left(  \left(  N-1\right)  I+B^{\mathrm{T}}%
R^{-1}B\right)  ^{-1}\\
w^{a} &  =\widetilde{Q}^{a}B^{\mathrm{T}}R^{-1}\left(  d-\overline{Y}\right)
.
\end{align}%
\begin{equation}
\overline{X}^{a}=\overline{X}+Aw^{a}%
\end{equation}%
\begin{align}
W &  =\left(  \left(  N-1\right)  \widetilde{Q}^{a}\right)  ^{1/2}%
\label{eq:W}\\
A^{a} &  =AW\\
X_{i}^{a} &  =\overline{X}^{a}+A_{i}^{a}\quad i=1,\ldots,N
\end{align}
Check:%
\[
\operatorname*{mean}(X^{a})=\overline{X}^{a}%
\]%
\[
\operatorname*{cov}(X^{a})=A\widetilde{Q}^{a}A^{\mathrm{T}}%
\]
exactly. Statistically, with error $O\left(  N^{-1/2}\right)  $,
$\operatorname*{mean}(X^{a})$ and $\operatorname*{cov}(X^{a})$ should be the
same as from EnKF.
\end{algorithm}

\subsection{Scalar data}

The L\ in LETKF stands for local. The first step in localization is to
assimilate data for one point at a time. In the case of scalar data, $h_{0}\in
R,\quad h_{1}\in R^{n},$
\[
\mathcal{H}\left(  X\right)  =h_{0}+h_{1}^{\top}X.
\]
for all $X$. Then, by direct computation, (\ref{eq:letkf-approx}) becomes%
\[
\mathcal{H}\left(  \overline{X}+Aw\right)  =\mathcal{H}\left(  \overline
{X}\right)  +h_{1}^{\top}Aw
\]%
\[
Y_{i}=\mathcal{H}\left(  X_{i}\right)  =h_{0}+h_{1}^{\top}X_{i}%
\]%
\begin{equation}
\overline{Y}=\frac{1}{N}\sum_{i=1}^{N}Y_{i}=\frac{1}{N}\sum\nolimits_{i=1}%
^{N}\left(  h_{0}+h_{1}^{\top}X_{i}\right)  =h_{0}+h_{1}^{\top}\overline
{X}=\mathcal{H}\left(  \overline{X}\right)  \label{eq:avg2avg}%
\end{equation}%
\[
Y_{i}-\overline{Y}=\mathcal{H}\left(  X_{i}\right)  -\mathcal{H}\left(
\overline{X}\right)  =h_{0}+h_{1}^{\top}X_{i}-\left(  h_{0}+h_{1}^{\top
}\overline{X}\right)  =h_{1}^{\top}\left(  X_{i}-\overline{X}\right)
\]%
\[
\sum\nolimits_{i=1}^{N}\left(  Y_{i}-\overline{Y}\right)  w_{i}=\sum
\nolimits_{i=1}^{N}h_{1}^{\top}\left(  X_{i}-\overline{X}\right)  w_{i}%
\]%
\begin{equation}
Bw=h_{1}^{\top}Aw. \label{eq:BishA}%
\end{equation}
From (\ref{eq:avg2avg})\ and (\ref{eq:BishA}), we can conclude that
\[
\mathcal{H}\left(  \overline{X}+Aw\right)  =\mathcal{H}\left(  \overline
{X}\right)  +h_{1}^{\top}Aw=\overline{Y}+Bw,
\]
which is (\ref{eq:letkf-approx})\ satisfied exactly instead of only
approximately.\addcontentsline{toc}{section}{References}
\bibliographystyle{siam}
\bibliography{../../references/other,../../references/geo}

\begin{thebibliography}{10}

\bibitem{Araujo-1980-CLT}
{\sc A.~Araujo and E.~Gin{\'e}}, {\em The central limit theorem for real and
  {B}anach valued random variables}, Wiley Series in Probability and
  Mathematical Statistics, John Wiley \& Sons, New York-Chichester-Brisbane,
  1980.

\bibitem{Babuska-2002-EES}
{\sc I.~Babu{\v{s}}ka and P.~Chatzipantelidis}, {\em On solving elliptic
  stochastic partial differential equations}, Comput. Methods Appl. Mech.
  Engrg., 191 (2002), pp.~4093--4122.

\bibitem{Babuska-2004-GFE}
{\sc I.~Babu{\v{s}}ka, R.~Tempone, and G.~E. Zouraris}, {\em {G}alerkin finite
  element approximations of stochastic elliptic partial differential
  equations}, SIAM J. Numer. Anal., 42 (2004), pp.~800--825.

\bibitem{Balakrishnan-1976-AFA}
{\sc A.~V. Balakrishnan}, {\em Applied functional analysis}, Springer-Verlag,
  New York, 1976.

\bibitem{Beezley-2009-HDA}
{\sc J.~D. Beezley}, {\em High-Dimensional Data Assimilation and Morphing
  Ensemble {K}alman Filters with Applications in Wildfire Modeling}, PhD
  thesis, University of Colorado Denver, 2009.

\bibitem{Bengtsson-2008-CRC}
{\sc T.~Bengtsson, P.~Bickel, and B.~Li}, {\em Curse-of-dimensionality
  revisited: collapse of the particle filter in very large scale systems}, in
  Probability and statistics: essays in honor of {D}avid {A}. {F}reedman,
  vol.~2 of Inst. Math. Stat. Collect., Inst. Math. Statist., Beachwood, OH,
  2008, pp.~316--334.

\bibitem{Billingsley-1995-PM}
{\sc P.~Billingsley}, {\em Probability and measure}, John Wiley \& Sons Inc.,
  New York, third~ed., 1995.

\bibitem{Bishop-2001-ASE}
{\sc C.~H. Bishop, B.~J. Etherton, and S.~J. Majumdar}, {\em Adaptive sampling
  with the ensemble transform {K}alman filter. {P}art {I}: {T}heoretical
  aspects}, Monthly Weather Review, 129 (2001), pp.~420--436.

\bibitem{Bogachev-1998-GM}
{\sc V.~I. Bogachev}, {\em Gaussian measures}, Mathematical Surveys and
  Monographs, Vol. 62, American Mathematical Society, Providence, RI, 1998.

\bibitem{Burgers-1998-ASE}
{\sc G.~Burgers, P.~J. van Leeuwen, and G.~Evensen}, {\em Analysis scheme in
  the ensemble {K}alman filter}, Monthly Weather Review, 126 (1998),
  pp.~1719--1724.

\bibitem{Chow-1988-PT}
{\sc Y.~S. Chow and H.~Teicher}, {\em Probability theory. {I}ndependence,
  interchangeability, martingales}, Springer-Verlag, New York, second~ed.,
  1988.

\bibitem{Cotter-2009-BIP}
{\sc S.~L. Cotter, M.~Dashti, J.~C. Robinson, and A.~M. Stuart}, {\em Bayesian
  inverse problems for functions and applications to fluid mechanics}, Inverse
  Problems, 25 (2009), pp.~115008, 43.

\bibitem{Cressie-1993-SSD}
{\sc N.~A.~C. Cressie}, {\em Statistics for Spatial Data}, John Wiley \& Sons
  Inc., New York, 1993.

\bibitem{Cupidon-2007-DMA}
{\sc J.~Cupidon, D.~S. Gilliam, R.~Eubank, and F.~Ruymgaart}, {\em The delta
  method for analytic functions of random operators with application to
  functional data}, Bernoulli, 13 (2007), pp.~1179--1194.

\bibitem{DaPrato-2006-IIA}
{\sc G.~Da~Prato}, {\em An introduction to infinite-dimensional analysis},
  Springer-Verlag, Berlin, 2006.

\bibitem{DaPrato-1992-SEI}
{\sc G.~Da~Prato and J.~Zabczyk}, {\em Stochastic equations in infinite
  dimensions}, vol.~44 of Encyclopedia of Mathematics and its Applications,
  Cambridge University Press, Cambridge, 1992.

\bibitem{Dauxois-1982-ATP}
{\sc J.~Dauxois, A.~Pousse, and Y.~Romain}, {\em Asymptotic theory for the
  principal component analysis of a vector random function: some applications
  to statistical inference}, J. Multivariate Anal., 12 (1982), pp.~136--154.

\bibitem{Evensen-2009-DAE}
{\sc G.~Evensen}, {\em Data Assimilation: The Ensemble {K}alman Filter},
  Springer, second~ed., 2009.

\bibitem{Furrer-2007-EHP}
{\sc R.~Furrer and T.~Bengtsson}, {\em Estimation of high-dimensional prior and
  posterior covariance matrices in {K}alman filter variants}, Journal of
  Multivariate Analysis, 98 (2007), pp.~227 -- 255.

\bibitem{Shi-xin-2002-CTp}
{\sc S.-x. Gan}, {\em Characterization of type {$p$} {B}anach spaces by the
  weak law of large numbers}, Wuhan Univ. J. Nat. Sci., 7 (2002), pp.~14--19.

\bibitem{Ganis-2008-SCM}
{\sc B.~Ganis, H.~Klie, M.~F. Wheeler, T.~Wildey, I.~Yotov, and D.~Zhang}, {\em
  Stochastic collocation and mixed finite elements for flow in porous media},
  Comput. Methods Appl. Mech. Engrg., 197 (2008), pp.~3547--3559.

\bibitem{Hida-2008-LWN}
{\sc T.~Hida and S.~Si}, {\em Lectures on white noise functionals}, World
  Scientific Publishing Co. Pte. Ltd., Hackensack, NJ, 2008.

\bibitem{Hoffmann-Jorgensen-1974-SIB}
{\sc J.~Hoffmann-J{\o}rgensen}, {\em Sums of independent {B}anach space valued
  random variables}, Studia Math., 52 (1974), pp.~159--186.

\bibitem{Hunt-2007-EDA}
{\sc B.~R. Hunt, E.~J. Kostelich, and I.~Szunyogh}, {\em Efficient data
  assimilation for spatiotemporal chaos: {A} local ensemble transform {K}alman
  filter}, Physica D: Nonlinear Phenomena, 230 (2007), pp.~112--126.

\bibitem{Jacod-2003-PE}
{\sc J.~Jacod and P.~Protter}, {\em Probability essentials}, Universitext,
  Springer-Verlag, Berlin, second~ed., 2003.

\bibitem{Kallenberg-2002-FMP}
{\sc O.~Kallenberg}, {\em Foundations of modern probability}, Probability and
  its Applications (New York), Springer-Verlag, New York, second~ed., 2002.

\bibitem{Kalman-1960-NAL}
{\sc R.~E. Kalman}, {\em A new approach to linear filtering and prediction
  problems}, Transactions of the ASME -- Journal of Basic Engineering, Series
  D, 82 (1960), pp.~35--45.

\bibitem{Kalnay-2003-AMD}
{\sc E.~Kalnay}, {\em Atmospheric Modeling, Data Assimilation and
  Predictability}, Cambridge University Press, 2003.

\bibitem{Kasanicky-2017-EKF}
{\sc I.~Kasanick\'{y}}, {\em Ensemble {K}alman filter on high and infinite
  dimensional spaces}, PhD thesis, Department of Probability and Mathematical
  Statistics, Faculty of Mathematics and Physics, Charles University, Prague,
  Czech Republic, 2017.
\newblock \url{http://hdl.handle.net/20.500.11956/86462}.

\bibitem{Kasanicky-2017-WBD}
{\sc I.~Kasanick\'y and J.~Mandel}, {\em On well-posedness of {B}ayesian data
  assimilation and inverse problems in {H}ilbert space}.
\newblock arXiv:1701.08298, 2017.

\bibitem{Kitanidis-1999-GCF}
{\sc P.~K. Kitanidis}, {\em Generalized covariance functions associated with
  the {L}aplace equation and their use in interpolation and inverse problems},
  {Water Resour. Res.}, {35} ({1999}), pp.~{1361--1367}.

\bibitem{Kreyszig-1989-IFA}
{\sc E.~Kreyszig}, {\em Introductory functional analysis with applications},
  Wiley Classics Library, John Wiley \& Sons Inc., New York, 1989.

\bibitem{Kuo-1975-GMB}
{\sc H.~H. Kuo}, {\em Gaussian measures in {B}anach spaces}, Lecture Notes in
  Mathematics, Vol. 463, Springer-Verlag, Berlin, 1975.

\bibitem{Kuo-1996-WND}
{\sc H.-H. Kuo}, {\em White noise distribution theory}, Probability and
  Stochastics Series, CRC Press, Boca Raton, FL, 1996.

\bibitem{Kwiatkowski-2015-CSR}
{\sc E.~Kwiatkowski and J.~Mandel}, {\em Convergence of the square root
  ensemble {K}alman filter in the large ensemble limit}, SIAM/ASA Journal on
  Uncertainty Quantification, 3 (2015), pp.~1--17.

\bibitem{Lang-1993-RFA}
{\sc S.~Lang}, {\em Real and functional analysis}, vol.~142 of Graduate Texts
  in Mathematics, Springer-Verlag, New York, third~ed., 1993.

\bibitem{Lax-2002-FA}
{\sc P.~D. Lax}, {\em Functional analysis}, Pure and Applied Mathematics (New
  York), Wiley-Interscience [John Wiley \& Sons], New York, 2002.

\bibitem{LeGland-2009-LSA}
{\sc F.~{Le~Gland}, V.~Monbet, and V.-D. Tran}, {\em Large sample asymptotics
  for the ensemble {K}alman filter}.
\newblock INRIA Report 7014, August 2009.

\bibitem{LeGland-2011-LSA}
{\sc F.~{Le~Gland}, V.~Monbet, and V.-D. Tran}, {\em Large sample asymptotics
  for the ensemble {K}alman filter}, in The Oxford Handbook of Nonlinear
  Filtering, D.~Crisan and B.~Rozovski\v{\i}, eds., Oxford University Press,
  2011, pp.~598--631.

\bibitem{Ledoux-1991-PBS}
{\sc M.~Ledoux and M.~Talagrand}, {\em Probability in {B}anach spaces},
  Ergebnisse der Mathematik und ihrer Grenzgebiete (3), Vol. 23,
  Springer-Verlag, Berlin, 1991.

\bibitem{Livings-2008-UES}
{\sc D.~M. Livings, S.~L. Dance, and N.~K. Nichols}, {\em Unbiased ensemble
  square root filters}, Phys. D, 237 (2008), pp.~1021--1028.

\bibitem{Loeve-1963-PT}
{\sc M.~Lo{\`e}ve}, {\em Probability theory}, Third edition, D. Van Nostrand
  Co., Inc., Princeton, N.J.-Toronto, Ont.-London, 1963.

\bibitem{Mandel-2009-CEK}
{\sc J.~Mandel, L.~Cobb, and J.~D. Beezley}, {\em On the convergence of the
  ensemble {K}alman filter}.
\newblock arXiv:0901.2951, January 2009.

\bibitem{Mandel-2011-CEK}
\leavevmode\vrule height 2pt depth -1.6pt width 23pt, {\em On the convergence
  of the ensemble {K}alman filter}, Applications of Mathematics, 56 (2011),
  pp.~533--541.

\bibitem{Marcinkiewicz-1964-CP}
{\sc J.~Marcinkiewicz}, {\em Collected papers}, Edited by Antoni Zygmund. With
  the collaboration of Stanislaw Lojasiewicz, Julian Musielak, Kazimierz
  Urbanik and Antoni Wiweger. Instytut Matematyczny Polskiej Akademii Nauk,
  Pa\'nstwowe Wydawnictwo Naukowe, Warsaw, 1964.

\bibitem{Marcinkiewicz-1937-FI}
{\sc J.~Marcinkiewicz and A.~Zygmund}, {\em Sur les foncions
  ind{\'e}pendantes}, Fund. Math., 29 (1937), pp.~60--90.
\newblock Reprinted in \cite{Marcinkiewicz-1964-CP}, pp. 233--259.

\bibitem{Marcus-1981-RFS}
{\sc M.~B. Marcus and G.~Pisier}, {\em Random {F}ourier series with
  applications to harmonic analysis}, vol.~101 of Annals of Mathematics
  Studies, Princeton University Press, Princeton, N.J., 1981.

\bibitem{McGee-1985-DKF}
{\sc L.~A. McGee and S.~F. Schmidt}, {\em Discovery of the {K}alman filter as a
  practical tool for aerospace and industry}.
\newblock NASA Technical Memorandum TM-86847, 1985.
\newblock \url{https://ntrs.nasa.gov/citations/19860003843}, accessed December
  3, 2021.

\bibitem{Mirouze-2010-RCF}
{\sc I.~Mirouze and A.~T. Weaver}, {\em Representation of correlation functions
  in variational assimilation using an implicit diffusion operator}, Quarterly
  Journal of the Royal Meteorological Society, {136} (2010), pp.~1421--1443.

\bibitem{Ocana-1999-FPC}
{\sc F.~A. Oca{\~n}a, A.~M. Aguilera, and M.~J. Valderrama}, {\em Functional
  principal components analysis by choice of norm}, J. Multivariate Anal., 71
  (1999), pp.~262--276.

\bibitem{Pettis-1938-IVS}
{\sc B.~J. Pettis}, {\em On integration in vector spaces}, Trans. Amer. Math.
  Soc., 44 (1938), pp.~277--304.

\bibitem{Quang-2006-WLL}
{\sc N.~V. Quang and L.~H. Son}, {\em On the weak law of large numbers for
  sequences of {B}anach space valued random elements}, Bull. Korean Math. Soc.,
  43 (2006), pp.~551--558.

\bibitem{Robinson-2018-IPF}
{\sc G.~Robinson, I.~Grooms, and W.~Kleiber}, {\em Improving particle filter
  performance by smoothing observations}, Monthly Weather Review, 146 (2018),
  pp.~2433--2446.

\bibitem{Sepanski-2005-WLL}
{\sc S.~Sepanski and Z.~Pan}, {\em {A Weak Law of Large Numbers for the Sample
  Covariance Matrix}}, Electronic Communications in Probability, 5 (2005),
  pp.~73--76.

\bibitem{Stuart-2010-IPB}
{\sc A.~M. Stuart}, {\em Inverse problems: A {B}ayesian perspective}, Acta
  Numer., 19 (2010), pp.~451--559.

\bibitem{Tippett-2003-ESR}
{\sc M.~K. Tippett, J.~L. Anderson, C.~H. Bishop, T.~M. Hamill, and J.~S.
  Whitaker}, {\em Ensemble square root filters}, Monthly Weather Review, 131
  (2003), pp.~1485--1490.

\bibitem{vanderVaart-2000-AS}
{\sc A.~W. {V}an~der Vaart}, {\em Asymptotic Statistics}, Cambridge University
  Press, 2000.

\bibitem{Weaver-2003-CBC}
{\sc A.~T. Weaver and S.~Ricci}, {\em Constructing a background-error
  correlation model using generalized diffusion operators}, in Proceedings of
  the ECMWF Seminar Series on Recent Developments in Atmospheric and Ocean Data
  Assimilation, Reading, UK, 2003, ECMWF, pp.~8--12.

\bibitem{Wei-2006-ETK}
{\sc M.~Wei, Z.~Toth, R.~Wobus, Y.~Zhu, C.~H. Bishop, and X.~Wang}, {\em
  Ensemble transform {K}alman filter-based ensemble perturbations in an
  operational global prediction system at ncep}, Tellus A: Dynamic Meteorology
  and Oceanography, 58 (2006), pp.~28--44.

\bibitem{Woyczynski-1978-GMB}
{\sc W.~A. Woyczy{\'n}ski}, {\em Geometry and martingales in {B}anach spaces.
  {II}. {I}ndependent increments}, in Probability on {B}anach spaces,
  J.~Kuelbs, ed., vol.~4 of Adv. Probab. Related Topics, Dekker, New York,
  1978, pp.~267--517.

\bibitem{Woyczynski-1980-MLL}
\leavevmode\vrule height 2pt depth -1.6pt width 23pt, {\em On
  {M}arcinkiewicz-{Z}ygmund laws of large numbers in {B}anach spaces and
  related rates of convergence}, Probab. Math. Statist., 1 (1980),
  pp.~117--131.

\bibitem{Xiu-2010-NMS}
{\sc D.~Xiu}, {\em Numerical methods for stochastic computations}, Princeton
  University Press, Princeton, NJ, 2010.
\newblock A spectral method approach.

\end{thebibliography}

\appendix

\section{Other approaches to convergence of sample covariance}

The proof of convergence of sample covariance in \cite[Lemma 3.3]%
{LeGland-2009-LSA} (see also Lemma \ref{lem:LeGland-cov} below) is similar to
the proof of Lemma \ref{lem:LpLLN-sample-cov}, but it proceeds separately on
each entry of the covariance matrix and it does not yield an explicit,
dimension independent bound. For other types of convergence of sample
covariance and convergence of its eigenvalues, see \cite{Dauxois-1982-ATP},
and used in \cite{Cupidon-2007-DMA} for the convergence of functions of the
sample covariance. The relevant field is called \textquotedblleft functional
principal component analysis (PCA)\textquotedblright, e.g.,
\cite{Cupidon-2007-DMA,Dauxois-1982-ATP,Ocana-1999-FPC}: functions are
considered as elements of a Hilbert space and PCA, the dominant application,
is based on eigenvalues and eigenvectors of the sample covariance.

Consider other possible arguments, first in finite dimension. Denote the
entries of a vector $X\in\mathbb{R}^{m}$ by $\left[  X\right]  _{i},$ the
entries of a matrix $A\in\mathbb{R}^{m\times m}$ by $\left[  A\right]  _{ij}$,
and note that%
\[
\left[  X\otimes Y\right]  _{ij}=\left[  X\right]  _{i}\left[  Y\right]  _{j}%
\]
for any $X$, $Y\in\mathbb{R}^{m}$.

The first argument follows \cite[proof of Lemma 3]{Mandel-2009-CEK}, with
additional details and using the framework here.

\begin{lemma}
\label{lem:MCB-cov}If $U_{k}\in L^{4}\left(  \Omega,\mathbb{R}^{m}\right)  $,
$k=1,...$ are i.i.d., then $C_{n}(X_{k})\Longrightarrow\operatorname*{Cov}%
\left(  X_{1}\right)  $, $n\rightarrow\infty$.
\end{lemma}

\begin{proof}
The proof follows \cite[proof of Lemma 3]{Mandel-2009-CEK}, with more details.
Since $U_{k}\in L^{4}\left(  \Omega,\mathbb{R}^{m}\right)  $, we have%
\begin{align*}
\left\Vert \left[  U_{1}\otimes U_{1}\right]  _{ij}\right\Vert _{2}^{2}  &
\leq\left\Vert U_{1}\otimes U_{1}\right\Vert _{2}^{2}=E\left(  \left\vert
U_{1}\otimes U_{1}\right\vert ^{2}\right) \\
&  =E\left(  \left\vert U_{1}\right\vert ^{4}\right)  =\left\Vert
U_{k}\right\Vert _{4}^{4},
\end{align*}
hence by the weak law of large numbers (Lemma \ref{lem:WLLN}),%
\begin{equation}
\left[  E_{n}\left(  U_{k}\otimes U_{k}\right)  \right]  _{ij}\Rightarrow
\left[  E\left(  U_{1}\otimes U_{1}\right)  \right]  _{ij}.
\label{eq:cov-conv-1}%
\end{equation}

Since $U_{1}\in L^{4}\left(  \Omega,\mathbb{R}^{m}\right)  $, we have
from$\left\vert \left[  U_{1}\right]  _{j}\right\vert \leq\left\vert
U_{1}\right\vert $that%
\[
\left\Vert \left[  U_{1}\right]  _{j}\right\Vert _{2}\leq\left\Vert
U_{1}\right\Vert _{2}\leq\left\Vert U_{1}\right\Vert _{4}%
\]
Then, for each entry separately, we have from the weak law of large numbers
(Lemma \ref{lem:WLLN})%
\[
\lbrack E_{n}\left(  U_{k}\right)  ]_{j}\Rightarrow\lbrack E(U_{1})]_{j},
\]
and by Slutsky's theorem,%
\begin{align}
\left[  E_{n}\left(  U_{k}\right)  \otimes E_{n}\left(  U_{k}\right)  \right]
_{ij}  &  =[E_{n}\left(  U_{k}\right)  ]_{i}[E_{n}\left(  U_{k}\right)
]_{j}\label{eq:cov-conv-2}\\
&  \Rightarrow[E(U_{1})]_{i}[E(U_{1})]_{j}=\left[  E\left(  U_{1}\right)
\otimes E\left(  U_{1}\right)  \right]  _{ij}.\nonumber
\end{align}

From (\ref{eq:cov-conv-1}) and (\ref{eq:cov-conv-2}) by Slutsky's theorem,%
\[
\left[  C_{n}\left(  U_{k}\right)  \right]  _{ij}\Rightarrow\left[
\operatorname*{Cov}\left(  U_{1}\right)  \right]  _{ij}%
\]
for all indices $i$, $j$.
\end{proof}

The next argument follows \cite{LeGland-2009-LSA}, with some additional
details and some simplifications afforded by the present framework.

\begin{lemma}
[{\cite[Lemma 3.3]{LeGland-2009-LSA}}]\label{lem:LeGland-cov}If $U_{k}\in
L^{p}\left(  \Omega,\mathbb{R}^{m}\right)  $, $p\geq2$, are i.i.d., then%
\begin{equation}
\varepsilon_{n}=\left\vert C_{n}\left(  U_{k}\right)  -\operatorname*{Cov}%
\left(  U_{1}\right)  \right\vert \rightarrow0\text{\quad a.s. as
}n\rightarrow\infty, \label{eq:conv-cov-LeGland-as}%
\end{equation}
and%
\begin{equation}
\sup_{n\geq1}\sqrt{n}\left\Vert \varepsilon_{n}\right\Vert _{p}<\infty.
\label{eq:conv-cov-LeGland-Lp}%
\end{equation}

\end{lemma}

\begin{proof}
Without loss of generality, let $E\left(  U_{1}\right)  =0$. Then%
\begin{align*}
C_{n}(U_{k})  &  =E_{n}\left(  U_{k}\otimes U_{k}\right)  -E_{n}\left(
U_{k}\right)  \otimes E_{n}\left(  U_{k}\right) \\
\operatorname*{Cov}\left(  U_{1}\right)   &  =E\left(  U_{1}\otimes
U_{1}\right)
\end{align*}
yield%
\begin{align}
\varepsilon_{n}  &  =\left\vert C_{n}(U_{k})-\operatorname*{Cov}\left(
U_{1}\right)  \right\vert \label{eq:LeGland-eps-by-coord}\\
&  \leq\left\vert E_{n}\left(  U_{k}\otimes U_{k}\right)  -E\left(
U_{1}\otimes U_{1}\right)  \right\vert +\left\vert E_{n}\left(  U_{k}\right)
\right\vert ^{2}\nonumber\\
&  \leq\sum_{i,j=1}^{n}\left\vert E_{n}\left(  \left[  \left(  U_{k}\otimes
U_{k}\right)  \right]  _{ij}\right)  -E\left(  U_{1}\otimes U_{1}\right)
\right\vert +\sum_{i=1}^{n}\left\vert \left[  E_{n}\left(  U_{k}\right)
\right]  _{i}\right\vert ^{2}.\nonumber
\end{align}
Since $p\geq2$, we have $E\left(  \left\vert \left[  U_{1}\otimes
U_{1}\right]  _{ij}\right\vert \right)  <\infty$ and $E\left(  \left\vert
\left[  U_{1}\right]  _{ij}\right\vert \right)  <\infty$, so it follows from
the strong law of large numbers for each entry separately that%
\[
\left[  E_{n}\left(  U_{k}\otimes U_{k}\right)  -E\left(  U_{1}\otimes
U_{1}\right)  \right]  _{ij}\rightarrow0\text{ a.s.,\quad}\left[  E_{n}\left(
U_{k}\right)  \right]  _{i}\rightarrow0\text{ a.s. , }n\rightarrow\infty,
\]
which, together with (\ref{eq:LeGland-eps-by-coord}), concludes the proof of
(\ref{eq:conv-cov-LeGland-as}).

From (\ref{eq:LeGland-eps-by-coord}), by triangle inequality for the
$L^{p}\left(  \Omega,\mathbb{R}\right)  $ norm,%
\begin{equation}
\left\Vert \varepsilon_{n}\right\Vert _{p}\leq\sum_{i,j=1}^{n}\left\Vert
E_{n}\left(  \left[  \left(  U_{k}\otimes U_{k}\right)  \right]  _{ij}\right)
-E\left(  \left[  U_{1}\otimes U_{1}\right]  _{ij}\right)  \right\Vert
_{p}+\sum_{i=1}^{n}\left\Vert \left[  E_{n}\left(  U_{k}\right)  \right]
_{i}\right\Vert _{2p}^{2}, \label{eq:LeGland-eps-pnorm}%
\end{equation}
using also $\left\Vert \left\vert X\right\vert ^{2}\right\Vert _{p}=\left(
E\left(  \left\vert X\right\vert ^{2p}\right)  \right)  ^{1/p}=\left\Vert
X\right\Vert _{2p}^{2}$ with $X=\left[  E_{n}\left(  U_{k}\right)  \right]
_{i}$. By Marcinkiewicz-Zygmund inequality in the form of the $L^{p}$ law of
large numbers (\ref{eq:Lp-LLN}),%
\[
\left\Vert E_{n}\left(  X_{k}\right)  -E\left(  X_{1}\right)  \right\Vert
_{p}\leq\frac{C_{p}}{\sqrt{n}}\left\Vert X_{1}-E\left(  X_{1}\right)
\right\Vert _{p},
\]
applied to each entry separately, (\ref{eq:LeGland-eps-pnorm}) yields%
\[
\left\Vert \varepsilon_{n}\right\Vert _{p}\leq\frac{C_{p}}{\sqrt{n}}%
\sum_{i,j=1}^{n}\left\Vert \left[  \left(  U_{1}\otimes U_{1}\right)  \right]
_{ij}-E\left(  \left[  U_{1}\otimes U_{1}\right]  _{ij}\right)  \right\Vert
_{p}+\frac{C_{p}^{2}}{n}\sum_{i=1}^{n}\left\Vert \left[  U_{1}\right]
_{i}-0\right\Vert _{2p}^{2},
\]
which proves (\ref{eq:conv-cov-LeGland-Lp}).
\end{proof}

Unlike the $L^{2}$ law of large numbers, the approaches to the convergence of
the sample covariance matrix in Lemma \ref{lem:MCB-cov} and Lemma
\ref{lem:LeGland-cov} fail to provide an explicit bound that would work
independently of dimension or carry over to the Hilbert space case. Both
approaches rely on treating one entry of the covariance matrix at a time. The
proof of Lemma \ref{lem:MCB-cov} is based on Slutsky's theorem (following a
suggestion in \cite{Furrer-2007-EHP}), which does not give a constructive
estimate of the speed of convergence. Lemma \ref{lem:LeGland-cov} gives an
estimate that becomes progressively worse with the dimension, because it
relies on the estimate of the operator norm by the sum of all entries.

In the Hilbert space case, the entry-by-entry approach can be immediately
generalized to estimates of $\left\Vert \left\langle u,\left(  C_{n}\left(
U_{k}\right)  -\operatorname*{Cov}\left(  U_{1}\right)  \right)
v\right\rangle \right\Vert _{p}$ for arbitrary $u,v\in H$ (in the
finite-dimensional case, $u$, $v$ are the canonical basis vectors), but the
utility of such approach is not clear.

A direct infinitely dimensional approach encounters fundamental issues in
probability on Banach spaces and geometry of Banach spaces. The sample
covariance is a random element in $\left[  H\right]  $, which is not a Hilbert
space but only a Banach space and not separable if $H$ is infinitely
dimensional \cite[page 23]{DaPrato-1992-SEI}. $L^{p}$ laws of large numbers do
not hold on a general Banach space; in fact, a Banach space is defined to be
Rademacher type $p$ if the Marcinkiewicz-Zygmund inequality holds
\cite{Ledoux-1991-PBS}, and then an $L^{p}$ law of large numbers follows just
like in (\ref{eq:Lp-LLN}). Conversely, certain $L^{p}$ laws of large numbers
imply that the Banach space is of type $p$ \cite{Shi-xin-2002-CTp}. See also
Proposition \ref{prop:type-p}.

Marcinkiewicz-Zygmund inequality and weak laws of large numbers are also
available in $p$-uniformly smooth Banach spaces, which are characterized by a
generalized form of the parallelogram equality,%
\[
\left\vert x+y\right\vert ^{p}+\left\vert x-y\right\vert ^{p}\leq2\left\vert
x\right\vert ^{p}+K\left\vert y\right\vert ^{p}%
\]
for some $K$ \cite{Quang-2006-WLL,Woyczynski-1978-GMB}.

For weak convergence of sample covariance in finite dimension, see also
\cite{Sepanski-2005-WLL}.

\section{White noise measures}

A theory of white noise somewhat different from Definition
\ref{def:white-noise} is obtained when the test functions are restricted to a
smaller test space $S\subset H$. However, $S$ cannot be arbitrary; it must be
a so-called nuclear space. For example, $H=L^{2}\left(  \mathbb{R}\right)  $
while $S=\mathcal{S}\left(  \mathbb{R}\right)  $, the Schwartz space, gives
white noise as random continous linear functionals in $\mathcal{S}^{\prime
}\left(  \mathbb{R}\right)  $, i.e., distributions. More generally, given a
complete orthonormal set $\left\{  y_{j}\right\}  $ in $H$ and $\lambda
_{1}\geq\lambda_{2}\geq\ldots>0$ such that $\sum_{j=1}^{\infty}\lambda
_{j}^{\theta}<\infty$ for some constant $\theta>0$, one can define%
\[
S_{p}=\left\{  x\in H:\left\vert x\right\vert _{p}^{2}<\infty\right\}
\quad\text{where }\left\vert x\right\vert _{p}^{2}=\sum_{j=1}^{\infty}%
\lambda_{j}^{-2p}\left\langle y_{j},x\right\rangle ,
\]
and define $S=%
{\textstyle\bigcap\nolimits_{p=1}^{\infty}}
S_{p}$, equipped with the projective limit topology, i.e., the topology
generated by the neighborhoods of zero $\left\{  y\in H:\left\vert
y\right\vert _{p}<\varepsilon\right\}  $, $p\in\mathbb{N},$ $\varepsilon>0$.
Such space $S$ is a \emph{nuclear space}. The Gaussian probability measure
$\mu$ on $\left(  S^{\prime},\mathcal{B}\left(  S^{\prime}\right)  \right)  $
is then given by its Fourier transform%
\[
\int_{S^{\prime}}e^{i\left\langle y,x\right\rangle }d\mu\left(  x\right)
=e^{-\frac{1}{2}\left\vert y\right\vert ^{2}},\quad y\in S,
\]
and its existence follows from the Bochner-Minlos theorem. The measure $\mu$
is called the \emph{white noise measure} and the measure space $\left(
S,\mu\right)  $ is called \emph{white noise space}. See, for example,
\cite{Kuo-1996-WND} or \cite[p. 25]{Hida-2008-LWN}.

For any test function $x\in S$, the mapping $y\mapsto\left\langle
y,x\right\rangle $ is a random variable with the distribution $N\left(
0,\left\vert x\right\vert ^{2}\right)  $, and if $x_{1}$,\ldots, $x_{n}$ are
orthonormal, the random variables $y\mapsto\left\langle y,x_{i}\right\rangle $
are independent.

\cite[p. 48]{DaPrato-1992-SEI} call a linear mapping $X$ from Hilbert space
$H$ to $L^{2}\left(  \Omega,\Sigma,\Pr\right)  $ (for some probability space
$\left(  \Omega,\Sigma,\Pr\right)  $) white noise if the values of $X$ are
Gaussian random variables and $E\left(  X\left(  z_{1}\right)  X\left(
z_{2}\right)  \right)  =\left\langle z_{1},z_{2}\right\rangle $ for all
$z_{1}$, $z_{2}\in H$. Clearly, if $W$ is as above, one has the correspondence
$\Omega=H$, $\Sigma$ are the Borel sets on $H$, $\Pr=\mu$, and, for a fixed
$z\in H$, the random variable $X\left(  z\right)  $ is defined by $X\left(
z\right)  :\omega\in H\mapsto W_{z}\left(  \omega\right)  $.

\section{Linear Algebra}


For a rectangular matrix $A\in\mathbb{C}^{m,n},$ there exist unitary matrices
$U\in\mathbb{C}^{m,m}$, $U^{\ast}U=I$, $V\in\mathbb{C}^{n,n}$, $V^{\ast}V=I$,
and logically diagonal matrix $S\in\mathbb{C}^{m,n}$ such that%
\[
A=USV^{\ast}%
\]
where $U=\left[  u_{1},\ldots,u_{m}\right]  $, $u_{k}\in\mathbb{C}^{m}$,
$V=\left[  v_{1},\ldots,v_{n}\right]  $, $v_{k}\in\mathbb{C}^{n}$, and
$S\in\mathbb{C}^{m,n}$ has only nonzero entries $s_{kk}>0$, $k=1,\ldots,p,$
$p=\operatorname{rank}\left(  A\right)  ,$ called singular values. In
particular,%
\[
\operatorname*{Range}A=\operatorname*{span}\left\{  u_{1},\ldots
,u_{p}\right\}  .
\]

Using SVD, we can show%

\begin{equation}
\operatorname*{Range}AA^{\mathrm{\ast}}=\operatorname*{Range}A \label{eq:AAT}%
\end{equation}
as follows:%

\begin{align*}
\operatorname*{Range}AA^{\ast}  &  =\operatorname*{Range}USV^{\ast}VSU^{\ast
}\\
&  =\operatorname*{Range}US^{2}U^{\ast}=\operatorname*{span}\left\{
u_{1},\ldots,u_{p}\right\} \\
&  =\operatorname*{Range}A.
\end{align*}

\end{document}